\documentclass[reqno]{amsart}

\usepackage[margin=1.5cm]{geometry}

\usepackage{xcolor}
\usepackage{hyperref}
\hypersetup{linkbordercolor={green}}
\newcommand{\draft}{}

\usepackage{ifthen}       
\usepackage{xparse}       

\usepackage{scrtime} 
\usepackage{multirow}

\newcommand{\markerO}{\fbox{\rule{0pt}{0.1ex}\textbf{\textsf{Olaf}}}}
\newcommand{\markerF}{\fbox{\rule{0pt}{0.1ex}\textbf{Fernando}}}
\newcommand{\markerJ}{\fbox{\rule{0pt}{0.1ex}\textbf{John}}}
\newcommand{\look}[1]{\markerO \textbf{*}
    \footnote{ #1 }}
\newcommand{\lookO}[1]{\markerO\textbf{*}
    \footnote{\textbf{\textsf{Olaf:}} #1 }}
\newcommand{\lookF}[1]{\markerF\textbf{*}
    \footnote{\textbf{Fernando:} #1 }}
\newcommand{\lookJ}[1]{\markerJ\textbf{*}
	\footnote{\textbf{John:} #1 }}

\newcommand{\Ignore}[1]{}
\ifthenelse{\isundefined \draft }
{
  \renewcommand{\look}[1]{}
  \renewcommand{\lookO}[1]{}%
  \renewcommand{\lookF}[1]{}%
  \renewcommand{\lookJ}[1]{}%

  \ifthenelse{\isundefined \AMSstylefile } 
  {\automark[subsection]{section}}  %
  {}
}
\usepackage{graphicx}                

\usepackage{amsmath, amssymb}
\usepackage{amsthm}

\usepackage{amscd}

\usepackage{accents}     

\usepackage{bbm}         

\usepackage{mathrsfs}    
\renewcommand\mathcal\mathscr  
\usepackage{float}

\numberwithin{equation}{section}

\newcounter{myenumi}

\newcommand{\itemref}[1]{\noindent(\ref{#1})}

  \ifthenelse{\isundefined \OlafsVersion} 
{ 
  \theoremstyle{plain} 
  \newtheorem{theorem}{Theorem}[section]

  \newtheorem{maintheorem}[theorem]{Main Theorem}
  \newtheorem{proposition}[theorem]{Proposition}
  \newtheorem{lemma}[theorem]{Lemma}
  \newtheorem{corollary}[theorem]{Corollary}
  \newtheorem{conjecture}[theorem]{Conjecture}
  
  \theoremstyle{definition}       
  \newtheorem{definition}[theorem]{Definition}
  \newtheorem{assumption}[theorem]{Assumption}
  \newtheorem{example}[theorem]{Example}
  
  \newtheorem{remark}[theorem]{Remark}
  \newtheorem*{remark*}{Remark}
  \newtheorem{notation}[theorem]{Notation}
  \newcommand{\myparagraph}[1]{\textbf{#1}}
} 
{ 

  \newcommand{\myfont}{\sffamily}

  \newcommand{\myparagraph}[1]{\noindent\textbf{\myfont{#1}}}

  \newtheoremstyle{mythmstyle}
  {\topsep}
  {\topsep}
  {\itshape}
  {}
  {\bfseries \myfont}
  {.}
  {.5em}
  {}
  
  \newtheoremstyle{mydefstyle}
  {\topsep}
  {\topsep}
  {\normalfont}
  {}
  {\bfseries \myfont}
  {.}
  {.5em}
  {}

  \swapnumbers                    

  \theoremstyle{mythmstyle}       

  \newtheorem{theorem}{Theorem}[section]
  
  \newtheorem{proposition}[theorem]{Proposition}
  \newtheorem{lemma}[theorem]{Lemma}
  \newtheorem{corollary}[theorem]{Corollary}
  
  \newcounter{intro}

  \theoremstyle{mydefstyle}        
  \newtheorem{definition}[theorem]{Definition}
  
  \newtheorem{example}[theorem]{Example}

  \newtheorem{remark}[theorem]{Remark}
  
  \newtheorem*{remark*}{Remark}
  
  \expandafter\let\expandafter\oldproof\csname\string\proof\endcsname
  \let\oldendproof\endproof
  \renewenvironment{proof}[1][\bfseries\myfont\proofname]{%
    \oldproof[\bfseries \myfont #1]%
  }{\oldendproof}
} 

\ifthenelse{\isundefined \OlafsVersion}
{} 
{ 
  \input{better-ams-style.tex}
} 

\newcommand{\Sec}[1]{Section~\ref{sec:#1}}

\newcommand{\Subsec}[1]{Subsection~\ref{subsec:#1}}

\newcommand{\Eq}[1]{Eq.~\eqref{#1}}

\newcommand{\Fig}[1]{Figure~\ref{fig:#1}}

\newcommand{\Thm}[1]{Theorem~\ref{thm:#1}}
\newcommand{\Thms}[2]{Theorems~\ref{thm:#1} and~\ref{thm:#2}}

\newcommand{\Thmenum}[2]{Theorem~\ref{thm:#1}~(\ref{#2})}

\newcommand{\Ex}[1]{Example~\ref{ex:#1}}

\newcommand{\Exenum}[2]{Example~\ref{ex:#1}~(\ref{#2})}

\newcommand{\Lem}[1]{Lemma~\ref{lem:#1}}

\newcommand{\Lemenum}[2]{Lemma~\ref{lem:#1}~(\ref{#2})}
\newcommand{\Lemenums}[3]{Lemma~\ref{lem:#1}~(\ref{#2}) and~(\ref{#3})}

\newcommand{\Cor}[1]{Corollary~\ref{cor:#1}}
\newcommand{\Cors}[2]{Corollaries~\ref{cor:#1} and~\ref{cor:#2}}

\newcommand{\Corenum}[2]{Corollary~\ref{cor:#1}~(\ref{#2})}

\newcommand{\Prp}[1]{Proposition~\ref{prp:#1}}

\newcommand{\Prpenum}[2]{Proposition~\ref{prp:#1}~(\ref{#2})}

\newcommand{\Rem}[1]{Remark~\ref{rem:#1}}

\newcommand{\Remenum}[2]{Remark~\ref{rem:#1}~(\ref{#2})}

\newcommand{\Def}[1]{Definition~\ref{def:#1}}
\newcommand{\Defs}[2]{Definitions~\ref{def:#1} and~\ref{def:#2}}

\newcommand{\Defenum}[2]{Definition~\ref{def:#1}~(\ref{#2})}

\newcommand{\DefenumS}[3]{Definition~\ref{def:#1}~(\ref{#2})--(\ref{#3})}

\newcommand{\abs}[2][{}]{\lvert{#2}\rvert_{{#1}}}    
\newcommand{\abssqr}[2][{}]{\lvert{#2}\rvert^2_{#1}} 
\newcommand{\bigabs}[2][{}]{\bigl\lvert{#2}\bigr\rvert_{#1}}     

\newcommand{\normsymb}{\|}

\newcommand{\norm}[2][{}]{\normsymb{#2}\normsymb_{{#1}}}    
\newcommand{\normsqr}[2][{}]{\normsymb{#2}\normsymb^2_{#1}} 

\newcommand{\iprod}[3][{}]{\langle{#2},{#3}\rangle_{#1}}  

\newcommand{\set}[2]{\{ \, #1 \, | \, #2 \, \} }      
\newcommand{\bigset}[2]{\bigl\{ \, #1 \, \bigl|\bigr. \, #2 \, \bigr\} }
\newcommand{\Bigset}[2]{\Bigl\{ \, #1 \, \Bigl|\Bigr. \, #2 \, \Bigr\} }

\newcommand{\map}[3]{ #1 \colon #2 \longrightarrow #3}    

\newcommand{\compl}[1]{#1^{\mathsf c}}

\newcommand{\bd}  {\partial}          

\newcommand{\restr}[1]{{\restriction}_{#1}} 

\newcommand{\card}[1]{\lvert#1\rvert}   

\DeclareMathOperator{\id}     {id}   

\DeclareMathOperator{\tr}     {tr}  

\newcommand{\de} {\mathord{\mathrm d}} 

\renewcommand{\phi}{\varphi}   

\newcommand{\R}{\mathbb{R}} 
\newcommand{\C}{\mathbb{C}} 
\newcommand{\N}{\mathbb{N}} 
\newcommand{\Z}{\mathbb{Z}} 

\newcommand{\1}{\mathbbm 1}                    

\newcommand{\e}{\mathrm e}  
\newcommand{\im}{\mathrm i} 

\newcommand{\G}{{G}}
\newcommand{\W}{\mathbf{G}}

\newcommand{\Wfull}{(\G,\alpha,\m)} 
\newcommand{\Wtwofull}{(\G',\alpha',\m')} 
\newcommand{\Wfulldeg}{(\G,\alpha,\deg)}
\newcommand{\Wfullcomb}{(\G,\alpha,\com)}

\newcommand{\m}{w}

\newcommand{\wt}{\widetilde}           

\newcommand{\HS}{\mathcal H}           

\newcommand{\lsymb}    {\ell}          

\newcommand{\lpspace}[1][p]    {\lsymb_{#1}}     
\newcommand{\lsqrspace}    {\lpspace[2]}          

\newcommand{\lp}[2][p]{\lpspace [#1]({#2})} 

\newcommand{\lsqr}[2][{}]{\lsqrspace^{#1}({#2})}   

\newcommand{\laplacian}[2][{}]{\Delta_{{#2}}^{{#1}}} 

\newcommand{\quadtext}[1]{\quad\text{#1}\quad}
\newcommand{\qquadtext}[1]{\qquad\text{#1}\qquad}

\DeclareMathAlphabet{\Ma}{U}{msa}{m}{n}
\DeclareMathAlphabet{\Mb}{U}{msb}{m}{n}
\DeclareMathAlphabet{\Meuf}{U}{euf}{m}{n}
\DeclareSymbolFont{ASMa}{U}{msa}{m}{n}
\DeclareSymbolFont{ASMb}{U}{msb}{m}{n}

\newcommand{\lessWithNumber}[1]{\stackrel{#1}\preccurlyeq}

\NewDocumentCommand{\less}{o}{%
  \IfNoValueTF{#1}
    {\preccurlyeq}
    {\lessWithNumber{#1}}%
}

\newcommand{\com}{\mathbbm{1}}
\newcommand{\Ga}{\mathscr{G}}
\newcommand{\Co}{\mathscr{G}_{\com}}
\newcommand{\De}{\mathscr{G}_{\deg}}
\newcommand{\lesse}{\sqsubseteq}
\usepackage{tikz}

\usepackage{verbatim} 
\usepackage{subfig}
\usepackage{cleveref}

\newcommand{\GAM}{MW}    
\newcommand{\aGAM}{an MW}   
\newcommand{\AGAM}{An MW}   
\newcommand{\GM}{M}  
\newcommand{\aGM}{an M}

\newcommand{\RmodZ}{\R/2\pi\Z}

\newcommand{\vx}[2][]{#2 \restr{V#1}}
\newcommand{\ed}[2][]{#2 \restr{E#1}}

\title{Spectral preorder and perturbations of discrete weighted graphs}%

\author{John Stewart Fabila-Carrasco} %
\address{Department of Mathematics, University Carlos III de Madrid,
  Avda. de la Universidad 30, 28911. Legan\'es (Madrid), Spain}
\email{jfabila@math.uc3m.es}

\author{Fernando Lled\'o} %
\address{Department of Mathematics, University Carlos III de Madrid,
  Avda. de la Universidad 30, 28911. Legan\'es (Madrid), Spain and
  Instituto de Ciencias Matem\'aticas (CSIC-UAM-UC3M-UCM), Madrid}
\email{flledo@math.uc3m.es}

\author{Olaf Post} %
\address{Fachbereich 4 -- Mathematik, Universit\"at Trier, 54286
  Trier, Germany} \email{olaf.post@uni-trier.de}

\thanks{JSFC was supported by Spanish Ministry of Economy and
Competitiveness through project DGI MTM2017-84098-P}

\thanks{FLl was supported by Spanish Ministry of Economy and
Competitiveness through project DGI MTM2017-84098-P and the
\emph{Severo Ochoa} Program for Centers of Excellence in R\&D
(SEV-2015-0554).}


\begin{document}


\ifthenelse{\isundefined \draft}
{\date{\today}}  
{\date{\today, \thistime,  \emph{File:} \texttt{\jobname.tex}}} 

\begin{abstract}
  In this article, we introduce a geometric and a spectral preorder
  relation on the class of weighted graphs with a magnetic
  potential. The first preorder is expressed through the existence of
  a graph homomorphism respecting the magnetic potential and fulfilling
  certain inequalities for the weights.  The second preorder refers to
  the spectrum of the associated Laplacian of the magnetic weighted
  graph. These relations give a quantitative control of the effect of
  elementary and composite perturbations of the graph (deleting edges,
  contracting vertices, etc.) on the spectrum of the corresponding
  Laplacians, generalising interlacing of eigenvalues.

  We give several applications of the preorders: we show how to
  classify graphs according to these preorders and we prove the
  stability of certain eigenvalues in graphs with a maximal
  $d$-clique.  Moreover, we show the monotonicity of the eigenvalues
  when passing to spanning subgraphs and the monotonicity of magnetic
  Cheeger constants with respect to the geometric preorder.  Finally,
  we prove a refined procedure to detect spectral gaps in the spectrum
  of an infinite covering graph.
\end{abstract}

\maketitle

\begin{center}
    \emph{To Hagen Neidhardt \emph{in memoriam}.}
  \end{center}


%
\section{Introduction}
\label{sec:intro}
%

Analysis on graphs is an active area of research that combines several
fields in mathematics including combinatorics, analysis, geometry or
topology.  Problems in this field range from discrete version of
results in differential geometry to the study of several combinatorial
aspects of the graph in terms of spectral properties of operators on
graphs (typically discrete versions of continuous Laplacians), see
e.g.~\cite{mohar:91,cds:95,chung:97,colin:98,hogben:05,sunada.in:08,sunada:12,brouwer-haemers:12}.  
The interplay between discrete and continuous structures are very apparent for the class of metric graphs
together with their natural Laplacians (see e.g.~\cite{ekkst:08,post:12} and references therein).

The spectrum of a finite graph mostly refers to the spectrum of the
adjacency matrix $A$ (e.g.\ in Cvetkovi\'c, Doob and Sachs
book~\cite{cds:95} or in Brouwer and Haemers'
book~\cite{brouwer-haemers:12}, while the latter book also contains many
results on the Laplacian $L=D-A$ and its signless version
$Q=D+A$.  Here, $D$ is the matrix with the degrees of the (numbered)
vertices on its diagonal.  In Chung's book~\cite[Section~1.2]{chung:97}
the spectrum of a graph refers to the spectrum of its \emph{standard}
Laplacian $\mathcal L=D^{-1/2}LD^{-1/2}=I-D^{-1/2}AD^{-1/2}$, where
$I$ is the identity matrix of order $\card G$ (the standard Laplacian
is sometimes also called \emph{normalised}, e.g.\ in \cite{chung:97},
or sometimes also \emph{geometric}).  Colin de
Verdi\`ere~\cite{colin:98} considers wider classes of discrete
operators, namely discrete weighted Laplacians with electric (but
without magnetic) potential.  A survey considering all the
above-mentioned matrices associated with a graph can be found
in~\cite{hogben:05}.  Note that the spectra of the combinatorial,
standard Laplacian and the adjacency operator are only related if the
underlying graph is regular (i.e.\ all vertices have the same degree).

In this article, we consider general weights on the edges and
vertices, in order to include the combinatorial and standard
Laplacians at the same time.  Moreover, we allow magnetic potentials,
which can be considered also as complex-valued edge weights (of
absolute value $1$).  Magnetic Laplacians or Schr\"odinger operators
on graphs have also attracted much interest (see,
e.g.~\cite{sunada:94c,
  higuchi-shirai:01,llpp:15,korotyaev-saburova:17,bgklm:20}); they are
defined via a phase $\e^{\im \alpha_e}$ for each oriented edge $e$ in
the discrete Laplacian; $\alpha_e$ is called the \emph{magnetic
  potential}. The concept of \emph{balanced} or \emph{signed} graphs
is related (as pointed out only recently in~\cite{llpp:15}, see also
the detailed reference list therein), and it can be seen as a special
case of a magnetic Laplacian with magnetic phases $1=\e^{0}$ and
$-1=\e^{\im \pi}$ only.  A prominent example of a magnetic Laplacian
already treated in some spectral graph theory articles or books
(e.g.~\cite{brouwer-haemers:12}) is the \emph{signless}
(combinatorial) Laplacian $Q=D+A$ mentioned above; it can be seen as a
magnetic combinatorial Laplacian with phase $-1=\e^{\im \pi}$ (i.e.\
vector potential $\alpha_e=\pi$ on all edges).

We will base our analysis in a rather general setting.  In particular,
we allow \emph{multigraphs} $G$ (i.e., graphs with multiple edges and
loops) which we simply call \emph{graphs} here.  Moreover, we allow
arbitrary weights on vertices and edges (denoted by the same symbol
$\m$) in order to cover the combinatorial and the standard Laplacian
(and all other weighted versions).  Finally, we allow a discrete
vector potential $\alpha$ describing a magnetic flux on each cycle of
the graph; in particular, our analysis allows to include also signed
graphs or signless versions of the Laplacian.  We call such graphs
\emph{magnetic weighted graphs} (or \GAM-graphs for short) and denote
the class by $\Ga$. The graphs in this class may have finite or
infinite order. A generic element in this class is written as
$\W=(\G,w,\alpha)$. If we restrict to \GAM-graphs with combinatorial
or standard weights, we use the symbols $\Co$ and $\De$, respectively.

In this article, we present two preorders on the class of \GAM-graphs:
the first one denoted by $\W \lesse \W'$ is geometric in nature and
basically assumes that there is a graph homomorphism from $\W$ to
$\W'$ respecting the magnetic potential and fulfilling certain
inequalities on the weights, called \emph{magnetic graph
  homomorphisms} (\GAM-homomorphisms for short, see
\Def{hom.gam-graph} for details).  The existence of an
\GAM-homomorphism is rather restrictive, e.g.\ for standard weights
(degree on the vertices, and $1$ on the edges), an $\GAM$-homomorphism
is a quotient map (cf.\ \Prp{gam-homo.std}).

The inequalities on the weights are made in such a way that
\begin{equation*}
  \W\lesse\W' \quad\Rightarrow \quad \lambda_k(\W)\leq \lambda_k(\W')
\end{equation*}
hold for all $k$ (assuming that the number of vertices fulfils
$\card{V(\W)} \ge \card{V(\W')}$).  Here we write the spectrum of the
magnetic weighted Laplacian in increasing order and counting
multiplicities.  This monotonicity is our first main result, see
\Thm{homomorphism}.  In particular, the inequalities on the weights
imply a similar inequality on the Rayleigh quotients.  We state the
above eigenvalue inequality as $\W \less \W'$, our second preorder on
the set of (finite) \GAM-graphs $\Ga$.

Similarly, the weight inequalities characterising MW-homomorphisms 
are compatible with a certain isoperimetric ratio. 
In fact, given $\W\in\Ga$ denote by $h_k(\W)$ the
$k$-th (magnetic weighted) Cheeger constant where we incorporate into
the analysis the magnetic field via the frustration index of the graph
(see \Subsec{cheeger} and~\cite{llpp:15}).  Then, for any
$\W,\W'\in\Ga$ we show in \Thm{cheeger-and-morphisms} the implication
\begin{equation*}
  \W\lesse\W' \quad\Rightarrow \quad h_k(\W)\leq h_k(\W')
\end{equation*}
for all $k$.

The relation $\less$ can be extended by a shift $r\in\N_0$ in the list
of eigenvalues in which case we use the symbol $\less[r]$ (cf.\
\Def{with-shift}). From the point of view of linear algebra, the
spectral preorder is a very flexible generalisation of eigenvalue
interlacing known for matrices (see
e.g.~\cite[Theorem~4.3.28]{horn-johnson:13}).  Interlacing applied to
graphs is also treated in~\cite[Sections~2.5, and
3.2]{brouwer-haemers:12}.  Some of our elementary operations on graphs
can hence be also seen as a geometric interpretation of eigenvalue
interlacing.  In particular, we have already mentioned above that the
geometric preorder is stronger than the spectral preorder (cf.\
\Thm{homomorphism}), i.e., if $\W,\W'\in\Ga$ then $\W\lesse\W'$
implies $\W\less\W'$.

We can also compare in a natural way the same graphs with different
weights.  In particular, in \Cor{std-com} we show that the $k$-th
eigenvalue of the standard magnetic Laplacian is always bounded above
by the $k$-th eigenvalue of the combinatorial magnetic Laplacian for
every possible vector potential $\alpha$. 

In \Sec{geo} we use the preorders $\lesse$ and $\less$ (with
appropriate shifts) to give a quantitative estimate of the spectral
effect that elementary perturbations have on the spectrum of the
corresponding Laplacians (see \Thms{delete-edge}{edge-contra} in the
case of general weights).  We also analyse in \Subsec{composite}
composite perturbations like edge contraction or vertex deletion.  In
the special cases of combinatorial and standard weights, we have the
following situations (cf., \Cors{delete-edge}{vert-contra}).
\begin{itemize}
\item \emph{Edge deletion}: Let $e_0$ be an edge and $\W,\W'\in\Ga$,
  where $\W'=\W-e_0$ (i.e., $e_0$ has been removed from $\W$).
  \begin{itemize}
  \item If $\W,\W'\in\Co$, then $\W \less[1] \W'$ and $\W' \lesse \W$,
    hence $\W \less[1] \W' \less \W$.  
  \item If $\W,\W'\in\De$, then $\W \less[1] \W' \less[1] \W$.
\end{itemize}

\item \emph{Vertex contraction}: Let $v_1,v_2$ be vertices and
  $\W, \wt \W\in\Ga$ with $\wt\W =\W/ \{v_1,v_2\}$ (i.e., the vertices have
  been \emph{identified} in $\wt \W$ keeping all the edges, i.e.\
  loops or multiple edges may occur).
  \begin{itemize}
  \item If $\W,\wt \W\in\Co$, then $\W\lesse \wt \W$ and $\wt \W\less[r+1] \W$,
    hence $\W \less \wt \W \less[r+1] \W$, where
    $r=\min\{\deg^\G(v_1),\deg^\G(v_2)\}$.
  \item
    If $\W,\wt \W\in\De$, then $\W\lesse \wt \W$ and $\wt \W \less[1] \W$,
    hence $\W\less \wt \W \less[1] \W$.
  \end{itemize}
\end{itemize}
These results are sharp in the sense that, in general, one cannot
lower the value of the spectral shift.  Let us comment on related
results in the literature: Van den Heuvel~\cite[Lemma~2]{heuvel:95}
proves the result on edge deletion for the combinatorial Laplacian and
its signless version, see also~\cite[Theorem~3.2]{mohar:91}
and~\cite[Corollary~3.2]{fiedler:73}; the result is also used to
spectrally exclude the existence of a Hamiltonian cycle e.g.\ in the
Peterson graph (see~\cite[Theorem~3.3]{mohar:92})
and~\cite[Theorem~1]{heuvel:95}).

In~\cite[Theorem~2.3]{chen:04}, the authors consider the specific case of
the standard Laplacian and edge deletion; this result was generalised
to signed graphs in~\cite[Theorem~8]{atay-tuncel:14}.
Similarly,~\cite[Theorem~2.7]{chen:04} (and again generalised to the case
of signed graphs in~\cite[Theorem~10]{atay-tuncel:14}) prove a weaker
version of our vertex contraction for the standard Laplacian, namely
$\W \less[1] \wt \W \less[1] \W$ in our notation, under the additional
assumption that the vertices $v_1,v_2$ have combinatorial distance at
least $3$.  The latter restriction is mainly due to the fact that both
papers avoid the use of multigraphs, namely multiple edges and loops.

There are related results for so-called \emph{quantum graphs}
in~\cite{bkkm:19} (for the notion of quantum graphs, see the
references therein or e.g.~\cite{post:12}): Let $\mathbf M$ be a
compact metric graph, and let $\wt{\mathbf M}$ be the metric graph
obtained from $\mathbf M$ by contracting two vertices.  The spectrum
of $\mathbf M$ is the ordered list of eigenvalues of its standard
(also called Kirchhoff) Laplacian (repeated with respect to their
multiplicity).  Then Theorem~3.4 of~\cite{bkkm:19} states
$\mathbf M \less \wt{\mathbf M} \less[1] \mathbf M$ in the sense that
$\lambda_k(\mathbf M) \le \lambda_k(\wt{\mathbf M}) \le
\lambda_{k+1}(\mathbf M)$
for all $k$ (see also the references therein).  This result shows
again that the standard (also called ``geometric'') Laplacian is
closer to the continuous case than the combinatorial Laplacian on a
graph.

We present a wide range of applications of the preorders and relations
studied before: one can use the preorders to give a geometrical and
spectral ordering of graphs (see \Subsec{spec.ord}).  We also show
that the eigenvalues of a magnetic weighted Laplacian of a spanning
subgraph and the original graph are monotonous (i.e., the spectral
ordering holds, see \Cor{spanning.subgraph}).  In addition, we show
the stability of certain eigenvalues in graphs with a maximal
$d$-clique (see \Thm{cliques} and \Cor{cliques}).  Moreover, the
spectral preorder can also be used to show that high multiplicity
eigenvalues remain eigenvalues after ''small`` perturbations and
taking certain minors (see \Subsec{minors}).  We also prove the above
mentioned monotonicity of Cheeger's constant with respect to the
geometric preorder (see \Thm{cheeger-and-morphisms}).

Finally, the spectral preorder can be used also to refine the
bracketing technique (known for continuous spaces under the name
``Dirichlet-Neumann-bracketing'') for discrete graphs.  We can apply
the results on vertex contraction at the level of the finite
fundamental domain to detect in some examples new spectral gaps and to
almost determine completely the spectrum of the discrete Laplacian on
the covering space (see \Subsec{covering}).  Let us conclude
mentioning that spectral gaps of Schr\"odinger operators play an
important role in spectral analysis and mathematical physics
(see~\cite{higuchi-nomura:09,flp:18,
  korotyaev-saburova:15,korotyaev-saburova:18} and references
therein.)

 \subsection*{Structure of the article}
 In \Sec{graphs.cohomg} we introduce the main discrete structures
 needed in this article. In particular the class $\Ga$ of magnetic
 weighted graphs (\GAM-graphs for short) and the subclasses of
 \GAM-graphs with combinatorial and standard weights denoted by $\Co$
 and $\De$, respectively. We introduce in \Def{relation1} the
 geometric preorder $\lesse$ on $\Ga$ which is based on the notion of
 a magnetic weighted graph homomorphism.  We show that it is a partial
 order on the class of finite \GAM-graphs with combinatorial or
 standard weights.  In \Sec{mag.lap.spec.ord} we introduce the
 discrete magnetic Laplacian $\Delta_\alpha=d_\alpha^* d_\alpha$,
 where $d_\alpha$ is a discrete exterior derivative twisted by the
 magnetic potential $\alpha$. We also introduce in
 \Def{spectral-order} the spectral preorder $\less$ on $\Ga$ and
 consider also the possibility to compare shifted lists of eigenvalues
 of the corresponding Laplacians.  In \Sec{geo} we use the preorders
 $\lesse$ and $\less$ to give a quantitative estimate of the spectral
 effect that elementary perturbations have on the spectrum of the
 corresponding Laplacians (see \Thms{delete-edge}{edge-contra} in the
 case of general weights). We also analyse in \Subsec{composite}
 composite perturbations like edge contraction or vertex deletion.  In
 the final section we present our applications on spectral ordering of
 combinatorial graphs, graph minors, cliques, multiple eigenvalues,
 magnetic Cheeger constants and existence of spectral gaps on covering
 graphs.

%
%
\section{Magnetic weighted graphs and their homomorphisms}
\label{sec:graphs.cohomg}
%

In this section, we introduce the discrete structures needed and
mention some basic properties and examples.  We will consider discrete
locally finite graphs with arbitrary weights on vertices and edges as
well as an $R$-valued function on the edges which correspond to a
discrete analogue of the magnetic potential.  Here, $R$ is a subgroup
of the Abelian group $\RmodZ$ written additively. In this section,
graphs may be finite or infinite, and we will not assume that the graphs are
necessarily connected.

\subsection{Discrete graphs}
\label{sec:disc.graphs}

A \emph{discrete graph} (or, simply, a \emph{graph}) $G=(V,E,\bd)$
consists of two disjoint (and at most countable) sets $V=V(G)$ and
$E=E(G)$, the set of \emph{vertices} and \emph{edges}, respectively,
and a \emph{connection map} $\map {\bd=\bd^G} E {V \times V}$, where
$\bd e = (\bd_-e,\bd_+e)$ denotes the pair of the \emph{initial} and
\emph{terminal} vertex, respectively.  We also say that $e$
\emph{starts} at $\bd_-e$ and \emph{ends} at $\bd_+e$.  We assume that
each edge $e$ (also called \emph{arrow}) comes with its oppositely
oriented edge $\bar e$, i.e.  that there is an involution
$\map {\bar \cdot}EE$ such that $e \ne \bar e$ and
$\bd_\pm \bar e = \bd_\mp e$ for all $e \in E$.\footnote{Note that in
  this article we switched to the more standard notation that an edge
  $e$ (also called an \emph{arrow}) always has its oppositely
  oriented counterpart $\bar e$ in $E$; in our older papers
  (e.g.~in~\cite{flp:18}) we used the convention that $E$ contains
  only one arrow (not its inverted arrow), hence in our older notation
  $E$ together with $\bd e=(\bd_+e,\bd_-e)$ determines already an
  orientation of the graph.}  We allow \emph{multiple edges} (i.e.\
$\bd$ is not necessarily injective, hence edges cannot be represented
as pairs $(v_1,v_2)$ of vertices in general) and also \emph{loops}
(i.e.\ edges $e$ with $\bd_-e=\bd_+e$).  Note that also loops $e$ come
in pairs $e \ne \bar e$.  If $V(\G)$ has infinitely many vertices, we
say that the graph $\G$ is \emph{infinite}.  If $V(\G)$ has $n \in \N$
vertices, we say that $\G$ is a \emph{finite} graph of \emph{order}
$n$ and we write $\card G=\card{V(\G)}=n$.

A \emph{path} $p=(e_1,\dots,e_r)$ \emph{of length $r$} in a graph $G$
is a finite sequence of $r$ edges $e_1,\dots,e_r \in E$ such that
$\bd_+e_{k-1}=\bd_-e_k$ for all $k=1,\dots,r$.  We say that $p$
\emph{joins} the vertices $\bd_-e_1$ and $\bd_+e_r$.  The
\emph{combinatorial distance} of two vertices is the length of the
shortest path joining these two vertices.  A graph is called
\emph{connected} if for any vertices $x,y \in V$ there is a path $p$
joining $x$ and $y$.  For two subsets $V_\pm \subset V$ we denote by
\begin{equation*}
  E(V_-,V_+):=\set{e \in E}{\bd_-e \in V_-, \bd_+ e \in V_+}
\end{equation*}
the set of all edges starting in $V_-$ and ending in $V_+$.  Note that
$e \in E(V_-,V_+)$ if and only if $\bar e \in E(V_+,V_-)$.  If we need
to stress the graph $G$ to which $E(V_-,V_+)$ refers, we write
$E^G(V_-,V_+)$.  As a shortcut, we also set $E(V_0):=E(V_0,V_0)$,
$E(v,V_0):=E(\{v\},V_0)$ and $E(v,x)=E(\{v\},\{x\})$ etc.\ for
$v,x \in V$ and $V_0 \subset V$.  Moreover, we denote by
\begin{equation*}
  E_v := E(v,V) = \set{e \in  E}{\bd_- e=v}
\end{equation*}
the set of all edges starting at $v$ (alternatively we may also write $E_v^\G$).
We define the \emph{degree} of the vertex $v$ in the graph $G=(V,E,\bd)$ by
\begin{equation*}
  \deg(v) := \deg^G(v) = \card{E_v}.
\end{equation*}
Note that a loop at a vertex $v$ increases the degree by $2$.  We assume
that the graph is \emph{locally finite}, i.e.\ $\deg(v) < \infty$
for all $v \in V$.
We call a graph \emph{simple} if it has no loops and no multiple
edges, i.e.\ if $E(v,v)=\emptyset$ and $\abs{E(v,x)}\leq 1$ for all
$v,x \in V$, $v \ne x$.  For a simple graph, the connection map $\bd$
is injective, hence an edge $e$ can be identified with its pair
$(\bd_-e, \bd_+e)$ of its initial and terminal vertex.

We first consider two elementary operations on a graph,
\emph{contracting} or \emph{identifying vertices} while keeping the
edges and \emph{deleting edges} while keeping the vertices:
\begin{definition}[Contracting and splitting vertices]
  \label{def:glueing-vertices}
  Let $\G=(V,E,\bd)$ be a graph and let ${\sim}$ be an equivalence
  relation on $V$. 
  \begin{enumerate}
  \item The \emph{quotient graph} $\wt\G=\G/{\sim}$ is defined by
    $\wt\G=(\wt V,\wt E,\wt \bd)$, where $\wt V=V/{\sim}$, $\wt E=E$ and
    $\wt \bd e = ([\bd_-e],[\bd_+e])$ for all $e\in E$.  We also say
    that $\wt\G$ is obtained from $\G$ by \emph{contracting} or
    \emph{contracting vertices} according to ${\sim}$.

  \item If the relation ${\sim}$ identifies only the vertices
    $v_1,\dots, v_r \in V(\G)$ to one vertex in $\wt\G$, we also say
    that $\wt\G$ is obtained from $\G$ by \emph{contracting} or
    \emph{identifying} the vertices $v_1, \dots, v_r\in V$.  We write
    $\wt\G=\G/\{v_1,\dots, v_r\}$ for short (see \Fig{vercon} for the
    case $r=2$).
	
  \item The reverse operation is called \emph{splitting}: we say that
    $\G$ is a \emph{vertex splitting} of $\wt \G$ if there is an
    equivalence relation ${\sim}$ on $V(\G)$ such that
    $\wt \G=\G/{\sim}$.
  \end{enumerate}
\end{definition}
\begin{remark}[Loops and multiple edges after vertex contraction]
  \label{rem:vx-contr-convention}
  Let us stress that in contrast to many combinatorial graph theory
  books we use a ``topological'' contraction of vertices as
  in~\cite[Section~2.3]{bondy-murty:08}: if we contract two \emph{adjacent}
  vertices $v_1$ and $v_2$, all edges joining $v_1$ and $v_2$ become
  \emph{loops} in $\G/\{v_1,v_2\}$.  Moreover, contracting two
  vertices of combinatorial distance $2$, leads to a double (or
  multiple) edge.
\end{remark}
\begin{definition}[Deleting and adding edges]
  \label{def:deleting edges}
  Let $\G=(V,E,\bd)$ be a graph and let $E_0 \subset E$.
  \begin{enumerate}
  \item We denote by $\G-E_0$ the graph given by
    $(V, E \setminus E_0, \bd \restr{E \setminus E_0})$.  We call
    $\G-E_0$ the graph obtained from $\G$ by \emph{deleting the edges}
    $E_0$.  If $E_0=\{e_0\}$ we simply write $\G-e_0$ instead of
    $\G - \{e_0\}$ (see, e.g.\ Figures~\ref{subfig:a}
    and~\ref{subfig:b}).

  \item The reverse operation is called \emph{adding edges}: We say
    that $\G$ is obtained from a graph $\G'$ by \emph{adding the
      edges} $E_0 \subset \G'$ if $\G'=\G-E_0$; for short we write
    $\G=\G'+E_0$ and also $\G=\G'+e_0$ if $E_0=\{e_0\}$.
  \end{enumerate}
\end{definition}

The operation of contracting two adjacent vertices and deleting the
edges joining them is called \emph{edge contraction}, and it is a
combination of contracting $v_1$ and $v_2$ and deleting the adjacent
edges $E(v_1,v_2)$.  Note that the order of the operations does not matter: first delete
$E(v_1,v_2)$ and then contracting $v_1,v_2$ or first contracting
$v_1,v_2$ and then delete the loops obtained from $E(v_1,v_2)$ gives
the same graph.

\begin{definition}[Isolated and pendant vertices, pendant and bridge edges]
  \label{def:pendant}
  Let $\G=(V,E,\bd)$ be a discrete graph.
  \begin{enumerate} 
  \item A vertex $v_0 \in V$ is called \emph{isolated} if $\deg(v_0)=0$.
  \item A vertex $v_0 \in V$ is called \emph{pendant} if $\deg(v_0)=1$.
  \item An edge $e_0 \in E$ is \emph{pendant} if at least one of its
    vertices $\bd_\pm e_0$ is a pendant vertex.
  \item An edge $e_0 \in E$ is a \emph{bridge (edge)} if $G-e_0$ has
    one more connected component than $G$.
  \end{enumerate}
\end{definition}

Another way of producing graphs from given ones are \emph{(induced)
  subgraphs}:
\begin{definition}[(Induced) subgraphs]
  \label{def:induced-subgraph}
  Let $\G=(V,E,\bd)$ be a graph.  A \emph{subgraph}
  $\G_0=(V_0,E_0,\bd \restr {E_0})$ of $\G$ is given by subsets
  $V_0 \subset V$ and $E_0 \subset E(V_0):=E(V_0,V_0)$.  An
  \emph{induced subgraph} is a subgraph such that $E_0=E(V_0)$.  The
  latter graph is also called \emph{subgraph induced by $V_0$} and is denoted
  by $\G[V_0]$.
\end{definition}
Note that $\bd(E(V_0))$ indeed maps into $V_0 \times V_0$: We have
$e \in E(V_0):=E(V_0,V_0)$ if and only if $\bd_\pm e \in V_0$, hence
$\bd e \in V_0 \times V_0$.

We introduce next another standard notation from graph theory:
\begin{definition}[Graph homomorphisms]
  \label{def:graph-homo}
  Let $\G=(V,E,\bd)$ and $\G'=(V',E',\bd')$ be two graphs.
  We say that $\map \pi {\G} {\G'}$ is a \emph{graph homomorphism}, if
  $\pi$ is a map on the vertices $\map {\vx \pi} V {V'}$ and on the
  edges $\map {\ed \pi} E {E'}$ (denoted by the same symbol $\pi$)
  such that
  \begin{equation}
    \label{eq:graph-homo-prop}
    \pi(\bd_+ e)=\bd_+'(\pi e)
    \qquadtext{and}
    \pi(\bd_- e)=\bd_-'(\pi e)
  \end{equation}
  for all $e \in E$.  If $\vx \pi$ and $\ed \pi$ are both bijective
  then $\pi$ is called an \emph{isomorphism.}  If there exists an
  isomorphism between $\G$ and $\G'$, then the graphs are called
  \emph{(graph-)isomorphic}, for short $\G\simeq\G'$.
\end{definition}

\begin{example}
  \label{ex:graph-homo}
  Let $\G=(V,E,\bd)$ and $\G'=(V',E',\bd')$ be two graphs,
  then some basic examples of homomorphisms are given as follows:
  \begin{enumerate}
  \item
    \label{graph-homo:2}
    Let ${\sim}$ be an equivalence relation on $V$, then the quotient
    map $\map {\kappa:=\kappa_{\sim}} \G {\G/{\sim}}$ given by
    $\kappa(v)=[v]$ and $\kappa(e)=e$ is a graph homomorphism.
	
  \item
    \label{graph-homo:3}
    Let $E_0 \subset E$ then the inclusion
    $\map {\iota:=\iota_{E_0}} {\G-E_0} \G $ is a graph homomorphism.
  \item Let $\map \pi \G {\G'}$ be a graph homomorphism.  Then the
    \emph{image graph} is defined by
    $\pi(\G):=(\pi(V),\pi(E),\bd'\restr{\pi(E)})$.  Note that this
    indeed defines a graph as
    $\bd'(\pi(E)) \subset \pi(V)\times \pi(V)$
    by~\eqref{eq:graph-homo-prop}, and $\pi(\G)$ is a subgraph of
    $\G'$.  Moreover, the inclusion $\map \iota {\pi(\G)} {\G'}$ is a
    graph homomorphism, injective on the vertex and edge set.
  \end{enumerate}	
\end{example}

The next lemma is an immediate consequence of Eq.~\eqref{eq:graph-homo-prop}:
\begin{lemma}
  \label{lem:graph-homo-edgeset}
  Let $\map \pi \G {\G'}$ be a graph homomorphism, then
  $\pi^{-1}(E^{\G'}(V_-',V_+'))=E^\G(\pi^{-1}(V_-'),\pi^{-1}(V_+'))$
  for $V_\pm' \subset V'$.
\end{lemma}

We finish this subsection of graph theory with the following
observation that will be useful in the following.
\begin{lemma}
  \label{lem:graph-aux}
  Let $\G=(V,E,\bd)$ and $\G'=(V',E',\bd')$ be two graphs and
  let $\map{\pi}{\G}{\G'}$ be a graph homomorphism.
  \begin{enumerate}
  \item 
    \label{graph-aux.a}
    If $\pi$ is injective on the edges, then
    \begin{equation*}
      \sum_{v\in V,\pi(v)=v'}\deg^\G(v)\leq\deg^{\G'}(v')
      \quad\text{for every $v'\in V'$.} 
    \end{equation*}
  \item 
    \label{graph-aux.b}
    If $\pi$ is injective, but not surjective on the edges, then there
    is a vertex $v_0' \in V'$ such that
    \begin{equation}
      \sum_{v\in V,\pi(v)=v_0'}\deg^\G(v)<\deg^{\G'}(v_0').
    \end{equation}
  \end{enumerate}
\end{lemma}
Note that the sum over $v \in V$ with $\pi(v)=v_0'$ could be $0$,
namely, if $v_0'$ is not in the range of $\pi$.

\begin{proof}
  \itemref{graph-aux.a}~For $v'\in V'$ let
  \begin{equation*}
    A_{v'}
    :=\bigcup_{v\in V, \pi(v)=v'}E_v^\G
    = \set{e \in E}{\pi(\bd_-e) =v'}
  \end{equation*}
  be the set of edges in $\G$ that are starting at $v$ for all
  preimages $v$ of $v'$ under $\pi$.  Then the map
  $\map \pi {A_{v'}} {E_{v'}^{\G'}}$ is well-defined, and injective
  because $\pi$ is.  Moreover, note that $(E_v^\G)_{v \in V}$ is a
  disjoint family of sets (as each edge has only one initial vertex),
  hence we have
  \begin{equation*}
    \sum_{v\in V,\pi(v)=v'}\card{E_v^\G}
      \leq \card {E_{v'}^{\G'}}.
  \end{equation*}
  The desired inequality follows from $\deg^\G(v) = \card{E_v^\G}$ and
  $\deg^{\G'}(v')=\card{E_{v'}^{\G'}}$.

  \itemref{graph-aux.b}~Choose an edge $e'_0 \in E' \setminus \pi(E)$
  and $v_0'=\bd'_-e'_0$.  Then
  $\map \pi {A_{v_0'}}{E_{v_0'}(G') \setminus \{e'_0\}}$ is injective and
  well-defined as a map, hence we conclude
  \begin{equation*}
    \sum_{v\in V,\pi(v)=v_0'}\card{E_v^\G}
      \leq \card {E_{v_0'}^{\G'}}-1
      <\card {E_{v_0'}^{\G'}}
  \end{equation*}
  proving the strict inequality claimed.
\end{proof}

\subsection{Weights on graphs}
\label{sec:disc.weighted.graphs}

Given a graph $G=(V, E,\bd)$ we consider a \emph{weight} on it, i.e.\ 
a \emph{vertex} and an \emph{edge weight} $\map \m V {(0,\infty)}$ and
$\map \m E {(0,\infty)}$ 
associating to a vertex $v$ its weight $\m(v)$ and to an edge $e$
its weight $\m_e$. 
For subsets $V_0 \subset V$ and $E_0 \subset E$, we may interpret $\m$
as a \emph{discrete measure} on the corresponding sets, and we use the
natural notation
\begin{equation}
  \label{eq:weights.meas}
  \m(V_0)=\sum_{v \in V_0} \m(v)
  \qquadtext{and}
  \m(E_0)=\sum_{e \in E_0} \m_e.
\end{equation}
\begin{example}[Standard and combinatorial weights]
  Given a graph $G=(V,E,\bd)$ one can define two important intrinsic
  weights on it.  The \emph{standard} weight in $V$ given by
  $\m(v):=\deg(v)$, $v\in V$, and $\m_e:=1$, $e\in E$, and we denote
  it simply by $\deg$. Note that $w(v)>0$ implies that a weighted
  graph with standard weights has no isolated vertices, i.e.\ vertices
  of degree $0$ (nevertheless see \Rem{std.deg.zero} for the
  convention with standard weights for the associated Laplacian in
  case that the graph has isolated vertices). The
  \emph{combinatorial} weight is given by $\m(v):=1$, $v\in V$, and
  $\m_e:=1$, $e\in E$ and we denote it by $\com$.
\end{example}
An edge weight on a graph determines a so-called \emph{weighted
  degree} of a vertex defined by
\begin{equation*}
  \deg^\m(v) := \m(E_v) = \sum_{e \in E_v} \m_e.
\end{equation*}
Recall that a loop counts twice in $E_v$.  In particular, the
combinatorial degree $\deg(v)$ agrees with the weighted degree
$\deg^\m v$ iff the edge weight equals $1$ for all edges.  We call the
weight \emph{normalised} if
\begin{equation*}
  \deg^\m(v) = \m(v),
  \qquadtext{or, equivalently,}
  \sum_{e \in E_v} \m_e = \m(v),\quadtext{ for all}  v \in  V.
\end{equation*}

We define the \emph{relative weight}
$\map {\varrho:=\varrho_\m} { V} {(0,\infty)}$ of a weighted graph
$(\G,\m)$ by
\begin{equation}
  \label{eq:rho}
  \varrho_\m(v)
  :=\frac {\deg^\m(v)} {\m(v)}
  = \frac 1 {\m(v)}\sum_{e \in E_v} \m_e;
\end{equation}
We assume that the relative weight is \emph{bounded}, i.e.\ the
\emph{maximal $\m$-degree} of $(\G,\m)$ is (uniformly) bounded:
\begin{equation}
  \label{eq:rho.bdd}
  \varrho_\infty:=\sup_{v \in  V} \varrho_\m(v) < \infty.
\end{equation}
Note that for the standard weight, or, more generally, for a
normalised weight, the relative weight is just $\varrho_\m=1$.  In
particular, the relative weight for any normalised weight is
bounded.\footnote{Note that in some references the \emph{standard}
  weight is also called \emph{normalised}.}

For the combinatorial weight, the relative weight is just the usual
degree, hence the relative weight is bounded if and only if the degree
of the graph is bounded.
\subsection{Magnetic potentials}
\label{sec:mag.field.pot}

Let $G=(V,E,\bd)$ be a graph and let $R$ be a subgroup of
$\RmodZ$ which we write additively. We consider the \emph{cochain
  groups} of $R$-valued functions on vertices and edges, which we
denote by
\begin{equation*}
  C^0(G,R)
  := \bigset{\map \xi V R}
  {\text{$\xi$ map}}
  \qquadtext{and} 
  C^1(G,R) 
  := \bigset{\map \alpha E R}
  {\forall e \in E\colon \alpha_{\bar e}=-\alpha_e},
\end{equation*}
respectively. The so-called \emph{coboundary operator} is given by
\begin{equation*}
  \map \de {C^0(G,R)}{C^1(G,R)}, 
  \qquad
  (\de \xi)_e=\xi(\bd_+e)-\xi(\bd_-e).
\end{equation*}
\begin{definition}
  \label{def:mag.pot}
  Let $G=(V,E,\bd)$ be a graph and $R$ be a
  subgroup of $\RmodZ$.  \indent
  \begin{enumerate}
  \item An \emph{$R$-valued magnetic potential} $\alpha$ is an element
    of $C^1(G,R)$.

  \item We say that $\alpha, \wt \alpha \in C^1(G,R)$ are
    \emph{cohomologous} or \emph{gauge-equivalent} and denote this as
    $\wt \alpha \sim \alpha$ if $\wt \alpha-\alpha$ is exact, i.e.\ if
    there is $\xi \in C^0(G,R)$ such that $\de \xi=\wt \alpha-\alpha$,
    and $\xi$ is called the \emph{gauge}. We denote the equivalence
    class or \emph{cohomology class} by
    $[\alpha]=\set{\wt \alpha \in C^1(G,R)}{\wt \alpha \sim \alpha}$.
    We say that $\alpha$ is \emph{trivial}, if it is cohomologous to
    $0$.
  \end{enumerate}
\end{definition}

In the sequel, we will omit the Abelian group $R$ for simplicity of
notation, e.g.\ we will write $C^1(G)$ instead of $C^1(G,R)$ for the
group of magnetic potential etc.

The next result says that if the vector potential is supported on a
bridge, then it is trivial.
\begin{lemma}
  \label{lem:bridge}
  Let $\G$ be graph such that $e_0 \in E(\G)$ is a bridge edge.  If
  $\alpha$ and $\wt \alpha$ are two vector potentials having different
  values only
  on $e_0$ (i.e.\ $\alpha_e=\wt \alpha_e$ for all
  $e\in E(\G-e_0)=E(\G) \setminus \{e_0\}$), then
  $\alpha \sim \wt \alpha$.
\end{lemma}
\begin{proof}
  Denote by $C_+$ (respectively, $C_-$) the two connected components of
  $\G-e_0$ with $\bd_+e_0\in C_+$ (respectively, $\bd_-e_0\in C_-$).  Define a
  function $\map \xi {V(\G-e_0)=V(\G)} R$ by
  $\xi(v)=\alpha_{e_0}-\wt \alpha_{e_0}$ for all $v \in C_+$ and $0$
  for all $v\in C_-$.  It follows that
  $(\de \xi)_e=\alpha_e-\wt \alpha_{e}$ hence $\alpha \sim \wt\alpha$.
\end{proof}

\subsection{\GAM-graphs and geometric preorder}
\label{sec:gam-graphs-morphisms}
In the following definition we collect all relevant structure needed:
a discrete weighted graph with vector potential.

\begin{definition}[Magnetic weighted graph, \GAM-graph]
  \label{def:mag.ext.der}
  We call $\W=\Wfull$ a \emph{magnetic weighted graph}
  (\emph{\GAM-graph} for short) if $\G=(V,E,\bd)$ is a discrete
  graph, $\m$ is a weight on the graph and
  $\alpha \in C^1(G)$ is an $R$-valued magnetic potential, i.e.\ a map
  $\map \alpha E R$ such that $\alpha_{\bar e}=-\alpha_e$ for all $e\in E$,
  where $R$ is a subgroup of $\RmodZ$.
\end{definition}
Note that $R$ can be chosen a priori.  If we choose $R=\{0\}$, then
the corresponding Laplacian defined in \Sec{disc.mag.lap} is the usual
Laplacian (without magnetic potential).  If we choose $R=\{0,\pi\}$,
then the magnetic potential is also called \emph{signature}, and $\W$
is called a \emph{signed graph} (see, e.g.~\cite{llpp:15} and
references therein for details).  This setting includes the so-called
\emph{signless} Laplacian (see \Exenum{lapl.weights}{lapl.signless})
by choosing $\alpha_e=\pi$ for all $e \in E$.

\begin{definition}[Classes of \GAM-graphs]
  \label{def:notation}
  We denote by $\mathscr{G}$ the class of all \GAM-graphs. We denote
  the subclasses of \GAM-graphs with combinatorial weight simply by
  $\Co$ and with standard weights just by $\De$.  Moreover, for a
  symmetric subset $R_0$ of $R$ (not necessarily a subgroup but being
  invariant under reflections, i.e.\ if $t\in R_0$ then $-t\in R_0$)
  we write
  \begin{equation*}
    \Ga^{R_0}:=\bigset{\W=\Wfull\in\Ga}{\alpha_e \in R_0,\;e\in E}
    \qquadtext{and}
    \Ga^t := \Ga^{\{t,-t\}}
  \end{equation*}
  for the subclass of \GAM-graphs having magnetic potential with
  values in $R_0$ respectively with constant value $t\in R$.  Similarly, we
  denote by $\De^{R_0}$ resp.\ $\De^t$ and $\Co^{R_0}$ resp.\ $\Co^t$
  the \GAM-graphs with combinatorial and standard with vector
  potential with values in $R_0$ respectively with constant value $t$.
\end{definition}	

We now introduce an important notion for this article:
\begin{definition}[\GAM-homomorphism]
  \label{def:hom.gam-graph}
  Let $\W=\Wfull$ and $\W'=\Wtwofull$ be two \GAM-graphs.
  We say that
  \begin{equation*}
    \map \pi {\W} {\W'} \quad \text{is an \emph{\GM-homomorphism}}
  \end{equation*}
  if the map satisfies the following two conditions:
  \begin{enumerate}
  \item
    \label{hom.mow-graph.a}
    $\map {\pi} \G {\G'}$ is a graph homomorphism
    (\Def{graph-homo}).
  \item
    \label{hom.mow-graph.b}
    The magnetic potential is invariant:
    $\alpha = \alpha' \circ \pi$, i.e, $\alpha_e = \alpha'_{\pi e}$
    for all $e \in E(\G)$.
  \end{enumerate}
  Moreover, we say that
  \begin{equation*}
    \map{\pi}{\W}{\W'} \quad \text{is an \emph{\GAM-homomorphism}}
  \end{equation*}
  if it is an \GM-homomorphism and satisfies the following two vertex
  and edge weight inequalities:
  \begin{enumerate}
    \addtocounter{enumi}{2}
  \item
    \label{hom.mow-graph.c}
    The following vertex weight inequality holds: 
    \begin{equation*}
      (\pi_*\m)(v')
      :=\sum_{v \in V, \pi(v)=v'} \m(v)
      \ge \m'(v')
      \qquadtext{for all $v' \in V'$,}
    \end{equation*}
    i.e.\ the push-forward vertex measure $\pi_*\m$ fulfils the
    inequality $\pi_*\m \ge\m'$ pointwise.

  \item
    \label{hom.mow-graph.d}
    The following edge weight inequality holds: 
    \begin{equation*}
      (\pi_*\m)_{e'}:=
      \sum_{e \in E, \pi(e)=e'} \m_e
      \le \m'_{e'}
      \qquadtext{for all $e' \in E'$,}
    \end{equation*}
    i.e.\ the push-forward edge measure $\pi_*\m$ fulfils the
    inequality $\pi_*\m \le\m'$ pointwise.
  \item
    \label{hom.mow-graph.e}
    We say that $\pi$ is \emph{vertex} or \emph{edge measure
      preserving} if equality holds in~\itemref{hom.mow-graph.c}
    or~\itemref{hom.mow-graph.d}, respectively, i.e.\ $\pi_*\m=\m'$
    for the vertex or the edge measure.  We simply say that $\pi$
    is \emph{measure preserving} if $\pi$ is vertex \emph{and} edge
    measure preserving.

    \item
    \label{hom.mow-graph.f}
    We say that $\pi$ is an \emph{\GAM-isomorphism}, if $\pi$ is
    bijective, and if $\pi$ and $\pi^{-1}$ are both
    \GAM-homomorphisms; in other words, if $\pi$ is a graph
    isomorphism, if $\alpha=\alpha'\circ \pi$ and if
    $\m=\m'\circ \pi$.  We say that the two \GAM-graphs are
    \emph{(\GAM-)isomorphic} (denoted by $\W\simeq \W'$) if there
    exists \aGAM-isomorphism between $\W$ and $\W'$.
  \end{enumerate}
\end{definition} 

Some examples of \GM-homomorphism and \GAM-homomorphism are the following.

\begin{example}[\GAM-homomorphisms]
  \label{ex:gam-homo}
  Let $\W=\Wfull$,, $\W'=\Wtwofull$ and
  $\wt \W=(\wt G,\wt \alpha,\wt \m)$ be \GAM-graphs.
  \begin{enumerate}
  \item
    \label{ex:gam-homo1}
    If $\G=\G'$ and $\alpha=\alpha'$, the identity
    $\map{\id_\G}{\G}{\G}$ is \aGM-homomorphism.  In order that
    $\id_\G$ is \aGAM-homomorphism, the weights must fulfil
    \begin{equation*}
      \m(v) \ge \m'(v) \quadtext{for all $v \in V(\G)$}
      \text{and}\quad
      \m_e \le \m'_e \quad\text{for all $e \in E(\G)$.}
    \end{equation*}
    In particular, this is true if $\m=\deg$ and $\m'=\com$, i.e.\ 
    $\map {\id_\G} {(\G,\alpha,\deg)}{(\G,\alpha,\com)}$ is
    \aGAM-homomorphism.
  \item
    \label{ex:gam-homo2}
    If $\wt \G=\G/{\sim}$ for some equivalence relation ${\sim}$ on
    $V(\G)$, then the quotient map $\map{\kappa}{\G}{\wt\G}$ of
    \Exenum{graph-homo}{graph-homo:2} is \aGM-homomorphism of the
    graphs with magnetic potentials $(\G,\alpha)$ and
    $(\wt \G,\alpha)$.  In order that $\kappa$ is \aGAM-homomorphism,
    the weights must fulfil
    \begin{equation*}
      \sum_{v\sim v_0}\m(v) \ge \wt\m([v_0]) 
      \quadtext{for all $[v_0] \in V(\wt\G)$}
      \text{and}\quad
      \m_e \le \wt\m_e \quad\text{for all $e \in E(\wt\G)$.}
    \end{equation*}
    This condition is automatically true when both \GAM-graphs $\W$
    and $\W'$ have combinatorial or standard weights.  Moreover,
    $\kappa$ is even measure preserving for standard weights.
  \item
    \label{ex:gam-homo3} 
    If $\G'=\G-e_0$ and $\alpha'=\alpha\restr{E(\G')}$, let
    $\map{\iota}{\G'}{\G}$ be the inclusion map of
    \Exenum{graph-homo}{graph-homo:3}, then $\iota$ is
    \aGM-homomorphism.  In order that $\iota$ is \aGAM-homomorphism,
    the weights must fulfil
    \begin{equation*}
      \m'(v) \ge \m(v) \quadtext{for all $v \in V(\G)$}
      \text{and}\quad
      \m'_e \le \m_e \quad\text{for all $e \in E(\G)$.}
    \end{equation*}
    The condition on the edges is true when both \GAM-graphs $\W$ and
    $\W'$ have combinatorial or standard weights. However, the
    condition on the vertex weights is true only for the combinatorial
    case ($\iota$ is even vertex measure preserving here).  Note that
    $\iota$ is \emph{not} edge measure preserving both for the
    combinatorial and standard weight.
  \item
    \label{ex:gam-homo5}
    Let $\W=\Wfull$, $\W'=\Wtwofull$ and $\W''=(\G'',\alpha'',w'')$ be
    three \GAM-graphs.  If $\map \pi \G {\G'}$ and
    $\map \tau {\G'}{\G''}$ are \GAM-homomorphisms, then it is easy to
    see that $\map{\tau \circ \pi}\G {\G''}$ is also
    \aGAM-homomorphism.
  \end{enumerate}
\end{example}

We state below some basic consequences of \GAM-homomorphisms:
\begin{proposition}[Basic properties of \GAM-homomorphism]
  \label{prp:inj-sur}
  Let $\W=\Wfull$ and $\W'=\Wtwofull$ be two \GAM-graphs and
  $\map \pi \W {\W'}$ \aGAM-homomorphism, then:
  \begin{enumerate}
  \item 
    \label{mow-hom.a}
    The map $\map \pi {V(\G)} {V(\G')}$ is surjective.
    
  \item
    \label{mow-hom.b}
    If there exists $c>0$ such that $\m_e=\m'_{e'}=c$ for all
    $e \in E(\G)$ and $e' \in E(\G')$ (e.g.\ if $\W$ has standard or
    combinatorial weights), then the map
    $\map {\ed \pi} {E(\G)} {E(\G')}$ is injective.
  \end{enumerate}
\end{proposition}
\begin{proof}
  \itemref{mow-hom.a}~Let $v' \in {V(\G')}$, then
  $0 < \m'(v') \le \sum_{v \in \pi^{-1}(v')} \m(v)$, and this implies
  that the sum is not empty, i.e.\ there is $v \in \pi^{-1}(v')$ with
  $\pi(v)=v'$.
	
  \itemref{mow-hom.b}~If the edge weights on $E$ and on $E'$ have
  constant value $c>0$, then
  $\sum_{e \in E, \pi(e)=e'} \m_e \le \m'_{e'}$ is equivalent with the
  fact that $\set{e \in E}{\pi(e)=e'}$ has at most one element, i.e.\ 
  $\ed \pi$ is injective.
\end{proof} 

For the combinatorial and for the standard weights, the
\GAM-homomorphism are characterised by geometrical conditions in the 
following lemmas.

\begin{proposition}[Characterisations of \GAM-homomorphisms for
  standard weights]
  \label{prp:gam-homo.std}
  Consider two \GAM-graphs with standard weights $\W=\Wfulldeg$ and $\W' = (\G',\alpha',\deg)$. Then the following conditions are
  equivalent:
  \begin{enumerate}
  \item 
    \label{gam-homo.std.a}
    There exists \aGAM-homomorphism $\map \pi {\W} {\W'}$.
  \item 
    \label{gam-homo.std.b}
    There exists a measure-preserving \GAM-homomorphism $\map \pi {\W} {\W'}$.
  \item
    \label{gam-homo.std.c}
    There is an equivalence relation $\sim$ on $V(\G)$ such that
    $\G'\simeq\G/{\sim}$ with $\alpha=\alpha'$.
  \end{enumerate}
\end{proposition}
\begin{proof}
  \itemref{gam-homo.std.a}$\Rightarrow$\itemref{gam-homo.std.b}.~ By
  \Prp{inj-sur}, the map $\pi$ is injective on edges, hence we have
  equality in~\eqref{hom.mow-graph.d} of \Def{hom.gam-graph}.
  Moreover, from \Lem{graph-aux}~\eqref{graph-aux.a} we conclude also
  equality in \Defenum{hom.gam-graph}{hom.mow-graph.c}.
	
  \itemref{gam-homo.std.b}$\Rightarrow$\itemref{gam-homo.std.c}.~Let
  $\map \pi \W {\W'}$ be \aGAM-homomorphism.  Define a relation by
  $v_1\sim v_2$ if $\pi(v_1)=\pi(v_2)$.  As $\pi$ is surjective by
  \Prpenum{inj-sur}{mow-hom.a}, this defines an equivalence relation
  on $V(\G)$.  Let $\map\Phi{\G/{\sim}}{\G'}$ be given by
  $\Phi([v])=\pi(v)$ and $\Phi (e)=\pi(e)$. It is easy to see that
  $\Phi$ is a graph homomorphism, bijective on vertices and injective
  on edges.  Moreover, the magnetic potentials are preserved.  If
  $\Phi$ was not surjective on the edges, then $\pi$ would not be
  surjective on the edges, hence \Lemenum{graph-aux}{graph-aux.b}
  contradicts the fact that $\pi$ is measure preserving on the
  vertices.  In particular, $\Phi$ is \aGAM-isomorphism, hence
  $\G'\simeq\G/{\sim}$.
	
  \itemref{gam-homo.std.c}$\Rightarrow$\itemref{gam-homo.std.a}.~Let
  $\map{f}{\G'}{\G/{\sim}}$ be an isomorphism and let
  $\map{\kappa}{\G}{\G/{\sim}}$ be the quotient map. It is
  straightforward to show that the composition $\pi=f\circ\kappa$ is
  \aGAM-homomorphism.
\end{proof}

The next lemma states a characterisation of \GAM-homomorphisms for the
combinatorial weight; its proof is similar to the previous lemma.
\begin{proposition}[Characterisations of \GAM-homomorphisms for
  combinatorial weight]
  \label{prp:gam-homo.comb}
  Let $\W=\Wfullcomb$ and $\W'=(\G',\alpha',\com)$, the following are
  equivalent:
  \begin{enumerate}
  \item 
    \label{gam-homo.comb.a}
    There exists \aGAM-homomorphism $\map \pi \W {\W'}$.
  \item 
    \label{gam-homo.comb.b}
    There is an equivalence relation ${\sim}$ on $V(\G)$ and a subset
    $E_0$ such that $\G'-E_0\simeq(\G/{\sim})$ with
    $\alpha\restr{E(\G)}=\alpha'$.
  \end{enumerate}
\end{proposition}

\begin{definition}[Geometric (pre)order of \GAM-graphs]
  \label{def:relation1}
  Let $\W$ and $\W'$ be two \GAM-graph.  If there exists
  $\map \pi {\W} {\W'}$ \aGAM-homomorphism, we write $\W\lesse\W'$.
\end{definition}
Note that, by definition, $\lesse$ is invariant under
\GAM-isomorphism.
\begin{proposition}
  \label{prp:relation1}
  The relation $\lesse$ is a preorder on $\Ga$, i.e.\ it is reflexive
  and transitive. Moreover, for finite \GAM-graphs $\lesse$ is a
  partial order on the equivalence classes of \GAM-isomorphic graphs
  in $\De$ and in $\Co$ (or on any subclass of weighted graphs with
  constant edge weight).
\end{proposition}
\begin{proof}
  For the reflexivity of $\lesse$ use $\pi=\id_G$ as
  \GAM-homomorphism.  The transitivity follows from
  \Exenum{gam-homo}{ex:gam-homo5}, hence $\lesse$ is a preorder on
  $\Ga$.
	
  Consider two finite \GAM-graphs $\W,\W'\in\Co$ with
  $\map{\pi}{\W}{\W'}$ and $\map{\pi'}{\W'}{\W}$, then by
  \Prpenum{inj-sur}{mow-hom.a} that both $\pi$ and $\pi'$ are
  surjective on the vertex sets.  Since both sets are finite, it
  follows that $\pi$ and $\pi'$ are bijective on the vertex sets.
  Similarly, from \Prpenum{inj-sur}{mow-hom.b} it follows that $\pi$
  and $\pi'$ are bijective on the edge sets: In particular, $\W$ and
  $\W'$ are \GAM-isomorphic, and hence $\lesse$ is antisymmetric
  (i.e.\ a partial order) on the equivalence classes of
  \GAM-isomorphic graphs from $\De$ or $\Co$.
\end{proof}
The preceding result remains true on any subclass $\Ga'$ of $\Ga$ with
edge weight given by a common constant.

\begin{remark}
  On \emph{infinite} \GAM-graphs, $\lesse$ is in general \emph{not} a
  partial order on $\De$ or $\Co$, as for \emph{infinite} \GAM-graphs,
  the antisymmetry may fail: Consider e.g.\ $\W=(\G,0,\deg)$ and
  $\W'=(\G',0,\deg)$, where $\G$ and $\G'$ are given in \Fig{counter}.
  It is easy to see that $\G'\simeq \G/\{u,v\}$ and
  $\G\simeq \G'/\{u',v'\}$, therefore $\W\lesse \W'$ and
  $\W'\lesse \W$ by \Prpenum{gam-homo.std}{gam-homo.std.c}.

  Nevertheless, $\W \not\simeq \W'$ because $\G$ and $\G'$ are not
  isomorphic as graphs.  In particular, $\lesse$ is not antisymmetric
  for infinite graphs for standard weights.  A similar argument works
  for combinatorial weights.
  \begin{figure}[h]
    \centering \subfloat[{The infinite graph
      $\G$.}\label{subfig:counter1}]{ \qquad \begin{tikzpicture}[auto,
        vertex/.style={circle,draw=black!100,fill=black!100, thick,
          inner sep=0pt,minimum size=1mm},scale=4]
        \node (C) at (-.55,0) [vertex,inner sep=.25pt,minimum size=.25pt,label=above:]{};
        \node (C) at (-.5,0) [vertex,inner sep=.25pt,minimum size=.25pt,label=above:]{};
        \node (C) at (-.45,0) [vertex,inner sep=.25pt,minimum size=.25pt,label=above:]{};                              
        \node (X2) at (-.4,0) [vertex,label=below:]{};
        \node (X3) at (-.2,0) [vertex,,label=below:]{};
        \node (X4) at (0,0) [vertex,label=below:]{};
        \node (X5) at (.2,0) [vertex,label=below:]{};
        \node (X6) at (.4,0) [vertex,label=below:$u$]{};   
        \node (X7) at (.6,0) [vertex,label=below:]{};
        \node (X8) at (.8,0) [vertex,label=below:]{};
        \node (X9) at (1,0) [vertex,label=below:]{};   
        \draw[-] (X2) to[] node[above] {} (X3);                              	                
        \draw[-] (X3) to[] node[above] {} (X4);  
        \draw[-]  (X4) to[] node[above] {} (X5); 
        \draw[-]  (X5) to[] node[above] {} (X6); 
        \draw[-]  (X6) to[] node[above] {} (X7);  
        \draw[-]  (X7) to[] node[above] {} (X8); 
        \draw[-]  (X8) to[] node[above] {} (X9);     
        \node (Y0) at (-.4,.2) [vertex,label=above:]{};    
       \draw[-]  (X2) to[] node[above] {} (Y0); 	   	
        \node (Y1) at (-.2,.2) [vertex,label=above:]{};    
       \draw[-] (X3) to[] node[above] {} (Y1); 
        \node (Y2) at (0,.2) [vertex,label=above:]{};    
        \draw[-]  (X4) to[] node[above] {} (Y2);
        \draw[-]  (X5) to [out=135,in=40,looseness=30] (X5);
        \node (Y3) at (.4,.2) [vertex,label=above:$v$]{};    
        \draw[-]  (X6) to[] node[above] {} (Y3);  
        \draw[-]  (X7) to [out=135,in=40,looseness=30] (X7);	
        \draw[-]  (X8) to [out=135,in=40,looseness=30] (X8);			
        \draw[-] (X9) to [out=135,in=40,looseness=30] (X9);		
        
        \node (C) at (1.05,0) [vertex,inner sep=.25pt,minimum size=.25pt,label=above:]{};
        \node (C) at (1.1,0) [vertex,inner sep=.25pt,minimum size=.25pt,label=above:]{};
        \node (C) at (1.15,0) [vertex,inner sep=.25pt,minimum size=.25pt,label=above:]{};                               
      \end{tikzpicture}\qquad}                
    \subfloat[{The infinite graph $\G'$.}\label{subfig:counter2}]{  \qquad  \begin{tikzpicture}[auto,
	vertex/.style={circle,draw=black!100,fill=black!100, thick,
          inner sep=0pt,minimum size=1mm},scale=4]
	\node (C) at (-.55,0) [vertex,inner sep=.25pt,minimum size=.25pt,label=above:]{};
	\node (C) at (-.5,0) [vertex,inner sep=.25pt,minimum size=.25pt,label=above:]{};
	\node (C) at (-.45,0) [vertex,inner sep=.25pt,minimum size=.25pt,label=above:]{};                              
	\node (X2) at (-.4,0) [vertex,label=below:]{};
	\node (X3) at (-.2,0) [vertex,,label=below:]{};
	\node (X4) at (0,0) [vertex,label=below:]{};
	\node (X5) at (.2,0) [vertex,label=below:$u'$]{};
	\node (X6) at (.4,0) [vertex,label=below:]{};   
	\node (X7) at (.6,0) [vertex,label=below:]{};
	\node (X8) at (.8,0) [vertex,label=below:]{};
	\node (X9) at (1,0) [vertex,label=below:]{};   
	\draw[-] (X2) to[] node[above] {} (X3);                              	                
	\draw[-] (X3) to[] node[above] {} (X4);  
	\draw[-] (X4) to[] node[above] {} (X5); 
	\draw[-] (X5) to[] node[above] {} (X6); 
	\draw[-] (X6) to[] node[above] {} (X7);  
	\draw[-] (X7) to[] node[above] {} (X8); 
	\draw[-] (X8) to[] node[above] {} (X9);     
	\node (Y0) at (-.4,.2) [vertex,label=above:]{};    
	\draw[-] (X2) to[] node[above] {} (Y0); 	   	
	\node (Y1) at (-.2,.2) [vertex,label=above:]{};    
	\draw[-] (X3) to[] node[above] {} (Y1); 
	\node (Y2) at (0,.2) [vertex,label=above:]{};    
	\draw[-] (X4) to[] node[above] {} (Y2);
	\node (Ya) at (.2,.2) [vertex,label=above:$v'$]{};    
    \draw[-] (X5) to[] node[above] {} (Ya);	
	\node (Y3) at (.4,.2) [vertex,label=above:]{};    
	\draw[-] (X6) to[] node[above] {} (Y3);  
	\draw[-] (X7) to [out=135,in=40,looseness=30] (X7);	
	\draw[-] (X8) to [out=135,in=40,looseness=30] (X8);			
	\draw[-] (X9) to [out=135,in=40,looseness=30] (X9);		
	
	\node (C) at (1.05,0) [vertex,inner sep=.25pt,minimum size=.25pt,label=above:]{};
	\node (C) at (1.1,0) [vertex,inner sep=.25pt,minimum size=.25pt,label=above:]{};
	\node (C) at (1.15,0) [vertex,inner sep=.25pt,minimum size=.25pt,label=above:]{};                               
      \end{tikzpicture}\qquad}                	
    \caption{}
    \label{fig:counter}
  \end{figure}
\end{remark}

%
\section{Magnetic Laplacians and spectral preorder}
\label{sec:mag.lap.spec.ord}
%

In this section, we will introduce the discrete magnetic Laplacian
associated to an \GAM-graph and present a new spectral relation
between the magnetic Laplacian associated to different graphs.

\subsection{Discrete magnetic exterior derivatives and Laplacians}
\label{sec:disc.mag.lap}

We define some standard spaces related with a weighted graph $(G,\m)$, namely for $p \in [1,\infty)$ we set
are
\begin{subequations}
  \label{eq:hs}
  \begin{align}
    \label{eq:hs.0}
     \lp {V,\m}
    &:= \Bigset{ \map \phi { V} \C}
          { \norm[V,\m] \phi^p := 
            \sum_{v \in  V} \abs{\phi(v)}^p \m(v) < \infty},\\
          \label{eq:hs.1}
         \lp {E,\m}
          &:= \Bigset{ \map \eta E \C}
            {\forall e \in E\colon \eta_{\bar e}=-\eta_e,\;
            \norm[E,\m] \eta^p := 
                  \frac12\sum_{e \in E} \abs{\eta_e}^p \m_e < \infty},
  \end{align}
\end{subequations}
For $p=2$, both spaces are
Hilbert spaces.  Note that functions on $V$ can be interpreted as
$0$-forms and functions on $E$ are considered as $1$-forms. We will
denote the corresponding canonical orthonormal basis by
$\set{\delta_v}{v\in V}$ (respectively, $\set{\delta_e}{e\in E}$) with
$\delta_v(u)=\m(v)^{-1/2}$ if $v=u$ and $\delta_v(u)=0$ otherwise
(respectively, $\delta_e(e')=w_e^{-1/2}$ if $e=e'$ and
$\delta_e(e')=0$ otherwise). If $\m=\1$, we sometimes simply write
$\lsqr{V,\1}=\lsqr V$ and $\lsqr{E,\1}=\lsqr E$.

Let $\W=\Wfull$ be \aGAM-graph.  The \emph{(discrete) magnetic
  exterior derivative} $d_\alpha$ is defined as
\begin{equation}
  \label{eq:d.alpha}
  \map {d_\alpha} 
  {\lsqr{V,\m}}{\lsqr{E,\m}}, \qquad
  (d_\alpha \phi)_e 
  = \e^{\im \alpha_e/2} \phi(\bd_+ e) 
  - \e^{- \im \alpha_e/2} \phi(\bd_-e).
\end{equation}
It is not hard to see that $d_\alpha$ is a bounded operator if the
relative weight $\varrho_\infty$ defined in~\Eq{eq:rho.bdd} is
bounded (actually, the boundedness of $d_\alpha$ is equivalent with
the boundedness of $\varrho_\infty$). In~\cite{mathai-yates:02},
$d_\alpha$ is called a \emph{coboundary operator} for a \emph{twisted
  complex}. In particular, if $\alpha = 0$, then $d_\alpha$ is the usual
coboundary operator.

The adjoint $\map{d_\alpha}{\lsqr{E,\m}}{\lsqr{V,\m}}$ is given by
\begin{equation*}
  (d_\alpha^* \eta)(v)
  = - \frac 1 {\m(v)} \sum_{e \in E_v} 
      \m_e \e^{\im \alpha_e/2}\eta_e\;.
\end{equation*}

\begin{definition}[Discrete magnetic weighted Laplacian]
  \label{def:dml}
  Let $\W$ be \aGAM-graph, the \emph{(discrete) magnetic (weighted)
    Laplacian} is defined as
  \begin{equation}
    \label{eq:def.disc.lapl}
    \map{\Delta_\alpha := d_\alpha^* d_\alpha} 
    {\lsqr { V,\m}} {\lsqr { V,\m}}.
  \end{equation}
\end{definition}
The discrete magnetic weighted Laplacian acts as
\begin{equation}
  \label{eq:disc.lapl}
  (\Delta_\alpha \phi)(v) 
  =  \frac 1 {\m(v)} \sum_{e \in E_v} 
       \m_e \bigl(\phi(v)-\e^{\im \alpha_e} \phi(\bd_+e)\bigr)
  = \varrho_\m(v) \phi(v) 
  - \frac 1 {\m(v)} \sum_{e \in E_v} 
       \m_e \e^{\im \alpha_e} \phi(\bd_+e),
\end{equation}
where $\varrho_\m$ is the relative weight defined in~\Eq{eq:rho}.
\begin{example}[Special cases of magnetic weighted Laplacians]
  \label{ex:lapl.weights}
  \indent
  \begin{enumerate}
  \item
    \label{lapl.comb}
    Let $\W=\Wfull\in\Ga$ be \aGAM-graph (respectively, $\W\in\De$,
    $\W\in\Co$), then $\Delta_\alpha$ is the magnetic weighted
    (respectively, standard, combinatorial) Laplacian.
  \item
    \label{lapl.std}
    Let $\W=(\G,0,\m)\in\Ga^0$ be \aGAM-graph (respectively,
    $\W\in\De^0$, $\W\in\Co^0$), then $\Delta_0$ is the weighted
    (respectively, standard, combinatorial) Laplacian.
  \item
    \label{lapl.signless}
    Let $\W=(\G,\pi,\m)\in\Ga^\pi$ be \aGAM-graph (respectively,
    $\W\in\De^\pi$, $\W\in\Co^\pi$), then $\Delta_\pi$ is the weighted
    (respectively, standard, combinatorial) \emph{signless} Laplacian.
  \end{enumerate}
\end{example}

By construction, the magnetic Laplacian is a bounded, non-negative
(hence self-adjoint) operator.  Moreover, its spectrum is contained in
the interval $\sigma(\Delta_\alpha) \subset [0,2\varrho_\infty]$. %
Let $\wt \alpha$ and $\alpha$ be two cohomologous (gauge-equivalent)
magnetic potentials for some gauge $\xi \in C^0(G)$ (i.e.\
$\map \xi V {A=\RmodZ}$ with $\wt \alpha = \alpha + \de \xi$).  Then
the gauge $\xi$ induces two unitary (multiplication) operators $\Xi^0$
and $\Xi^1$ on $\lsqr{V,\m}$ and $\lsqr {E,\m}$, respectively, defined
by
\begin{equation}
  \label{eq:unit.0.1}
  (\Xi^0 \phi)(v) := \e^{\im \xi(v)} \phi(v)
  \quadtext{and}
  (\Xi^1 \eta)_e := \e^{\im(\xi(\bd_+ e) + \xi(\bd_- e))/2} \eta_e.
\end{equation}
Here, $\xi \mapsto \Xi^0$ and $\xi \mapsto \Xi^1$ are unitary
representations of $\xi \in C^0(G)$ seen as an additive group on
$\lsqr{V,\m}$ and $\lsqr{E,\m}$.

We will now show that magnetic Laplacians with gauge-equivalent
magnetic potentials are unitarily equivalent:
\begin{proposition}
  \label{prp:sunada.weight}
  If $\alpha\sim \wt\alpha$ (or, more precisely, if
  $\wt \alpha=\alpha+\de \xi$), then
  \begin{equation*}
    d_\alpha \Xi^0 = \Xi^1 d_{\wt\alpha}
    \qquadtext{and}
    \laplacian \alpha \Xi^0= \Xi^0\laplacian {\wt \alpha}.
  \end{equation*}
  In particular, $\laplacian {\alpha}$ and $\laplacian {\wt \alpha}$ are
  unitarily equivalent, and the spectral properties of $\laplacian
  \alpha$ depend only on the \GAM-graph class $[\alpha]$.
\end{proposition}
\begin{proof}
  The first equation follows by a straightforward calculation, namely
  \begin{align*}
    (d_\alpha \Xi^0 \phi)_e
    & = \e^{\im \alpha_e/2+\im \xi(\bd_+e)}\phi(\bd_+e)
      - \e^{-\im \alpha_e/2+\im \xi(\bd_-e)}\phi(\bd_-e)\\
    &=  \e^{\im (\xi(\bd_+e) + \xi(\bd_-e))/2}
      \bigl(
          \e^{\im \wt \alpha_e/2}\phi(\bd_+e)
        - \e^{-\im \wt \alpha_e/2}\phi(\bd_-e)
      \bigr) = (\Xi^1 d_{\wt \alpha} \phi)_e.
  \end{align*}
  The second intertwining relation follows from the first one and the
  fact that $\Xi^0$ and $\Xi^1$ are unitary.
\end{proof}

Now, we will prove some results related with the spectrum of the
magnetic Laplacian that will be useful in the next sections.  The
first result says that a magnetic potential increases the smallest
eigenvalue.
\begin{lemma}
  Let $\W=\Wfull$ be \aGAM-graph, such that the underlying discrete
  graph $\G$ is connected and finite.  Then
  $0\in\sigma(\laplacian {\alpha})$ if and only if $\alpha$ is
  trivial, i.e.\ cohomologous to $0$.
\end{lemma}
\begin{proof}
  ``$\Rightarrow$'': Suppose that $0\in\sigma(\laplacian \alpha)$,
  then there exists a non-zero $\phi \in \lsqr {V,\m} $ such that
  $0=\iprod {\Delta_\alpha \phi}\phi=\normsqr{d_\alpha \phi}$.  In
  particular, $d_\alpha \phi=0,$ i.e.\ 
  $\phi(\bd_-e)=\e^{\im \alpha_e}\phi(\bd_-e)$ for all $e \in E$.  As
  the graph $\G$ is connected we can define a function
  $\xi\colon V \rightarrow A$ adjusting the phases in such a way that
  $\phi(v)\e^{\im \xi(v)}$ is constant on $V$.  We then have
  $\alpha_e=(\de \xi)_e$ ($e\in E$), hence $\alpha\sim 0$.

  ``$\Leftarrow$'': If $\alpha \sim 0$, then by \Prp{sunada.weight},
  $\Delta_\alpha$ is unitarily equivalent with the Laplacian
  $\Delta_0$ without magnetic potential.  For the latter, a constant
  function is an eigenfunction with eigenvalue $0$, hence
  $0\in\sigma(\Delta_\alpha)$.
\end{proof}

We conclude this section recalling some variations of the well-known
variational characterisation of the eigenvalues (min-max principle):

\begin{theorem}[Courant-Fischer]
  \label{thm:courant}
  Let $\HS$ be an $n$-dimensional (complex) Hilbert space and
  $\map A \HS\HS$ linear and $A^*=A$.  Moreover, denote by
  $\lambda_1\leq \lambda_2 \leq \dots \leq \lambda_n$ be the
  eigenvalues of $A$ written in ascending order and repeated according
  to their multiplicities. Let $k\in\{1,2,\dots, n\}$, then
  \begin{equation}
    \label{eq:courant1}
    \lambda_k
    =\min_{S \in \mathcal S_{n-k}}\max_{\substack{\phi\perp S \\ \phi\neq 0}}
    \dfrac{\iprod{A \phi} \phi}{\iprod{\phi}\phi}
    \qquadtext{and}
    \lambda_k
    =\max_{S \in \mathcal S_{k-1}} \min_{\substack{\phi\perp S \\ \phi\neq 0}}
    \dfrac{\iprod{A \phi} \phi}{\iprod{\phi}\phi},
  \end{equation}
  and
  \begin{equation}
    \label{eq:courant2}
    \lambda_k
    =\min_{S \in \mathcal S_k} \max_{\substack{\phi \in S \\ \phi\neq 0}}
    \dfrac{\iprod{A \phi} \phi}{\iprod{\phi}\phi}
    \qquadtext{and}
    \lambda_k
    =\max_{S \in \mathcal S_{n-k+1}} \min_{\substack{\phi \in S \\ \phi\neq 0}}
    \dfrac{\iprod{A \phi} \phi}{\iprod{\phi}\phi},
  \end{equation}
  where $\mathcal S_k$ denotes the set of $k$-dimensional subspaces of
  $\HS$.
\end{theorem}

Let $\W=(\G,\alpha,\m)\in\Ga$ be a finite \GAM-graph of order $n$.
The eigenvalues of its discrete magnetic Laplacian $\Delta_\alpha$ can
be written as
\begin{equation}
  \label{eq:spec-order}
  \sigma(\W)
  =\sigma(\Delta_\alpha)=
  \left(   \lambda_1(\W), \lambda_2(\W), \dots ,\lambda_n(\W)
  \right),
\end{equation}
where the eigenvalues are written in ascending order and repeated
according to their multiplicities, as in the above theorem.

\begin{remark}
  \label{rem:std.deg.zero}
  For a graph with only one vertex $v$ with weight $w(v)>0$ and no
  edge, the corresponding Laplacian is $0$, as the sum over
  $e \in E_v$ is empty; the spectrum in this case is $(0)$.  More
  generally, any isolated vertex in a graph contributes with an extra
  $0$ in the list of eigenvalues.
  
  For the standard weight, we have $w(v)=\deg(v)=0$.  In this case, it
  is convenient to associate to the graph with only one vertex and no
  edges again the eigenvalue $0$.  Hence any isolated vertex of the
  standard Laplacian contributes with one extra eigenvalue $0$ in the
  list of eigenvalues (see e.g.~\cite[bottom of p.~2]{chung:97}; in
  this way, the case $r=0$ also applies in
  \Thmenum{vert-contra}{vert-contra.a} for the standard weight).
\end{remark}

\subsection{Order on sets of finite sequences}
\label{subsec:ord.seqJ} 

We next relate spectra of different \GAM-graphs.  To do so, we first
introduce an order on the set of finite increasing sequences of real
numbers:
\begin{definition}
  \label{def:with-shift}
  Let $\Lambda$ and $\Lambda'$ two sequences of real numbers written in
  increasing order with lengths $n$ and $n'$ respectively, i.e.\ 
  \begin{equation*}
    \Lambda
    := \set{(\lambda_1, \lambda_2 \dots,\lambda_n)}
    { \lambda_1 \le \lambda_2 \le \dots \lambda_{n-1}\le\lambda_n};
  \end{equation*}	
  \begin{equation*}
    \Lambda'
    := \set{(\lambda'_1, \lambda'_2 \dots,\lambda'_{n'})}
    {\lambda'_1 \le \lambda'_2 \le \dots \lambda'_{n'-1}\le\lambda'_{n'}}.
  \end{equation*}
  Given $r \in \N_0$, we say that $\Lambda$ is \emph{smaller than (or
    equal to) $\Lambda$ with shift $r$} (and denote this by
  $\Lambda \less[r] \Lambda'$) if $n\geq n'-r$ and
  \begin{equation*}
    \lambda_k \leq \lambda'_{k+r} \quadtext{for} 1\leq k\leq n'-r.
  \end{equation*}	
  The \emph{length} of the sequence $\Lambda$ is defined by
  $\abs{\Lambda }:=n$.
\end{definition}

\begin{remark}
  \label{rem:spec.ord}
  Let $\Lambda, \Lambda'$ and $\Lambda''$ be increasing sequences of real numbers
  as above.
  \begin{enumerate}
  \item We denote $\Lambda \less[0] \Lambda'$ simply by
    $\Lambda \less \Lambda'$. 
  \item If $\Lambda \less[r] \Lambda'$ and
    $\Lambda' \less[s] \Lambda''$ we will write
    $\Lambda \less[r] \Lambda' \less[s] \Lambda''$.
  \item The case $\Lambda \less \Lambda' \less[r] \Lambda$ implies
    that $n= \abs \Lambda \ge \abs{\Lambda'} \ge n-r$.  If
    $\abs{\Lambda'}=n-r$, then
    $\Lambda \less \Lambda' \less[1] \Lambda$ is equivalent with the
    \emph{interlacing} of $\Lambda$ and $\Lambda'$ similarly as
    in~\cite[Section~2.5]{brouwer-haemers:12}, namely
    \begin{equation*}
      \lambda_1 \le \lambda_1' \le \lambda_{1+r}, \quad
      \lambda_2 \le \lambda_2' \le \lambda_{2+r}, \quad
      \dots, \quad
      \lambda_{n-r} \le \lambda_{n-r}' \le \lambda_n.
    \end{equation*}
    Especially if $r=1$ it becomes the usual interlacing (explaining
    also the name).
    \begin{equation*}
      \lambda_1 \le \lambda_1' \le \lambda_2 \le \lambda_2'
      \le \dots \le \lambda_{n-1} \le \lambda_{n-1}' \le \lambda_n.
    \end{equation*}
  \end{enumerate}
\end{remark}

We state in the next lemma some direct consequences of the preceding definition.
\begin{lemma}
  \label{lem:basic}
  Let $\Lambda$, $\Lambda'$ and $\Lambda''$ be three increasing sequences of real
  numbers and consider $r,s\in\mathbb{N}_0$.
  \begin{enumerate}
  \item 
    \label{basic:1} 
    If $r \in \N_0$, then $\Lambda \less[r] \Lambda$
    \emph{(reflexivity)}.
  \item 
    \label{basic:2} 
    If $\Lambda \less \Lambda' \less \Lambda$, then $\Lambda=\Lambda'$
    \emph{(antisymmetry)}.
  \item 
    \label{basic:3} 
    If $\Lambda \less[r] \Lambda'$ and
    $\Lambda' \less[s] \Lambda''$, then $\Lambda \less[r+s]\Lambda''$
    \emph{(transitivity)}.
  \item 
    \label{basic:4} 
    If $\Lambda \less[r] \Lambda'$ and $r\leq s$, then
    $\Lambda \less[s] \Lambda'$.
  \item
    \label{basic:5}
    If $\Lambda \less \Lambda'$ with $n=\abs \Lambda = \abs{\Lambda'}$
    and if $\sum_{i=1}^n \lambda_i=\sum_{i=1}^n \lambda'_i$, then
    $\Lambda=\Lambda'$, i.e.\ $\lambda_i=\lambda'_i$, $i=1,\dots,n$
    (equal sums).
  \end{enumerate}
\end{lemma}
  In case $\abs\Lambda=\abs{\Lambda'}$ then $\Lambda \less \Lambda'$
  implies (in the usual sense of majorisation of vectors in $\R^d$, cf.,~\cite{moa:11}) that
  $\Lambda'$ majorises $\Lambda$. Nevertheless the fact that, in addition, we allow
  the shift $r$ as a parameter is particularly convenient to study and
  relate spectra of Laplacians.
\begin{proof}
  The reflexivity property and part \itemref{basic:4} follow from the
  fact that the sequence $\Lambda$ is written in increasing order. The
  antisymmetry property \itemref{basic:2} is also a direct consequence
  of \Def{with-shift}. To show \itemref{basic:3}, note
  that $\card{\Lambda}+r\geq \card{\Lambda'}$ and
  $\lambda_k \leq \lambda'_{k+r}$ for $1\leq k\leq \card{\Lambda'}-r$,
  as well as $\card{\Lambda'}+s\geq \card{\Lambda''}$ and
  $\lambda_l \leq \lambda'_{l+s}$ for
  $1\leq l\leq \card{\Lambda''}-s$. This implies
  $\card{\Lambda}+r+s\geq \card{\Lambda''}$ and
  $\lambda_k \leq \lambda'_{k+r+s}$ for
  $1\leq k\leq \card{\Lambda''}-r-s$.  To show \itemref{basic:5} note
  first that the statement is trivial for $n=1$. Assume that it is
  true for $n \in \N$.  We then conclude from $\Lambda \less \Lambda'$
  for two sequences of length $n+1$, that
  \begin{equation*}
    \sum_{k=1}^n \lambda_k -\sum_{k=1}^n \lambda'_k \le 0
    \qquadtext{and}
    \lambda'_k - \lambda_k \ge 0.
  \end{equation*}
  By assumption (equality of the traces for $n+1$) both above left
  hand sides are equal, hence equal to $0$.  It follows by the
  induction hypothesis that $\lambda_k = \lambda'_k$ for
  $k=1,\dots, n$, and hence also $\lambda_{n+1}=\lambda'_{n+1}$, again
  by the equality of the traces for $n+1$.
\end{proof}

\subsection{Spectral preorder}
\label{subsec:spec-order}
We will now apply the relation $\less$ described before to relate the
spectrum of the magnetic Laplacian on different \GAM-graphs.
\begin{definition}[Spectral preorder of magnetic weighted graphs and
  isospectrality]
  \label{def:spectral-order}
  Let $\W,\W'\in \Ga$ be two finite \GAM-graphs.  We say that
  \emph{$\W$ is (spectrally) smaller than $\W'$ with shift $r$}
  (denoted by $\W\less[r]\W'$) if $\sigma(\W) \less[r] \sigma(\W')$,
  where $\sigma(\W)$ and $\sigma(\W')$ are the spectra of the
  corresponding discrete magnetic weighted Laplacians as
  in~\Eq{eq:spec-order}, i.e.\ if $\card \W + r \ge \card {\W'}$ and
  if
  \begin{equation*}
    \lambda_k(\W) \le \lambda_{k+r}(\W')
    \qquad \text{for all $1 \le k \le \card{\W'}-r$.}
  \end{equation*}
  If $r=0$ we write again simply $\W \less \W'$.

  We say that $\W$ and $\W'$ are \emph{isospectral}, if $\W \less \W'$
  and $\W' \less \W$, i.e., if the (magnetic weighted) Laplace
  spectrum of $\W$ and $\W'$ agrees; in other words if both graphs
  have the same number $n$ of vertices and
  $\lambda_k(\W)=\lambda_k(\W')$ for all $k=1,\dots,n$.
\end{definition}

\begin{proposition}
  \label{prp:relation2}
  The relation $\less$ is a preorder on $\Ga$. 
\end{proposition}
\begin{proof}
  \Lemenums{basic}{basic:1}{basic:3} show the reflexivity and
  transitivity of $\less$ using shifts $r=s=0$.
\end{proof}

Note that $\less$ is invariant under \GAM-isomorphisms, as
\GAM-isomorphisms have isospectral magnetic Laplacians.  Since there
are non-isomorphic isospectral graphs it follows that $\less$ is not
antisymmetric, i.e.\ equality of spectra does not imply that the
graphs are (\GAM-)isomorphic.  In particular, $\less$ is not a partial
order.

\begin{remark}
  \label{rem:spectral.order}
\begin{itemize}
\item[(i)] The second smallest eigenvalue of the usual Laplacian gives
  a measure of the connectivity of the graph (see~\cite{fiedler:73}
  and \Subsec{cheeger}). Defining $a(\W):=\lambda_2(\W)$ we obtain
  directly from the definition of the spectral preorder that
  \begin{equation*}
    \W\less\W' \implies a(\W)\leq a(\W')\;.
  \end{equation*}
  \Cors{delete-edge}{vert-contra} give a quantitative measure of the
  fact that deleting edges reduces the connectivity and contracting
  vertices increases the connectivity of the graph.

 \item[(ii)] The name \emph{spectral order} has been introduced by
  Olson~\cite{olson:71} for two self-adjoint (bounded) operators $T_1$
  and $T_2$ in a Hilbert space $\HS$ with spectral resolutions
  $E_j(t):=\1_{(-\infty,t]}(T_j)$ ($j=1,2$).  Then $T_1 \preceq T_2$ if
  and only if $E_1(t) \le E_2(t)$ for all $t \in \R$ (i.e.\ if
  $\iprod{E_1(t)\phi}\phi \le \iprod{E_2(t)\phi} \phi$ for all
  $\phi \in \HS$).  If $T_1 \ge 0$ and $T_2 \ge 0$, then
  $T_1 \preceq T_2$ is equivalent with $T_1^p \le T_2^p$ for all
  $p \in \N$.  If both operators have purely discrete spectrum
  $\lambda_k(T_j)$ (written in increasing order and repeated according
  to multiplicity) then we have the implications
  \begin{equation*}
    T_1 \preceq T_2 
    \quad\Rightarrow\quad
    T_1 \le T_2
    \quad\Rightarrow\quad
    T_1 \less T_2,
  \end{equation*}
  where the latter means that $\lambda_k(T_1)\le \lambda_k(T_2)$ for
  all $k$ (the latter implication follows from the min-max principle
  as in \Thm{courant}).
\end{itemize}

\end{remark}

Next we lift \GAM-homomorphism to spaces of functions on vertices and edges:
\begin{lemma}
  \label{lem:j}
  Let $\map{\pi}{\W}{\W'}$ be \aGAM-homomorphism where $\W=\Wfull$,
  $\W'={\Wtwofull}$ and $\G=(V,E,\bd)$, $\G'=(V',E',\bd')$ denote the underlying 
  graphs. Define the natural
  the identification operators on functions over vertices
  $\map{J^0}{\lsqr{V',\m'}}{\lsqr{V, \m}}$ and edges
  $\map{J^1}{\lsqr{E',\m'}}{\lsqr{E, \m}}$ by
  $J^0\phi = \phi \circ \pi$ and $J^1\eta = \eta \circ \pi$,
  respectively. Then the following holds:
  \begin{enumerate}
  \item
    \label{j.a}
    We have
    $\norm[\lsqr{V,\m}] {J^0\phi} \ge \norm[\lsqr{V',\m'}] \phi$ for
    all $\phi \in \lsqr{V',\m'}$.  In particular, $J^0$ is injective.
    If $\pi$ is vertex measure preserving, then $J^0$ is an isometry.
  \item
    \label{j.b}
    We have
    $\norm[\lsqr{E, \m}] {J^1\eta} \le \norm[\lsqr{E',\m'}] \eta$ for
    all $\eta \in\lsqr{E',\m'}$.  If $\pi$ is edge measure preserving,
    then $J^1$ is an isometry.
  \item
    \label{j.c}
    We have $d_{\alpha} J^0 = J^1d_{\alpha'}$.
  \end{enumerate}
\end{lemma}
\begin{proof}
  \itemref{j.a}~From \Defenum{hom.gam-graph}{hom.mow-graph.c} we have:
  \begin{align*}
    \normsqr[\lsqr{V, \m}] {J^0\phi}
    &= \sum_{v\in V}  \m(v) \abssqr{(\phi\circ\pi)(v)}\\
    & = \sum_{v' \in V'}  
      \Bigl(\sum_{v \in \pi^{-1}(v')} \m(v) \Bigr) \abssqr{\phi(v')}
      = \sum_{v' \in V'}  
      (\pi_* \m)(v') \abssqr{\phi(v')} \\ 
    &\ge \sum_{ v'\in V'} \m'(v') \abssqr{\phi(v')}
      = \normsqr[\lsqr{V',\m'}] \phi.
  \end{align*}
  Clearly, if $\pi$ is vertex measure preserving, then $\pi_*\m=\m'$
  on $V'$, and equality in the above estimate holds.
	
  \itemref{j.b}~The assertion of the identification map $J^1$ on the
  edges follows similarly from
  \Defenum{hom.gam-graph}{hom.mow-graph.d}.
	
  \itemref{j.c}~This intertwining equation follows immediately from
  the properties of \GAM-homomorphism given in
  \DefenumS{hom.gam-graph}{hom.mow-graph.a}{hom.mow-graph.b}.
  \qedhere
\end{proof}

The following result showing that the geometric preorder is stronger
than the spectral preorder follows from the min-max principle
mentioned in \Thm{courant}
Recall that $\W \lesse \W'$ means that there is \aGAM-homomorphism
$\pi\colon\W\to\W'$, see \Defs{hom.gam-graph}{relation1}.
\begin{theorem}
  \label{thm:homomorphism}
  Let $\W, \W'\in\Ga$, then
  \begin{equation*}
   \W\lesse \W'\qquadtext{implies} \W \less \W'.
  \end{equation*}
  Moreover, if the \GAM-homomorphism $\pi\colon\W\to\W'$ is (vertex
  and edge) measure preserving, then we have additionally
  \begin{equation*}
    \W' \less[r] \W, \qquadtext{where}  r=\card{ \W}-\card{\W'}\geq 0.
  \end{equation*}
\end{theorem}
\begin{proof}
  First note that $\pi$ is
  surjective on the set of vertices by \Prpenum{inj-sur}{mow-hom.a},
  hence $\abs \W \ge \abs{\W'}$ and therefore $r \ge 0$.  From \Lem{j}
  we conclude
  \begin{equation}
    \label{eq:rayleigh.id.op}
    \frac{\normsqr[\lsqr{E, \m}]{d_{\alpha} J^0 \phi'}}
    {\normsqr[\lsqr{V, \m}]{J^0 \phi'}}
    =
    \frac{\normsqr[\lsqr{E, \m}]{J^1 d_{\alpha'} \phi'}}
    {\normsqr[\lsqr{V, \m}]{J^0 \phi'}}
    \le
    \frac{\normsqr[\lsqr{E',\m'}]{d_{\alpha'} \phi'}}
    {\normsqr[\lsqr{V',\m'}]{\phi'}}.
  \end{equation}
  Denote by $S_k'$ the $k$-dimensional subspace of $\lsqr{V',\m'}$
  spanned by the first $k$ eigenfunctions of $\Delta_{\W'}$.  From the
  min-max characterisation of the $k$-th eigenvalue (first equality
  in~\Eq{eq:courant2}), we then have by the preceding estimate:
   \begin{align*}
     \lambda_k(\W)
     =\min_{S \in \mathcal S_k} \max_{\substack{\phi\in S \\ \phi\neq 0}}
     \frac{\normsqr[\lsqr{E, \m}]{d_{\alpha} \phi}}
     {\normsqr[\lsqr{V, \m}] \phi}
     &\leq \max_{\substack{\phi' \in S_k' \\ \phi' \neq 0}}
     \frac{\normsqr[\lsqr{E, \m}]{d_{\alpha} J^0 \phi' }}
     {\normsqr[\lsqr{V, \m}] {J^0 \phi'} } \\
     &\leq \max_{\substack{\phi' \in S'_k \\ \phi' \neq 0}} 
     \frac{\normsqr[\lsqr{E',\m'}]{d_{\alpha'} \phi'}}
     {\normsqr[\lsqr{V',\m'}]{\phi'}}
     = \lambda_k(\W')
   \end{align*}
   for all $1\leq k\leq \abs{\W'}$, where $\mathcal S_k$ is the set of
   all $k$-dimensional subspaces of $\ell_2(V,\m)$.  Moreover, as $J^0$
   is injective, $S=J^0(S_k')$ is also $k$-dimensional, i.e.\ 
   $J^0(S_k') \in \mathcal S_k$. This shows $\W\less\W'$.

   If $\pi$ is measure preserving, then $J^0$ and $J^1$ are
   isometries, hence we have equality in~\Eq{eq:rayleigh.id.op}.
   Moreover, let $n=\abs \W$, $n'=\abs{\W'}$ and denote by $T_k'$ the
   space generated by the $n-k+1$ eigenfunctions $\phi'_{n'-n+k}$,
   \dots, $\phi_{n'}'$ of the Laplacian on $\W'$, then we have
   similarly as before (second equality
  in~\Eq{eq:courant2})
   \begin{align*}
     \lambda_k(\W)
     =\max_{S \in \mathcal S_{n-k+1}} \min_{\substack{\phi \in S \\ \phi\neq 0}}
     \frac{\normsqr[\lsqr{E, \m}]{d_\alpha \phi}}
     {\normsqr[\lsqr{V, \m}] \phi}
     &\ge \min_{\substack{\phi' \in T'_k \\ \phi' \neq 0}}
     \frac{\normsqr[\lsqr{E, \m}]{d_{\alpha} J^0 \phi' }}
     {\normsqr[\lsqr{V, \m}] {J^0 \phi'} } \\
     &= \min_{\substack{\phi' \in T'_k \\ \phi' \neq 0}} 
     \frac{\normsqr[\lsqr{E',\m'}]{d_{\alpha'} \phi'}}
     {\normsqr[\lsqr{V',\m'}]{\phi'}}
     = \lambda_{n'-(n-k+1)+1}(\W')
     = \lambda_{k-r}(\W'),
   \end{align*}
   where $S=J^0(T_k')$ is $(n-k+1)$-dimensional since $J^0$ is
   injective. From Definition~\ref{def:with-shift} and~\ref{def:spectral-order}
   it follows that $\W'\less[r]\W$.
\end{proof}

\begin{remark}
  \label{rem:counterex.homo}
  Note that the converse statement of \Thm{homomorphism} is wrong in general,
  i.e.\ there are \GAM-graphs such that $\W \less \W'$ but not
  $\W \lesse \W'$.  As an example consider the preorder relations between 
  $\W_9$ and $\W_{10}$ in \Fig{order-graph}.
\end{remark}

We have the following simple consequence of the previous theorem and
\Exenum{gam-homo}{ex:gam-homo1};
\begin{corollary}
  \label{cor:std-com}
  Let $\W=(\G,\alpha,\deg)$ and $\W'=(\G,\alpha,\com)$, then
  $ \W \lesse \W'$.  In particular $\W \less \W'$, i.e.\ the
  (magnetic) eigenvalues of the standard Laplacian are always lower or
  equal than the (magnetic) eigenvalues of the combinatorial
  Laplacian.
\end{corollary}

%
\section{Geometric perturbations and preorders}  
\label{sec:geo}
%

In this section, we present several elementary and composite
perturbations of \emph{finite} \GAM-graphs (deleting edges,
contracting vertices, etc.) and study systematically their effect on
the spectrum of the magnetic Laplacian. We will apply the geometric
and spectral preorders to quantify the effect of the
perturbations. The results are stated for general weights. We treat
the important special cases of combinatorial and standard weights as
corollaries.

\subsection{Elementary perturbations}  
\label{subsec:elementary}

We consider first two elementary perturbations on graphs: deleting an
edge and contracting two vertices.
\subsubsection{Deleting an edge}  
\label{subsec:delete-edge}
Let $\W=(\G,\alpha,\m)$ be \aGAM-graph with $\G=(V,E,\bd)$.
\emph{Deleting an edge $e_0 \in E$ of $\W$} gives the \GAM-graph
$\W'=(\G',\alpha',\m')$ where $\G'=\G-e_0$ and
$\alpha'=\alpha\restr{E \setminus\{e_0\}}$; we write $\W'=\W-e_0$
for the \GAM-graph obtained in this way and will specify the weight $\m'$
later on.

Recall that $\G'=(V',E',\bd')$ is obtained from $\G=(V,E,\bd)$ by
deleting $e_0 \in E$, i.e.\ $V=V'$, $E'=E \setminus \{e_0\}$ and
$\bd'=\bd \restr{E' \times E'}$ (see \Def{deleting edges} and
Figure~\ref{subfig:a}--\ref{subfig:b}).  In particular, the inclusion
$\map \iota {\G'} {\G}$ is a graph homomorphism (see
\Exenum{graph-homo}{graph-homo:3}).

The next result is a generalisation of these ideas in the previous
articles to arbitrary vector potentials and weights.  Our results also
applies to the case when a loop or a multiple edge is deleted (see
e.g.\ \Remenum{cliques}{cliques.b}).

\begin{theorem}[General weights]
  \label{thm:delete-edge}
  Let $\W,\W'\in\Ga$ with $\W'=\W-e_0$ for some $e_0\in E=E(\W)$.
  \begin{enumerate}
  \item
    \label{delete-edge.a}
    If $w'_e\leq w_e$ for all edges $e\in E \setminus \{e_0\}$ and
    $w(v)\leq w'(v)$ for all
    $v \in V \setminus \{\bd_- e_0, \bd_+ e_0 \}$, then
    \begin{enumerate}
    \item
      \label{delete-edge.a1}
      $w(v)\leq w'(v)$ for $v \in \{\bd_- e_0, \bd_+ e_0\}$ implies
      $\W' \lesse \W$ (and hence $\W' \less \W $).
    \item
      \label{delete-edge.a2}
      $w(v)-w_{e_0}\leq w'(v)$ for $v \in \{\bd_- e_0, \bd_+ e_0\}$ and
      $\varrho_\infty\leq 1$ (maximal relative weight,
      see~\eqref{eq:rho.bdd}) implies $\W' \less[1] \W $.
  		
      Moreover, if $e_0$ is a loop and $\alpha_{e_0}=\pi$ then
      $\W'\less{} \W$.
    \end{enumerate}
  \item
    \label{delete-edge.b}
    If $w_e\leq w'_e$ for all edges $e\in E \setminus \{e_0\}$ and
    $w'(v)\leq w(v)$ for all $v\in V$, then $\W \less[1] \W' $.
  	
    Moreover, if $e_0$ is a loop with $\alpha_{e_0}=0$ then
    $\W \less{} \W'$.
  	
  \item
    \label{delete-edge.c}
    If $w_e=w'_e$ for all edges $e\in E \setminus \{e_0\}$,
    $w'(v)=w(v)$ for all $v\in V$ and if $e_0$ is a loop with
    $\alpha_{e_0}=0$ then $\W'\less \W \less \W'$, i.e.\ $\W$ and
    $\W'$ are isospectral.
  \end{enumerate}
\end{theorem}
\begin{proof}
  Let $\W=(\G,\alpha,\m)$ and $\W=(\G',\alpha',m')$ be two \GAM-graph
  with $\G'=\G-e_0$, $\alpha'=\alpha\restr{E(\G')}$ and note that
  $\card{\W}=\card{\W'}$.  \itemref{delete-edge.a1}~To show
  $\W'\lesse\W$ just observe that the inclusion
  $\iota\colon \W'\rightarrow \W$ is \aGAM-homomorphism, hence
  $\W' \lesse \W$ and therefore $\W'\less \W$ by \Thm{homomorphism}.
	
  \itemref{delete-edge.a2}~For the relation $\W'\less[1]\W$, we have
  (using \Thm{courant} twice)
  \begin{align*}
    \lambda_{k}(\W')
    &=\max_{S \in \mathcal S_{k-1}} \min_{\substack{\phi\perp S \\ \phi\neq 0}} 
    \frac{\normsqr[\lsqr {E',w'}] {d_{\alpha'} \phi}}
    {\normsqr[\lsqr{V',w'}] \phi}\\
    &\le \max_{S \in \mathcal S_{k-1}} \min_{\substack{\phi\perp S \\ \phi\neq 0}} 
    \frac{\normsqr[\lsqr {E,w}] {d_\alpha \phi}-|{(d_\alpha \phi)_{e_0}}|^2w_{e_0}}
    {\normsqr[\lsqr{V,w}] \phi - \left(\abssqr{\phi(\bd_-e_0)} +
    \abssqr{\phi(\bd_+e_0)}\right) w_{e_0}}\\
    &\le \max_{S \in \mathcal S_{k-1}} \min_{\substack{\phi\perp S \cup L'(e_0) \\ \phi\neq 0}} 
    \frac{\normsqr[\lsqr {E,w}] {d_\alpha \phi}-\abssqr{(d_\alpha \phi)_{e_0}}w_{e_0}}
    {\normsqr[\lsqr{V,w}] \phi - \left( \abssqr{\phi(\bd_-e_0)} +
    \abssqr{\phi(\bd_+e_0)}\right)w_{e_0} }\\              
    &= \max_{S \in \mathcal S_{k-1}} \min_{\substack{\phi\perp S \cup L'(e_0) \\ \phi\neq 0}} 
    \frac{\normsqr[\lsqr E] {d_\alpha \phi}-4\abssqr{\phi(\bd_+e_0)}w_{e_0}}
    {\normsqr[\lsqr{V,w}] \phi - 2 \abssqr{\phi(\bd_+e_0)}w_{e_0}}\\
    &\le \max_{S \in \mathcal S_{k-1}} \min_{\substack{\phi\perp S \cup L'(e_0) \\ \phi\neq 0}} 
    \frac{\normsqr[\lsqr {E,w}] {d_\alpha \phi}}
    {\normsqr[\lsqr{V,w}] \phi}\\
    &\le \max_{S \in \mathcal S_{k}} \min_{\substack{\phi\perp S \\ \phi\neq 0}} 
    \frac{\normsqr[\lsqr {E,w}] {d_\alpha \phi}}
    {\normsqr[\lsqr{V,w}] \phi}
    =\lambda_{k+1}(\W),
  \end{align*}
  for $k=1,\dots, n-1$, where $L'(e_0)=\C \psi'$ denotes the linear
  space generated by
  \begin{equation}
    \label{eq:def.psi'}
    \psi' 
    = \frac1{(w(\bd_+{e_0}))^{1/2}}\delta_{\bd_+{e_0}}
    +\e^{\im\alpha_{e_0}} \frac 1{(w(\bd_-{e_0}))^{1/2}}
    \delta_{\bd_-{e_0}}.
  \end{equation}
  We use the vertex weight inequality $w'(v) \ge w(v)-w_{e_0}$
  in the second line; we use the fact that $\phi \perp L'(e_0)$
  implies
  $\abssqr{(d_\alpha \phi)_{e_0}}=4\abssqr{\phi(\bd_+e_0)}$ and
  $\abssqr{\phi(\bd_-e_0)}=\abssqr{\phi(\bd_+e_0)}$ in the
  fourth line; and for the fifth line, we use the following
  inequality between real numbers $a$, $b$ and $\gamma$
  (see~\cite[Lemma~2.9]{chen:04}), namely
  \begin{equation*}
    a^2-2\gamma^2\geq 0, 
    \qquad b^2-\gamma^2>0 \qquadtext{and}
    \frac{a^2}{b^2}\leq 2
    \qquadtext{implies}
    \frac{a^2-2\gamma^2}{b^2-\gamma^2}\leq \frac{a^2}{b^2}.
  \end{equation*}
  We also used the fact that $\rho_\infty\leq 1$ as $a^2/b^2$ is
  the Rayleigh quotient for the graph $\W$ and hence $a^2/b^2 \le 2$.
  
	The proof of the second part of \itemref{delete-edge.a2} is
        similar to the previous.  We observe that
        $(d_\alpha \phi)_{e_0}=0$ if $\alpha_{e_0}=\pi$ for a loop
        $e_0$, and $\psi'=0$ in~\Eq{eq:def.psi'}.  In particular, we
        do not have to introduce the function $\psi'$, hence
        $\lambda_k(\W') \le \lambda_k(\W)$.
	
	\itemref{delete-edge.b}~Using \Thm{courant} twice we
	obtain
	\begin{align*}
	\lambda_{k+1}(\W')
	&=\min_{S \in \mathcal S_{n-(k+1)}} \max_{\substack{\phi\perp S \\ \phi\neq 0}} 
	\frac{\normsqr[\lsqr{E',w'}] {d_{\alpha'} \phi}}
	{\normsqr[\lsqr{V',w'}] \phi}\\
	&\geq \min_{S \in \mathcal S_{n-(k+1)}}
	\max_{\substack{\phi\perp S \cup L(e_0) \\ \phi\neq 0}}
	\frac{\normsqr[\lsqr {E',w'}] {d_{\alpha'} \phi}}
	{\normsqr[\lsqr{V', w'}] \phi}\\
	&= \min_{S \in \mathcal S_{n-(k+1)}}
	\max_{\substack{\phi\perp S \cup L(e_0) \\ \phi\neq 0}}
	\frac{\normsqr[\lsqr {E,w}] {d_\alpha \phi}}
	{\normsqr[\lsqr{V,w}] \phi}\\
	&\geq \min_{S \in \mathcal S_{n-k}} \max_{\substack{\phi\perp S \\ \phi\neq 0}}
	\frac{\normsqr[\lsqr {E,w}] {d_\alpha \phi}}
	{\normsqr[\lsqr{V,w}] \phi}
	=\lambda_{k}(\W)\;,
	\end{align*}
	for $k =1,\dots, n-1$, where $L(e_0)=\C \psi$ denotes the linear
	space generated by
	\begin{equation}
	\label{eq:def.psi}
	\psi 
	= \frac1{(w(\bd_+{e_0}))^{1/2}}\delta_{\bd_+{e_0}}
	-\e^{\im\alpha_{e_0}} \frac 1{(w(\bd_-{e_0}))^{1/2}}\delta_{\bd_-{e_0}}
	\end{equation}
	for the canonical orthonormal basis $(\delta_v)_v$ of
	$\lsqr{V,w}$; and where we used the fact that
	$(d_\alpha \phi)_{e_0}=0$ if $\phi \perp L(e_0)$, hence we can just
	take the norm over $E'$ instead of $E$ for the second equality.
	
	The proof of the second part \itemref{delete-edge.b} is
        similar to the previous, we observe that
        $(d_\alpha \phi)_{e_0}=0$ if $\alpha_{e_0}=0$ for a loop
        $e_0$, and $\psi=0$ in~\Eq{eq:def.psi}.  In particular, we do
        not have to introduce the function $\psi$, hence
        $\lambda_k(\W') \le \lambda_k(\W)$.
	
	\itemref{delete-edge.c}~From part \itemref{delete-edge.a1} we
        conclude that $\W'\less\W$ and from part
        \itemref{delete-edge.b} $\W\less\W'$ follows; finally, observe
        that $G$ and $G'$ have the same number of vertices; hence
        $\W$ and $\W'$ are isospectral.
\end{proof}

The above theorem generalises some known interlacing results,
namely~\cite[Lemma~2]{heuvel:95} (combinatorial Laplacian and its
signless version, see also~\cite[Theorem~3.2]{mohar:91}
and~\cite[Corollary~3.2]{fiedler:73}) and~\cite[Theorem~2.3]{chen:04}
(standard Laplacian) and~\cite[Theorem~8]{atay-tuncel:14} (signed
standard Laplacians).

We state these cases now for standard and combinatorial weights as a
corollary:
\begin{corollary}
  \label{cor:delete-edge}
  Let $\W,\W'\in\Ga$ where $\W'=\W-e_0$ for some $e_0\in E(\W)$.
  \begin{enumerate}
  \item
    \label{cor.delete-edge.b}
    If $\W,\W'\in\Co$, then $\W \less[1] \W'$ and $\W' \lesse \W$,
    hence $\W \less[1] \W' \less \W$.  Furthermore, if $e_0$ is not a
    loop, then there exists $1\leq k\leq\card{ \W}$ such that
    $\lambda_k(\W')<\lambda_k(\W)$.
  \item
    \label{cor.delete-edge.a}
    If $\W,\W'\in\De$, then $\W \less[1] \W' \less[1] \W$.
  \end{enumerate}
\end{corollary}

\begin{proof}
  \itemref{cor.delete-edge.b}~By \Thmenum{delete-edge}{delete-edge.a}
  we have $\W' \lesse \W $ (and hence $\W' \less \W $) and
  by~\Thmenum{delete-edge}{delete-edge.b} we conclude
  $\W \less[1]\W'$. For the second part, note that
  \begin{equation*}
    \sum_{k=1}^n \lambda_k(\W)
    = \tr(\Delta_\alpha)
    = \sum_{v\in V} \deg^{G}(v)
    >\sum_{v\in V'} \deg^{G'}(v)
    = \tr(\Delta_{\alpha'}) 
    = \sum_{k=1}^n \lambda_k(\W'),
  \end{equation*} 

  hence there exists an index $k \in \{1,\dots,n\}$ such that
  $\lambda_k(\W')<\lambda_k(\W)$. 

  \itemref{cor.delete-edge.a}~For $\W \less[1] \W'$ we use
  \Thmenum{delete-edge}{delete-edge.b}.  Moreover,
  $\deg^{G'}(v)=\deg^{G}(v)-1$ for $v=\bd_\pm {e_0}$ and
  $\deg^{G'}(v)=\deg^{G}(v)$ for all other vertices, hence
  by~\Thmenum{delete-edge}{delete-edge.a2} it follows that
  $\W' \less[1]\W$.
\end{proof}

\begin{remark}
  \indent
  \begin{enumerate}
  \item Note that part \itemref{cor.delete-edge.b} of the preceding
    corollary is sharp in the sense that one cannot lower the shift to
    the value $0$ (except in the trivial case when we delete a loop
    edge without magnetic potential, see
    \Corenum{delete-edge-loop}{cor.delete-edge-loop.b'}).  For example
    $\W \less \W'$ is false for combinatorial weights in \Ex{deledge}.
    Similarly, one can find counterexamples with standard weights,
    showing that one can not lower the values of the shifts in part
    \itemref{cor.delete-edge.a} either (see e.g.\ the graph presented
    in~Fig.~1 of~\cite{atay-tuncel:14}).
  
  \item The number $t(\G)$ of spanning trees of a combinatorial graph
    can be computed in terms of the spectrum of the Laplacian (without
    magnetic potential) by the formula
    \begin{equation*}
      t(\G)
      =\frac{1}{|G|}\prod_{i=2}^{|G|}\lambda_i(\W)
    \end{equation*}
    using the matrix-tree theorem (see, e.g.~\cite[Corollary~4.2]{mohar:91}
    and references therein or~\cite[Theorem~4.11]{bap:10}).  If $\W'$
    is obtained from $\W$ by edge deletion it is immediate that
    \begin{equation*}
      \W\less\W' \implies  t(\W)\leq t(\W').
    \end{equation*}
  \end{enumerate}
\end{remark}
Another simple consequence for spanning subgraphs is given in
\Cor{spanning.subgraph}.

If we delete a loop, we can slightly improve the previous corollary:
\begin{corollary}
  \label{cor:delete-edge-loop}
  Let $\W,\W'\in\Ga$ with $\W'=\W-e_0$ for a loop $e_0\in E(\W)$.
  \begin{enumerate}
  \item
    \label{cor.delete-edge-loop.b'}
    If $\W,\W'\in\Co$, and $\alpha_{e_0}=0$ then
    $\W'\less \W \less \W'$, i.e.\ $\W$ and $\W'$ are isospectral.
  \item
    \label{cor.delete-edge-loop.a'}
    If $\W,\W'\in\De$ and $\alpha_{e_0}=0$ then $\W'\less \W$ and if
    $\alpha_{e_0}=\pi$ then $\W \less \W'$.
  \end{enumerate}
\end{corollary}
\begin{proof}
  \itemref{cor.delete-edge-loop.b'}~follows
  from \Thmenum{delete-edge}{delete-edge.c}.
  \itemref{cor.delete-edge-loop.a'} If $\alpha_{e_0}=0$ it follows from
  \Thmenum{delete-edge}{delete-edge.b} that $\W'\less \W$; and if
  $\alpha_{e_0}=\pi$ it follows from
  \Thmenum{delete-edge}{delete-edge.a2} that $\W \less \W'$.
\end{proof}

\begin{example}
  \label{ex:deledge}
  For $t\in[0,2\pi]$ we consider the \GAM-graph $\W_t\in\Co^t$ defined
  by $\G$ in Figure~\ref{subfig:a}.  We orient the edges along the
  closed path such that the flux through it adds up to $t$.  The
  spectrum $\sigma(\W_t)$ consists of five eigenvalues plotted as a
  solid line in Figure~\ref{subfig:c} (the spectrum depends of the
  value $t$).  Let $\W'_t=\W - e_0$ with combinatorial weights, i.e.\
  $\W'_t \in\Co^t$ (see Figure~\ref{subfig:b}).  Since $\G'$ is a
  tree, we have $\sigma(\W'_t)=\sigma(\W'_0)$ for all $t$.  In
  particular, $\sigma(\W'_0)$ consists of five eigenvalues (dotted
  lines in Figure~\ref{subfig:c}).  From
  \Corenum{delete-edge}{cor.delete-edge.b} we conclude
  $\W_t \less[1] \W'_0 \less \W_t$.  In particular
  \begin{equation*}
    \sigma(\W'_0)
    =
    \left(
      0, 
      \frac{1}{2} \left(3-\sqrt{5}\right),
      \frac{1}{2} \left(5-\sqrt{5}\right), 
      \frac{1}{2} \left(\sqrt{5}+3\right),
      \frac{1}{2} \left(\sqrt{5}+5\right) \\
    \right)
    \approx
    (0, 0.381966, 1.38197, 2.61803, 3.61803),
  \end{equation*} 
  hence we can localise the spectrum of $\sigma(\W'_t)$ for
  any $t\in[0,2\pi]$, i.e.\ 
  $\lambda_i(\W_t) \in [\lambda_{i}(\W'), \lambda_{i+1}(\W')]$ for
  $i=1,2,3 \text{ and } 4$.
  \begin{figure}[h]\label{fig:deledge}
    \subfloat[\label{subfig:a}]{ 
      \begin{tikzpicture}[baseline,
        vertex/.style={circle,draw,fill, thick, inner sep=0pt,minimum
          size=1mm},scale=.5]
        
        \node (1) at (0,3) [vertex,label=below:] {};
        \node (2) at (2,3) [vertex,label=below:] {};
        \node (3) at (4,5) [vertex,label=left:] {};
        \node (4) at (4,1) [vertex,label=below:] {};
        \node (5) at (6,3) [vertex,label=left:] {};
        
        \draw[-](1) edge node[below] {} (2);
        \draw[-](2) edge node[left] {$e_0$} (3);
        \draw[-](2) edge node[below] {} (4);
        \draw[-](3) edge node[right] {$e_1$} (4);
        \draw[-](3) edge node[below] {} (5);
        
      \end{tikzpicture} } 
    \subfloat[\label{subfig:b}]%
    { 
      \begin{tikzpicture}[baseline,
        vertex/.style={circle,draw,fill, thick, inner sep=0pt,minimum
          size=1mm},scale=.5]
        
        \node (1) at (0,3) [vertex,label=below:] {};
        \node (2) at (2,3) [vertex,label=below:] {};
        \node (3) at (4,5) [vertex,label=left:] {};
        \node (4) at (4,1) [vertex,label=below:] {};
        \node (5) at (6,3) [vertex,label=left:] {};
        
        \draw[-](1) edge node[below] {} (2);
        \draw[-](2) edge node[below] {} (4);
        \draw[-](3) edge node[right] {$e_1$} (4);
        \draw[-](3) edge node[below] {} (5);
        
      \end{tikzpicture}   }
    \subfloat[\label{subfig:c}]
    {\includegraphics[width=8cm,height=5cm]{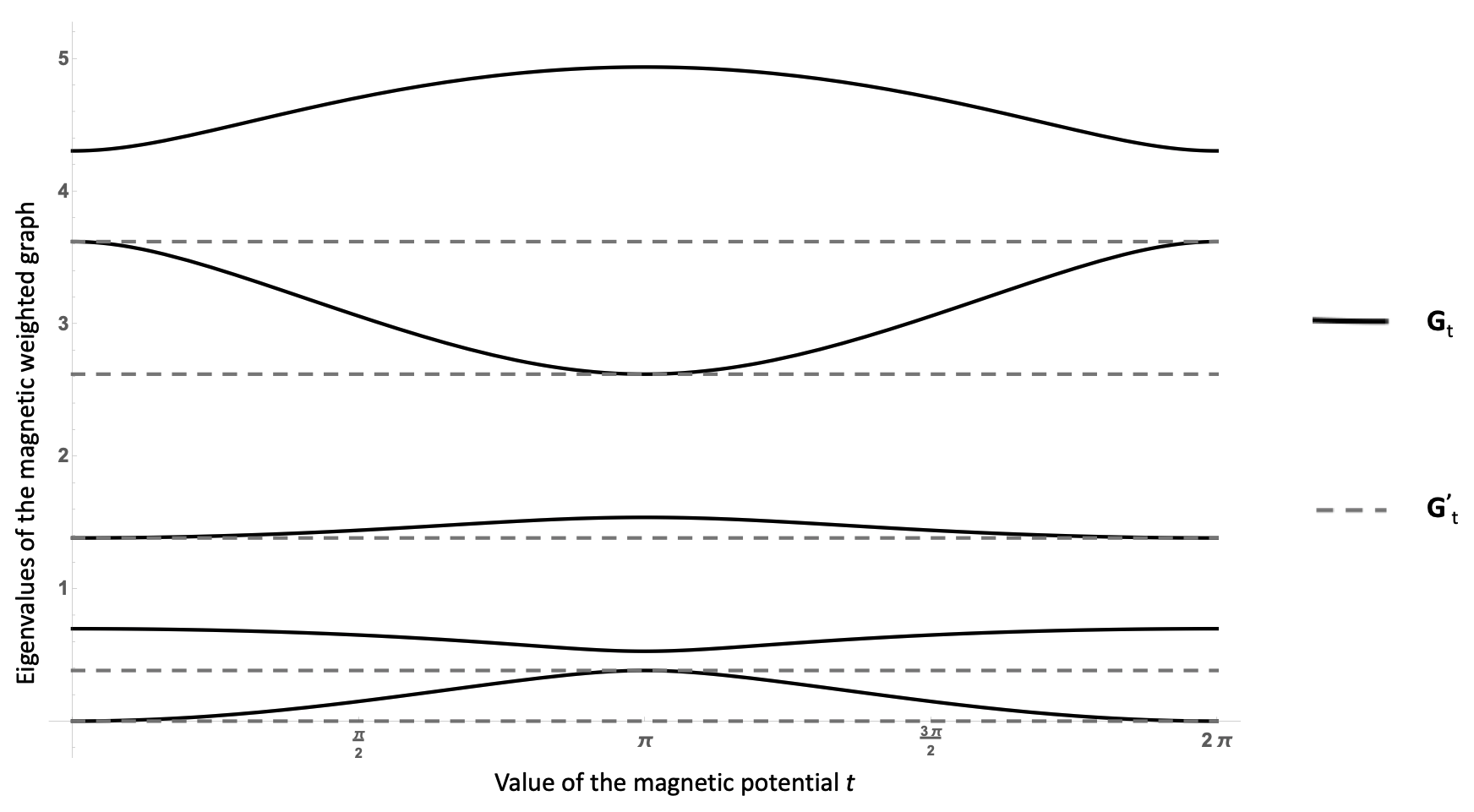}}
    \caption{If we delete the edge $e_0$ from the graph $\G$ in
      Figure~\ref{subfig:a}, we obtain the graph $\G'=\G - e_0$ in
      Figure~\ref{subfig:b}.  Let $\W_t$ (respectively $\W_t'$) be in
      $\Co^t$ with underlying graphs $\G$ (respectively, $\G'$).  In
      Figure~\ref{subfig:c} we plot $\sigma(\W_t)$ (respectively,
      $\sigma(\W_t')$) as a solid (respectively, dashed) line for all
      $t\in[0,2\pi]$. Note that $\W_t\less[1]\W'_t\less \W_t$, i.e.,
      the eigenvalues interlace.}
    \label{fig:1}
  \end{figure}
\end{example}

\subsubsection{Contracting vertices}
\label{subsec:vert-contra}
Let $\W=(\G,\alpha,\m)$ be \aGAM-graph, a \emph{vertex contraction}
of $\W$ is the \GAM-graph $\wt \W=(\wt\G,\wt\alpha,\wt\m)$ where
$\wt \G=\G/ \{v_1,v_2\}$ for two different vertices $v_1,v_2\in V(\G)$
(see \Def{glueing-vertices} and \Fig{vercon}); we specify the weight
$\wt \m$ later.  Recall that $\wt \G$ is obtained from $\G$ by contracting
the vertices $v_1$ and $v_2$ to one vertex
$\wt v_0=[v_1]=[v_2]=\{v_1,v_2\}$ while keeping all edges and vector
potentials.  Then the quotient map $\map \kappa {\G}{\wt \G}$ is
a graph homomorphism and preserves the magnetic potential (see
\Exenum{gam-homo}{ex:gam-homo2}).  We also write
$\wt \W=\W / \{v_1,v_2\}$.  We would like to stress that contracting two
\emph{adjacent} vertices $v_1,v_2$ turns any edge in $E(v_1,v_2)$ into
a loop in $\G/\{v_1,v_2\}$ (see also \Rem{vx-contr-convention} for
further cases).

\begin{figure}[h]
\centering
  \subfloat[{}\label{subfig:1}]{     \begin{tikzpicture}[baseline, vertex/.style={circle,draw,fill, thick,
                         inner sep=0pt,minimum size=1mm},scale=.5]
                        
             \node (1) at (0,3) [vertex,label=below:$v_0$] {};
             \node (2) at (2,3) [vertex,fill=red,label=below:{\color{red}{$v_2$}}] {};
             \node (3) at (4,5) [vertex,label=left:] {};
             \node (4) at (4,1) [vertex,label=below:] {};
             \node (5) at (6,3) [vertex,label=left:] {};
             \node (7) at (2,1) [vertex,fill=red,label=left:{\color{red}{$v_1$}}] {};
                
       	    \draw[-](1) edge node[below] {} (2);
       	  	\draw[-](2) edge node[left] {} (3);
       	  	\draw[-](7) edge node[below] {} (4);
       	  	\draw[-](3) edge node[right] {$e_0$} (4);
       	  	\draw[-](3) edge node[below] {} (5);
       	  	\draw[-](3) edge node[below] {} (5);
         \end{tikzpicture}\qquad }
       \subfloat[{}\label{subfig:2}]{
         \begin{tikzpicture}[baseline,
           vertex/.style={circle,draw,fill, thick, inner
             sep=0pt,minimum size=1mm},scale=.5]
           
           \node (1) at (0,3) [vertex,label=below:] {};
           \node (2) at (2,3) [vertex,fill=red,label=below:{\color{red}{$[v_1]$}}] {};
           \node (3) at (4,5) [vertex,label=left:] {};
           \node (4) at (4,1) [vertex,label=below:] {};
           \node (5) at (6,5) [vertex,label=left:] {};
           
           \draw[-](1) edge node[below] {} (2);
           \draw[-](2) edge node[left] {} (3);
           \draw[-](2) edge node[left] {} (4);
           \draw[-](3) edge node[right] {$e_0$} (4);
           \draw[-](3) edge node[below] {} (5);
           
           \end{tikzpicture}  }  
         \subfloat[{}\label{subfig:3}]{ \includegraphics[width=8cm,height=5cm]{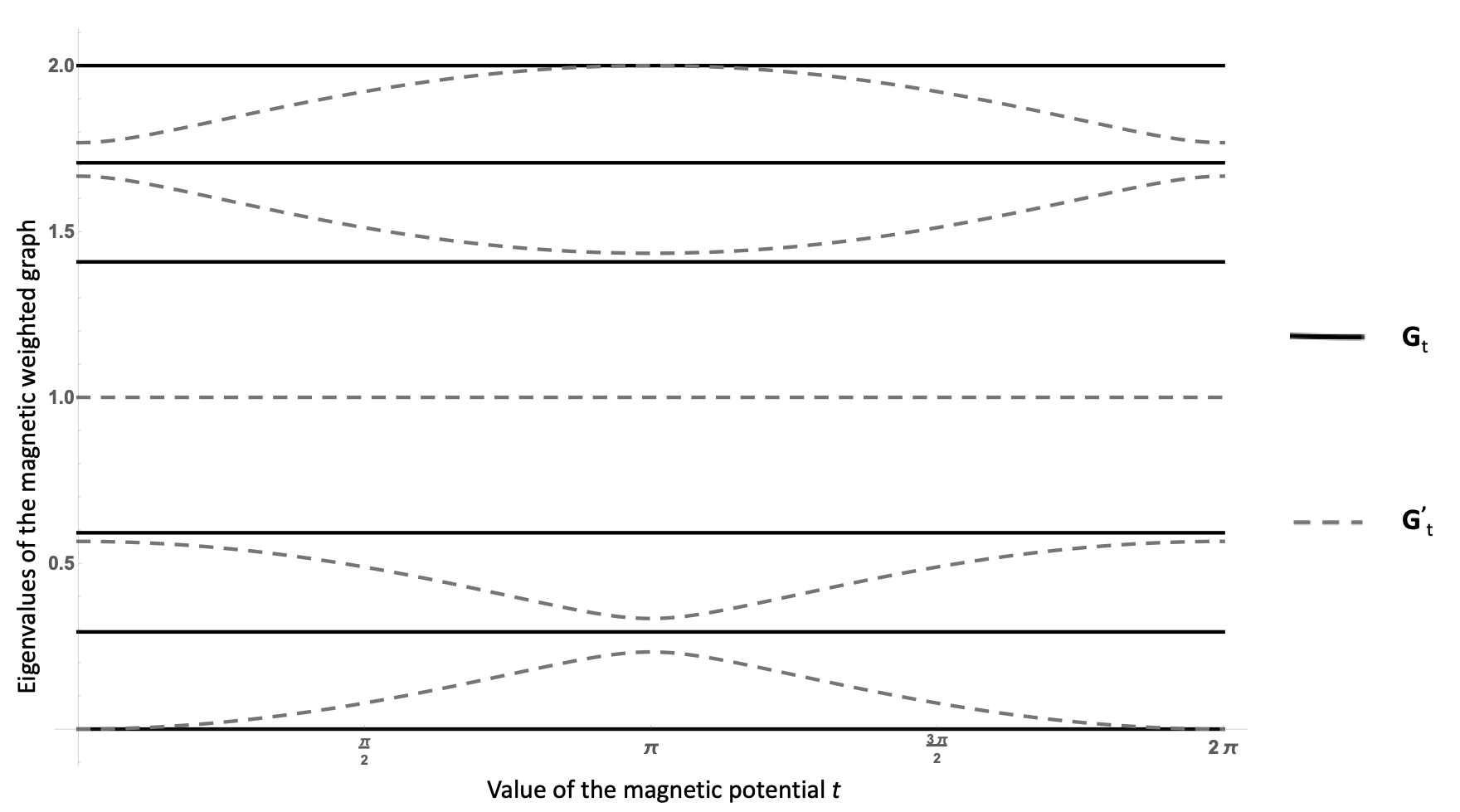}}
         \caption{Contracting the vertices $v_1$ and $v_2$ of the
           graph $\G$ in~\ref{subfig:1} gives the graph
           $\wt\G=\G/\{v_1,v_2\}$ in Figure~\ref{subfig:2}. Let
           $\W_t,\wt \W_t\in\Co^t$ defined by $\G$ (respectively, $\wt \G$), then in
           Figure~\ref{subfig:3} we plot as dashed lines
           $\sigma(\W_t)$ and as solid lines $\sigma(\wt \W_t)$ for
           $t\in[0,2\pi]$.}
\label{fig:vercon} 
\end{figure}
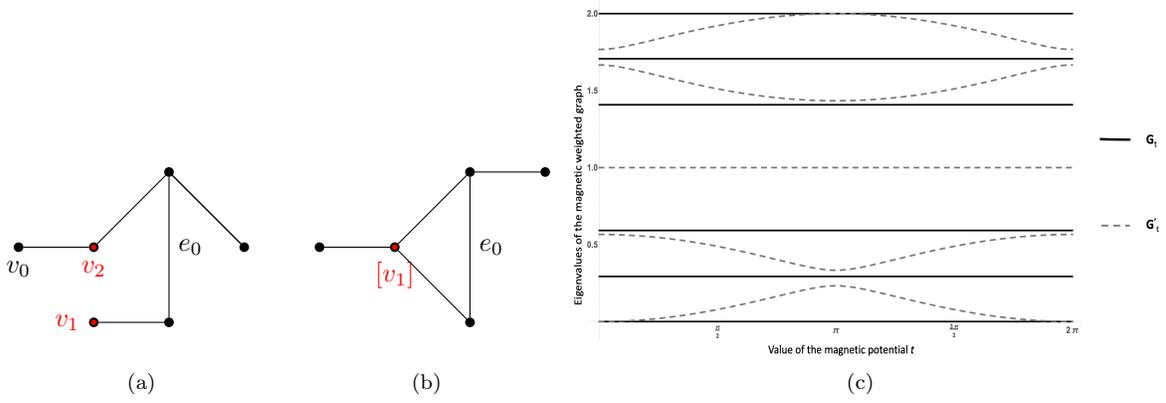

\begin{theorem}
  \label{thm:vert-contra}
  Let $\W,\wt \W\in\Ga$ with $\wt \W=\W/ \{v_1,v_2\}$.
  \begin{enumerate}
  \item
    \label{vert-contra.a}
    If $w_e\leq \wt w_e$ for all $e\in E(G)$ and $\wt w([v])\leq w(v)$
    for all $v \in V(G) \setminus \{v_1,v_2\}$ and
    $\wt w([v_1])\leq w(v_1)+w(v_2)$, then $\W \lesse \wt \W $ and
    $\wt \W \less[r+1-s] \W$ (and hence
    $\W \less \wt \W\less[r+1-s] \W$) where
    \begin{equation}
      \label{eq:def.r.s}
      r=\min\{\deg^{G}(v_1),\deg^{G}(v_2)\}
      \qquadtext{and}
      s=\card{\set{e\in E^G(v_1,v_2)}{\alpha_e =0}}
    \end{equation}
    is the minimal degree and $s$ the number of (unoriented) edges
    joining $v_1$ and $v_2$ having no magnetic potential.  In
    particular, if $v_1,v_2$ are not adjacent, then $s=0$.
  \item
    \label{vert-contra.b}
    If $w_e= \wt w_e$ for all $e\in E(G)$, and $\wt w(v)= w(v)$ for all
    $v \in V(G) \setminus \{v_1,v_2\}$ and $\wt w([v_1])= w(v_1)+w(v_2)$,
    then $\W \lesse \wt \W $ and $\wt \W \less[1] \W$ (and hence
    $\W \less \wt \W \less[1] \W $).
  \end{enumerate}
\end{theorem}
\begin{proof}
  \itemref{vert-contra.a} By our assumption and by definition of
  \aGAM-homomorphism, $\map{\kappa}{\W}{\wt \W}$ is \aGAM-homomorphism,
  i.e., $\W \lesse \wt \W$, and $\W \less\wt \W$ follows from
  \Thm{homomorphism}.

  Suppose $r=\deg(v_1)$ is the minimal degree of $v_1$ and $v_2$ and
  that $v_1$, $v_2$ are not adjacent ($E(v_1,v_2)=\emptyset$).  We now
  prove $\wt \W \less[r+1] \W$ by consecutively deleting the $r$ edges
  of $E^\G_{v_1}$.  Let $\W'=(\G',\alpha',\m')$ with
  $\G'=\G-E^\G_{v_1}$, and let $\alpha'$ and $\m'$ be the restrictions
  of the corresponding magnetic potential and weights on $\G$ onto
  $\G'$. Similarly, we define $\wt \W'$ with underlying graph
  $\wt \G'=\wt G-E^\G_{v_1}$ (recall that the edge sets of $\G$ and
  $\wt G$ are the same).  From \Thmenum{delete-edge}{delete-edge.b}
  applied $r$ times and the transitivity in \Lemenum{basic}{basic:3}
  we conclude $\wt\W \less[r] \wt \W'$.  Moreover, from
  \Thmenum{delete-edge}{delete-edge.a} applied $r$ times and the
  transitivity in \Lemenum{basic}{basic:3} we conclude $\W' \less \W$.
  Since the order of vertex contraction and edge deletion does not
  matter, we have $\wt G-E^\G_{v_1}= (\G'-E^\G_{v_1})/\{v_1,v_2\}$, so
  that $\W'$ is $\wt \W'$ together with $v_1$ as an isolated vertex.  In
  particular, the spectrum of $\W'$ is just the one of $\wt \W'$ with
  an extra $0$, and therefore $\wt \W' \less[1] \W'$.  The result then
  follows again by transitivity.

  If $v_1,v_2$ are adjacent, then each edge $e \in E(v_1,v_2)$ with
  $\alpha_e=0$ turns into a loop in $\wt G$, hence the spectral shift
  is $0$ for each such edge (\Thmenum{delete-edge}{delete-edge.b}).
  
    \itemref{vert-contra.b} Here, the \GAM-homomorphism
  $\map{\kappa}{\W}{\wt \W}$ is measure preserving, hence we conclude
  $\W \less \wt \W$ and $\wt \W \less[1] \W$ both from \Thm{homomorphism}
  with $r=\card{V(\W)}-\card{V(\wt \W)}=1$.
\end{proof}

Similarly,~\cite[Theorem~2.7]{chen:04} (and again generalised to the case
of signed graphs in~\cite[Theorem~10]{atay-tuncel:14}) prove a weaker
version of our vertex contraction for the standard Laplacian, namely
$\W \less[1] \wt \W \less[1] \W$ in our notation, under the additional
assumption that the vertices $v_1,v_2$ have combinatorial distance at
least $3$.  The latter restriction is mainly due to the fact that both
papers avoid the use of multigraphs, namely multiple edges and loops.

As corollary we restrict the theorem to the case of combinatorial and
standard weights.  Note that our result improves in
particular~\cite[Theorem~2.7]{chen:04} (standard Laplacian)
and~\cite[Theorem~10]{atay-tuncel:14} (signed standard Laplacians: in
both articles, only $\W \less[1] \wt \W \less[1] \W$ is proven (in our
notation) for vertices $v_1,v_2$ with combinatorial distance at least
$3$.  Our corollary does not need this restriction and gives a better
shift (in the standard case):
\begin{corollary}
  \label{cor:vert-contra}
  Let $\W,\wt \W\in\Ga$ with $\wt \W=\W/ \{v_1,v_2\}$.
  \begin{enumerate}
  \item
    \label{cor.vert-contra.b}
    If $\W,\wt \W\in\Co$, then $\W\lesse \wt \W$ and
    $\wt \W\less[r+1-s] \W$, hence $\W \less \wt \W \less[r+1-s] \W$,
    where $r$ and $s$ are defined in~\eqref{eq:def.r.s}.
  \item
    \label{cor.vert-contra.a}
    If $\W,\wt \W\in\De$, then $\W\lesse \wt \W$ and $\wt \W \less[1] \W$,
    hence $\W\less \wt \W \less[1] \W$.
  \end{enumerate}
\end{corollary}
\begin{proof}
  \itemref{cor.vert-contra.b} The claim follows as the combinatorial
  weights fulfil the condition in
  \Thmenum{vert-contra}{vert-contra.a}.

  \itemref{cor.vert-contra.a} The standard weights fulfil the
  condition in \Thmenum{vert-contra}{vert-contra.b} as
  $\deg^{\G'}([v_1])=\deg^\G(v_1)+\deg^\G(v_2)$, in particular we have
  $\W \lesse \wt \W$ and $\W \less \wt \W \less[1] \W $.
\end{proof}

\begin{example}
  For $t\in[0,2\pi]$, consider the \GAM-graph $\W_t\in\De^t$ with
  underlying graph $\G$ as in Figure~\ref{subfig:1}. 
  Since $\G$ is a tree, we have $\sigma(\W_t)=\sigma(\W_0)$ and $\sigma(\W_0)$
  consists of six eigenvalues (dashed lines in Figure~\ref{subfig:3}).
  Let now $\wt \W_t=\W_t/\{v_1,v_2\}$, see Figure~\ref{subfig:2}; we
  orient the edges in the cycle such that the flux adds up to
  $t$. Figure~\ref{subfig:3} shows the
  five eigenvalues of $\sigma(W'_t)$ changing $t$ from $0$ to $2\pi$.
  By \Corenum{vert-contra}{cor.vert-contra.a} we have
  $\W_t\less \wt \W_t\less[1] \W_t$ for any $t$, then
  $\W_0\less \wt \W_t\less[1] \W_0$.  In particular, we can use the
  spectrum of the tree $\W$ to localised the spectrum of $\wt \W$ for any
  vector potential, i.e.\
  \begin{align*}
    \lambda_1(\wt \W_t)\in \left[ 0, 1-\frac{1}{\sqrt{2}}\right], \quad
    \lambda_2(\wt \W_t)\in \left[1-\frac{1}{\sqrt{2}}, 1-\frac{1}{\sqrt{6}}\right],
    \quad 
    \lambda_3(\wt \W_t)\in \left[1-\frac{1}{\sqrt{6}}, 
       1+\frac{1}{\sqrt{6}}\right],\\
    \lambda_4(\wt \W_t)\in \left[  1+\frac{1}{\sqrt{6}}, 
       1+\frac{1}{\sqrt{2}}\right]
    \quadtext{and}
    \lambda_5(\wt \W_t)\in \left[ 1+\frac{1}{\sqrt{2}},2\right].
  \end{align*}
  In this example, the previous localisation of the spectrum in the
  bracketing intervals remains the same if we identify any other pair
  of distinct vertices, i.e.\ $\W_0\less \wt \W'_t\less[1] \W_0$ where
  $\wt \W'_t=\W/\{u,v\}$ for any distinct vertices $u,v\in V(\G)$.
\end{example}

\subsection{Composite perturbations}  
\label{subsec:composite}

The following perturbation of graphs are composite and can be obtained
by the operations introduced in the preceding subsection.

\subsubsection{Contracting an edge}
\label{subsec:edge-contra}
Contracting an edge (not being a loop) is just the composition of the
two operations: deleting an edge $e_0$ and contracting the adjacent
vertices (note that the order of the perturbations does not
matter). Formally, let $\W$ be \aGAM-graph, an \emph{edge
  identification} of $\W$ is the \GAM-graph $\W'$ where
$\W'=(\W-e_0)/\{\bd_+e_0,\bd_-e_0\}=\W/\{\bd_+e_0,\bd_-e_0\}-e_0$ for
some edge $e_0 \in E(\G)$ (see \Exenum{gam-homo}{graph-homo:2} and
\Fig{edgecont}).  We write this operation simply as $\W'=\W/\{\e_0\}$
(again, we specify the weight later).

\begin{figure}[h]
  \centering
  
  \subfloat[{}\label{subfig:edgecontr1}]{    
    \begin{tikzpicture}[baseline, vertex/.style={circle,draw,fill, thick,
        inner sep=0pt,minimum size=1mm},scale=.5]
      \node (-1) at (-2,5) [vertex,label=below:] {};
      \node (0) at (-2,1) [vertex,label=below:] {};                                         
      \node (1) at (0,3) [vertex,fill=red,label=below:{\color{red}{$u$}}] {};
      \node (2) at (2,3) [vertex,fill=red,label=below:{\color{red}{$v$}}] {};
      \node (3) at (4,5) [vertex,label=left:] {};
      \node (4) at (4,1) [vertex,label=right:] {};

      \draw[-](-1) edge node[below] {} (1);	  
      \draw[-](0) edge node[below] {} (1);         	                          
      \draw[-,red](1) edge node[above] {$e_0$} (2);
      \draw[-](2) edge node[below] {} (3);
      \draw[-](2) edge node[below] {} (4);
      \draw[-](3) edge node[right] {$e_1$} (4);  	  	
    \end{tikzpicture}}  
  \subfloat[{}\label{subfig:edgecontr2}]{   
    \begin{tikzpicture}[baseline, vertex/.style={circle,draw,fill, thick,
        inner sep=0pt,minimum size=1mm},scale=.5]
      \node (-1) at (0,5) [vertex,label=below:] {};
      \node (0) at (0,1) [vertex,label=below:] {};                                         
      \node (2) at (2,3) [vertex,fill=red,label=below:{\color{red}{$w$}}] {};
      \node (3) at (4,5) [vertex,label=left:] {};
      \node (4) at (4,1) [vertex,label=right:] {};

      \draw[-](-1) edge node[below] {} (2);	  
      \draw[-](0) edge node[below] {} (2);         	                          
      \draw[-](2) edge node[below] {} (3);
      \draw[-](2) edge node[below] {} (4);
      \draw[-](3) edge node[right] {$e_1$} (4);  	  	
    \end{tikzpicture}}  
  \subfloat[{}\label{subfig:edgecontr3}]{ 
    \includegraphics[width=8cm,height=5cm]{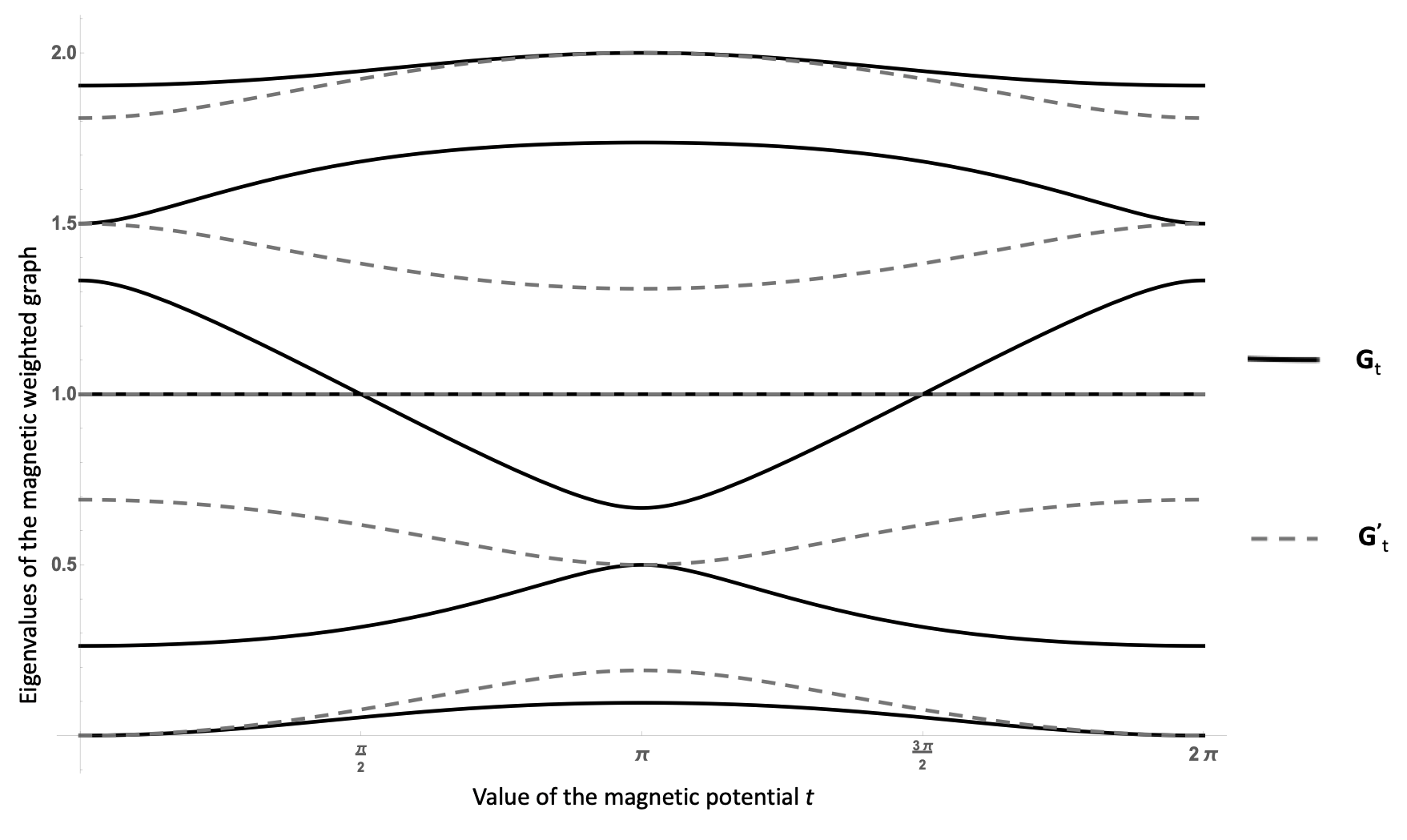}\qquad}            
  \caption{If we contract the edge $e_0$ of the graph $\G$ in
    Fig.~\ref{subfig:edgecontr1}, we obtain the graph $\G'=\G/\{e_0\}$
    as in Fig.~\ref{subfig:edgecontr2}. Let $t\in[0,2\pi]$ and let
    $\W_t,\W'_t\in\De^t$ be the corresponding magnetic weighted
    graphs, then $\sigma(\W_t)$ (respectively, $\sigma(\W'_t)$) are
    plotted in Fig.~\ref{subfig:edgecontr3} as solid (respectively,\
    dashed) lines for $t\in[0,2\pi]$.  Here, we have
    $\W_t \less \W_t' \less[1] \W_t$, and one can see this classical
    interlacing by the fact that the solid and dashed lines do not
    intersect, and solid and dashed lines alternate.  Note that the
    horizontal eigenvalue (independent of $t$) is an eigenvalue for
    both graphs.}
  \label{fig:edgecont}
\end{figure}
\begin{theorem}
  \label{thm:edge-contra}
  Let $\W,\W'\in\Ga$ with $\W'=\W/ \{e_0\}$, where $e_0 \in E(\G)$ is
  not a loop and simple (i.e., $\card {E(\bd_-e_0,\bd_+e)_0}=1$).
  \begin{enumerate}
  \item
    \label{edge-contra.b}
    If $\W,\W'\in\Co$, then $\W\less[1] \W'\less[r+1] \W$ where
    $r=\min\{\deg(\partial_+e_0),\deg(\partial_-e_0)\}$. If
    $\alpha_{e_0}=0$ or if $e_0$ is a bridge edge, then
    $\W\less \W' \less[r] \W$.
  \item
    \label{edge-contra.a}
    If $\W,\W'\in\De$ then $\W\less[1] \W' \less[2] \W$. If
    $\alpha_{e_0}=0$ then $\W\less[1] \W' \less[1] \W$.  If
    $\alpha_{e_0}=\pi$ then $\W\less \W' \less[2] \W$.  If $e_0$ is a
    bridge edge, then $\W\less \W' \less[1] \W$.
  \end{enumerate}
\end{theorem}
\begin{proof}
  \itemref{edge-contra.b}~Suppose that $\W,\W'\in\Co$.  Let
  $\wt \W=\W/\{\bd_-e_0,\bd_+e_0\} \in \Co$ be the graph with
  combinatorial weights obtained from $\W$ by contracting the two
  vertices of $e_0$.  Recall that $e_0$ becomes a loop (for simplicity
  also denoted by $e_0$) in $\wt \W$.  From
  \Corenum{vert-contra}{cor.vert-contra.b} we obtain
  $\W \less \wt \W \less[r+1] \W$.  Next we delete the loop in $\wt \W$,
  so that $\W'=\wt \W-e_0$; from
  \Corenum{delete-edge}{cor.delete-edge.b} we obtain
  $\wt\W \less[1] \W' \less \wt \W$.  Combining both arguments, we have
  \begin{equation*}
    \W \less \wt \W \less[1] \W' \less \wt \W \less[r+1] \W
  \end{equation*}
  and the transitivity in \Lemenum{basic}{basic:3} gives us the
  result.

  If $\alpha_{e_0}=0$, then we use
  \Corenum{delete-edge-loop}{cor.delete-edge-loop.b'} and obtain
  $\wt\W \less \W' \less \wt\W$ ($\wt\W$ and $\W'$ are trivially
  isospectral); the same argument as above (now with $s=1$ in
  \Corenum{vert-contra}{cor.vert-contra.b}) then gives
  $\W \less \W' \less[r] \W$.  If $e_0$ is a bridge edge, then we can
  find an equivalent vector potential $\hat \alpha$ with
  $\hat \alpha_{e_0}=0$ by \Lem{bridge}.  From \Prp{sunada.weight} we
  conclude that $\W$ and $\hat \W=(\G,\hat \alpha,\m)$ are
  isospectral, hence the above argument with $\W$ replaced by
  $\hat \W$ yields the result.

  \itemref{edge-contra.a}~Suppose that $\W,\W'\in\De$.  We follow the
  same strategy and define $\wt\W=\W /\{\bd_-e_0,\bd_+e_0\} \in \De$.
  Now \Corenum{vert-contra}{cor.vert-contra.a} implies
  $\W \less \wt\W \less[1] \W$.  Since again $\W'=\wt\W-e_0$, we
  conclude from \Corenum{delete-edge}{cor.delete-edge.a} that
  $\wt\W \less[1] \W' \less[1] \wt\W$.  From the transitivity in
  \Lemenum{basic}{basic:3} we conclude $\W\less[1] \W' \less[2] \W$.

  If $\alpha_{e_0}=0$ then
  \Corenum{delete-edge-loop}{cor.delete-edge-loop.a'} gives us
  $\W' \less \wt\W$.  If $\alpha_{e_0}=\pi$, we have $\wt\W \less \W'$.
  If $e_0$ is a bridge edge, then we can find equivalent vectors
  potentials with value $0$ or $\pi$ on $e_0$, and use the same
  argument as in the combinatorial case to conclude the better
  estimate $\W\less \W' \less[1] \W$.
\end{proof}

\begin{example}
  \label{ex:contraedge}
  For any $t\in[0,2\pi]$, consider the \GAM-graphs $\W_t\in\De^t$
  defined by the graph $\G$ in Figure~\ref{subfig:edgecontr1}; again, 
  we orient the edges along the closed path such that the flux through it adds up to $t$. 
  Then $\sigma(\W_t)$ consists of six eigenvalues that depend on the value
  of $t$.  The spectrum $\sigma(\W_t)$ is plotted as a solid line in
  Figure~\ref{subfig:edgecontr3} for all $t\in[0,2\pi]$.  If we
  consider $\W'\in\Co^t$ given by $\W'_t=\W/\{e_0\}$, i.e.\ $\W'$ is
  defined by the graph $\G'=\G/\{e_0\}$ in~\ref{subfig:edgecontr2},
  the spectrum $\sigma(\W')$ consists of five eigenvalues (dashed
  lines in Figure~\ref{subfig:edgecontr3}).  Since $e_0$ is a bridge
  edge, we see graphically the interlacing given by \Thm{edge-contra}
  (the solid and dashed line alternate); i.e.\
  $\W_t\less \W_t' \less[1] \W_t$ for any $t\in[0,2\pi]$.
\end{example}

\subsubsection{Contracting a pendant edge}
\label{subsec:delete-pendant}
We now consider a special case of contraction a \emph{pedant} edge
(see Figure~\ref{fig:pendant}).  In particular, a pendant edge is
always a bridge edge.  In particular, \Thm{edge-contra} gives us (now
with $r=1$):

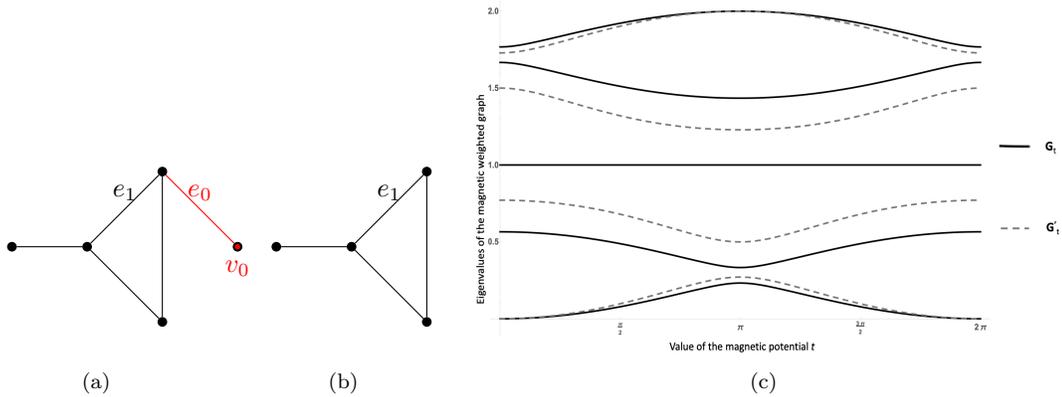
\begin{figure}[h]
  \centering \subfloat[{}\label{subfig:pendant1}]{
    \qquad \begin{tikzpicture}[baseline,
      vertex/.style={circle,draw,fill, thick, inner sep=0pt,minimum
        size=1mm},scale=.5]
      
      \node (1) at (0,3) [vertex,label=below:] {};
      \node (2) at (2,3) [vertex,label=below:] {};
      \node (3) at (4,5) [vertex,label=left:] {};
      \node (4) at (4,1) [vertex,label=below:] {};
      \node (5) at (6,3) [vertex,fill=red,label=below:{\color{red}{$v_0$}}] {};
      
      \draw[-](1) edge node[below] {} (2);
      \draw[-](2) edge node[above] {$e_1$} (3);
      \draw[-](2) edge node[below] {} (4);
      \draw[-](3) edge node[below] {} (4);
      \draw[-,red](3) edge node[above] {$e_0$} (5);
      
    \end{tikzpicture}}  
  \color{black}\subfloat[{}\label{subfig:pendant2}]{
    \begin{tikzpicture}[baseline, vertex/.style={circle,draw,fill, thick,
        inner sep=0pt,minimum size=1mm},scale=.5]	
      \node (1) at (0,3) [vertex,label=below:] {};
      \node (2) at (2,3) [vertex,label=below:] {};
      \node (3) at (4,5) [vertex,label=left:] {};
      \node (4) at (4,1) [vertex,label=below:] {};
      \draw[](1) edge node[below] {} (2);
      \draw[](2) edge node[above] {$e_1$} (3);
      \draw[](2) edge node[below] {} (4);
      \draw[](3) edge node[below] {} (4);
    \end{tikzpicture}}
  \quad
  \subfloat[{}\label{subfig:pendant3}]{
    \includegraphics[width=8cm,height=5cm]{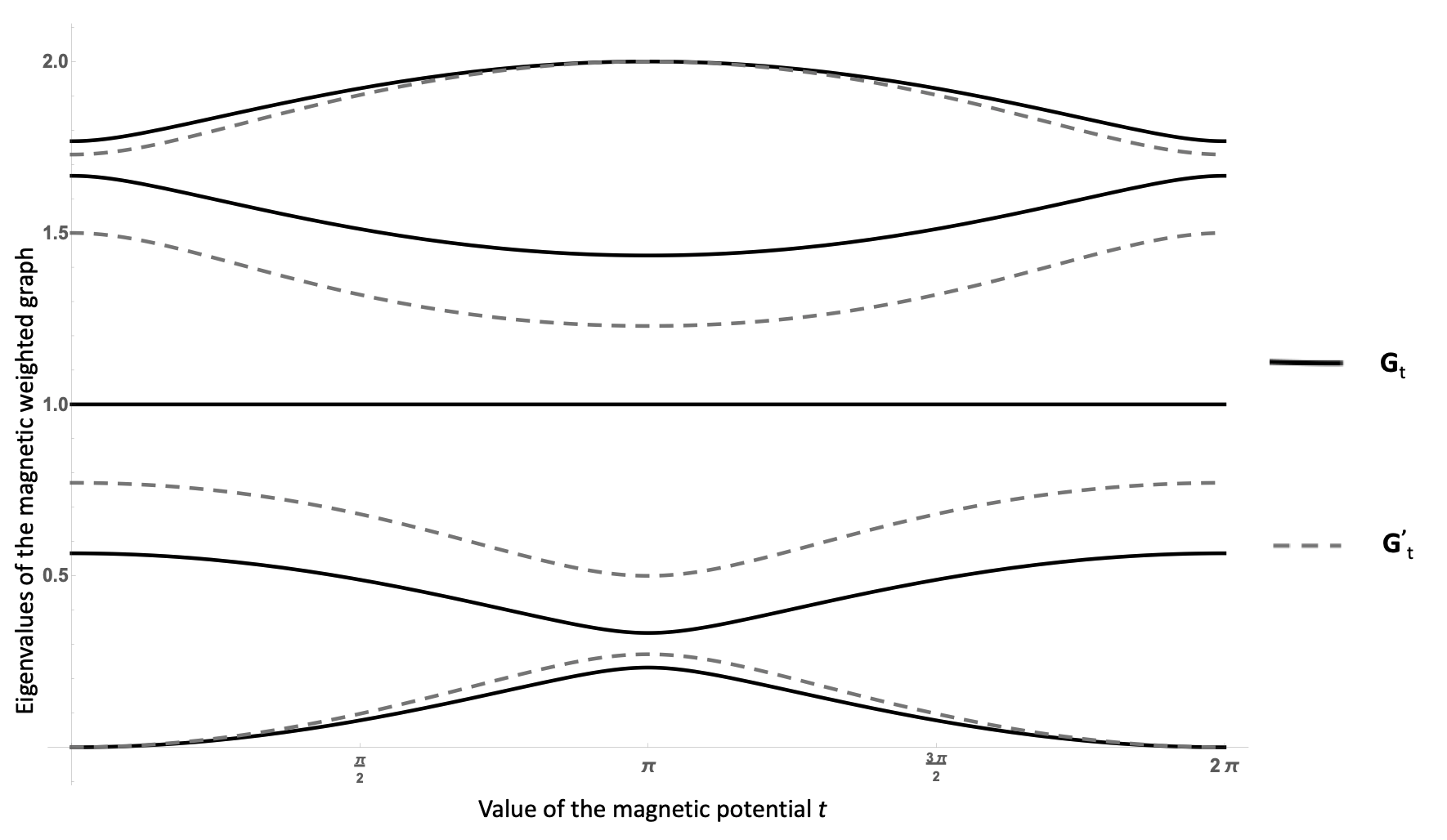}}        
  \caption{The graph $\G$ with $e_0$ a pendant edge
    (Fig.~\ref{subfig:pendant1}.  If we make the edge contraction of
    $e_0$, we obtain the graph $\G'=\G/\{e_0\}$ in
    Fig.~\ref{subfig:pendant2}. For any $t\in[0,2\pi]$, consider
    $\W_t,\W'_t\in\De^t$ defined by $\G$ (respectively, $\G'$).  In
    Fig.~\ref{subfig:pendant3}, we plot $\sigma(\W_t)$ (respectively,
    $\sigma(\W'_t)$) in solid (respectively,\ dashed) line for all
    $t\in[0,2\pi].$ }
  \label{fig:pendant}
\end{figure}

\begin{corollary}
  \label{cor:delete-pendant}
  Let $\W,\W'\in\Ga$ where $\W'=\W/\{e_0\}$ and $e_0$ is a pendant edge.
  \begin{enumerate} 
  \item
    \label{delete-pendant.b}
    If $\W,\W'\in\Co$, then $\W\less\W'\less[1] \W$.
  \item
    \label{delete-pendant.a}
    If $\W,\W'\in\De$, then $\W \less \W'\less[1] \W$. 
  \end{enumerate}
\end{corollary}

\begin{example}
  For all $t\in[0,2\pi]$, consider the \GAM-graph $\W_t\in\Co^t$
  defined by the graph $\G$ in Figure~\ref{subfig:pendant1} and choose
  the orientation of the edges along the cycle such that the flux
  through it adds up to $t$.  Define $\W'_t\in\Co^t$ where
  $\W'=\W/\{\e_0\}$, i.e.\ $\W'_t$ is defined by the graph in
  Figure~\ref{subfig:pendant2}. For each $t$, we have $\sigma(\W_t)$
  (respectively,\ $\sigma(\W'_t)$) consists of five (respectively,
  four) eigenvalues plotted as solid (respectively, dashed) lines for
  all $t\in[0,2\pi]$ in Figure~\ref{subfig:pendant3}.  From
  \Cor{delete-pendant} we conclude $\W_t\less \W_t'\less[1] \W_t$.
  This interlacing property for the combinatorial weights is only true
  for pendant edges and not for general bridge edges.
\end{example}

\subsubsection{Deleting a vertex}
\label{subsec:delete-vertex}
Let $G$ be a graph and $v_0\in V(G)$.  For simplicity, we assume that
there is no loop at $v_0$.  The graph $G'=G-v_0$ is obtained from $\G$
by \emph{deleting the vertex} $v_0$ and all its adjacent edges
$e \in E_{v_0}$, i.e.\ $V(G')=V(G)\setminus\{v_0\}$,
$E':=E(G')=E(G)\setminus E_{v_0}$ and $\bd^{G'}=\bd\restr{E'}$
(see~\cite[Section~2.1]{bondy-murty:08}).  We say that the \GAM-graph
$\W'=(\G',\alpha',w')$ is obtained from $\W=(\G,\alpha,\m)$ by
\emph{deleting} the \emph{vertex} $v_0$ (denoted by $\W'=\W-v_0$) if
$\G'=\G-v_0$ and if $\alpha'=\alpha\restr{E'}$; we specify the weight
$\m'$ in the following cases:
\begin{corollary}
  \label{cor:delete-vertex}
  Let $\W,\W'\in\Ga$ where $\W'=\W-v_0$. If $r=\deg(v_0)$, then
  \begin{enumerate}
  \item
    \label{delete-vertex.b}
    Let $\W,\W'\in\Co$, then $\W\less[r-1]\W'\less[1] \W$.
  \item
    \label{delete-vertex.a}
    Let $\W,\W'\in\De$, then $\W \less[r-1] \W'\less[r] \W$.
  \end{enumerate}
\end{corollary}
\begin{proof}
  First delete $r-1$ edges adjacent to $v_0$ and apply
  \Cor{delete-edge} in each case, i.e.\ for combinatorial and standard
  weights.  Finally, delete the pendant edge and apply the
  corresponding parts of \Cor{delete-pendant} .
\end{proof}

%
\section{Applications}
\label{sec:applications}
%
In this final section, we present a large variety of applications of
the preorder relations of \GAM-graphs studied before.  In particular,
we apply our results to study certain combinatorial aspects of graphs,
to prove how Cheeger's constant change under MW-homomorphisms and to
study the stability of eigenvalues under perturbation of graphs with
high multiplicity; we will also identify spectral gaps in the spectrum
of Laplacians on infinite covering graphs. Our results are also useful
in order to show the monotonicity of certain combinatorial numbers like,
e.g.\ the algebraic connectivity of a graph under elementary
perturbations.

\subsection{Spectral graph theory and combinatorics}
\label{sub:combinatorics}

In this subsection, we assume that all graphs are \emph{finite}.
\subsubsection{Spectral order of graphs} %
\label{subsec:spec.ord}
We begin mentioning some natural interaction between the preorder
relations $\less$ and $\lesse$ mentioned before and combinatorics.
Recall that the spectral preorder $\W \less \W'$ means that the
increasingly ordered list of eigenvalues $\lambda_k$ and $\lambda_k'$
of $\W$ and $\W'$, respectively (repeated according to their
multiplicity) fulfil $\lambda_k \le \lambda_k'$ for all indices $k$,
see \Def{spectral-order}.  Moreover, the geometric (pre)order
$\W \lesse \W'$ means that there is \aGAM-homomorphism
$\map \pi \W {\W'}$, i.e., a graph homomorphism respecting the
magnetic potential and fulfilling certain inequalities on the vertex
and edge weights, see \Def{hom.gam-graph}).

First, we apply the geometric perturbation and elementary operations
on graphs established in \Sec{geo} to present a new
spectral order of MW-graphs.  We illustrate the method for some simple graphs up
to order $6$ with combinatorial weights:
We have seen in \Prp{relation1} that for any fixed value
$t\in[0,2\pi]$ the family $\Co^t$ (see \Def{notation}) is partially
ordered with respect to $\lesse$.  In particular, the spectral
relations below include the cases of the combinatorial Laplacian (if
$t=0$) and the signless Laplacian (if $t=\pi$). In \Fig{order-graph},
we specify the spectral relations of a chain of simple graphs up to
order $6$. Note first that $\W_i\lesse\W_{i+1}$ for $1\leq i\leq 7$ is
a consequence of \Corenum{delete-edge}{delete-edge.b} and the fact any
two consecutive graphs from $\W_1$,\dots,$\W_8$ differ by an
edge. Moreover, $\W_8\lesse\W_{9}$ follows from
\Corenum{vert-contra}{vert-contra.b} since the graph $\W_9$ is
obtained from $\W_8$ by contracting the upper right vertex with the
lower right vertex.  Recall also that by \Thm{homomorphism} we
directly obtain also the relation $\W_i\less\W_{i+1}$,
$1\leq i\leq 8$.  Finally, note that $\W_9\lesse\W_{10}$ is false by
\Prp{inj-sur} since \aGAM-homomorphism $\W_9$ and $\W_{10}$ is
injective on the edges.  \Cor{delete-pendant} gives the relation
$\W_9\less\W_{10}$ because both graphs differ by a pendant edge.
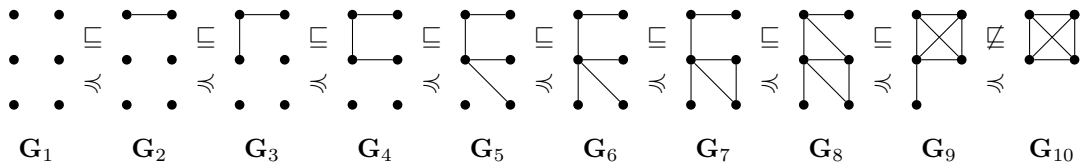
\begin{figure}[h]
  \centering
  {   \begin{tikzpicture}[auto, vertex/.style={circle,draw=black!100,fill=black!100, thick,
        inner sep=0pt,minimum size=1mm},scale=1.5]
      \foreach \x in {1,...,8} 
      { \node (A\x) at (\x,0) [vertex,label=below:] {};
        \node (B\x) at (\x,.4) [vertex,label=below:] {};
        \node (C\x) at (\x,.8) [vertex,label=below:] {};
        \node (D\x) at (\x+.4,0) [vertex,label=below:] {};
        \node (E\x) at (\x+.4,.4) [vertex,label=below:] {};
        \node (F\x) at (\x+.4,.8) [vertex,label=below:] {};	
        \node at (\x+.2,-.4) {$\W_{\x}$};   
      }
      \foreach \x in {9} 
      { \node (A\x) at (\x,0) [vertex,label=below:] {};
        \node (B\x) at (\x,.4) [vertex,label=below:] {};
        \node (C\x) at (\x,.8) [vertex,label=below:] {};
        \node (D\x) at (\x+.4,.8) [vertex,label=below:] {};
        \node (E\x) at (\x+.4,.4) [vertex,label=below:] {};
        \node (F\x) at (\x+.4,.8) [vertex,label=below:] {};	
        \node at (\x+.2,-.4) {$\W_{\x}$};   
      }
      \foreach \x in {10} 
      {
        \node (B\x) at (\x,.4) [vertex,label=below:] {};
        \node (C\x) at (\x,.8) [vertex,label=below:] {};
        \node (D\x) at (\x+.4,.8) [vertex,label=below:] {};
        \node (E\x) at (\x+.4,.4) [vertex,label=below:] {};
        \node (F\x) at (\x+.4,.8) [vertex,label=below:] {};	
        \node at (\x+.2,-.4) {$\W_{\x}$};   
      }
      
      \foreach \x in {1,...,8}  \node at (\x+.7,.6) {$\lesse$};	
      \foreach \x in {9}  \node at (\x+.7,.6) {$\not\lesse$};	
      \foreach \x in {1,...,9}  \node at (\x+.7,.2) {$\less$};	 
      \foreach \y in {2,...,10}  \path [-](C\y) edge node[right] {} (F\y);
      \foreach \y in {3,...,10} \path [-](B\y) edge node[right] {} (C\y);  
      \foreach \y in {4,...,10} \path [-](B\y) edge node[right] {} (E\y); 
      \foreach \y in {5,...,10} \path [-](B\y) edge node[right] {} (D\y); 
      \foreach \y in {7,...,10} \path [-](E\y) edge node[right] {} (D\y);
      \foreach \y in {6,...,9} \path [-](B\y) edge node[right] {} (A\y);
      \foreach \y in {8,...,10} \path [-](C\y) edge node[right] {} (E\y);
    \end{tikzpicture}}  
  \caption{Example of the preorders relations in $\Co^t$ for simple graphs up to order $n=6$. One has $\W_i\lesse\W_{i+1}$, $1\leq i\leq 9$ hence
  $\W_i\less\W_{i+1}$, $1\leq i\leq 8$. Moreover, $\W_9\less \W_{10}$ but  $\W_9\not\lesse \W_{10}$ showing that the geometric preorder is stronger 
  than the spectral preorder.}
  \label{fig:order-graph}
\end{figure}

A \emph{spanning} subgraph of a graph $\G$ is a subgraph $\G-E_0$
obtained from $\G$ by deleting all edges in $E_0$ where
$E_0\subset E(G)$.  Note that a spanning subgraph has the same set of
vertices of the original graph.  A \emph{spanning tree} of $\G$ is a
spanning subgraph which is a tree.  For example $\G_6$ is a spanning
tree for $\G_8$. We have the following simple consequence of
\Corenum{delete-edge}{delete-edge.b} and our definition of spectral
preorder $\less$ which generalises a result of
Fiedler~\cite[Corollary~3.2]{fiedler:73}, see
also~\cite[Section~3]{mohar:91}:
\begin{corollary}
  \label{cor:spanning.subgraph}
  Let $\W'=\W-E_0 \in \Co$ be a spanning subgraph of $\W \in \Co$,
  then $\lambda_k(\W') \le \lambda_k(\W)$ for all
  $k \in \{1,\dots,\card{V(\W)}\}$.
\end{corollary}

\subsubsection{Cliques and stability of eigenvalues} %
Let $\G=(V,E,\bd)$ be a graph and $d \in \N$.  A \emph{$d$-clique} of
$\G$ is an induced subgraph $\G[V_0]$ ($V_0\subset V$) isomorphic to
the complete graph $K_d$ of order $d=\card{V_0}$ (see
\Def{induced-subgraph}).  A $d$-clique is \emph{maximal} if it is not
a subgraph of a $(d+1)$-clique of $\G$.  The \emph{clique number} of
$\G$ is the maximal $d$ such that $\G$ has a $d$-clique.
The notion of a clique can be naturally extended to \GAM-graphs by
restricting the weights and vector potential to the corresponding
substructures.  For simplicity, we will denote the $d$-clique by
$K_d$. All graphs here are assumed to have the combinatorial
weight.

In the next theorem, we will apply the geometric and spectral preorder
relations given in \Defs{relation1}{spectral-order} to identify the
eigenvalue $d$ in the spectrum of the Laplacian of the graph with a
$d$-clique and to give a lower bound of its multiplicity.  Roughly
speaking, the \emph{clique number} $d$ of a graph can be seen in its
(combinatorial) spectrum for graphs with number of edges in a certain
range depending on $d$, see Eq.~\eqref{eq:edge.adm}.
\begin{theorem}
  \label{thm:cliques}
  Let $\W$ be a connected graph with combinatorial weights having $m$
  edges and a maximal $d$-clique.  Assume that $m < (d-1)(d+2)/2$ then
  $d$ is in the spectrum of $\W$ with multiplicity at least
  \begin{equation*}
    \frac{(d-1)(d+2)}2 - m=d-r-1,
  \end{equation*}
  where $r=m-d(d-1)/2$ is the number of edges of $G$ not in the clique.
\end{theorem}
\begin{remark}
  \label{rem:cliques}
  \indent
  \begin{enumerate}
  \item
    \label{cliques.a}
    \Thm{cliques} applies to graphs with a maximal $d$-clique and
    number of edges $m$ fulfilling
    \begin{equation}
      \label{eq:edge.adm}
      m \in \Bigl\{\frac {d(d-1)}2, \dots, 
      \frac{(d+2)(d-1)}2-1\Bigr\}.
    \end{equation}
    Let us call such numbers of edges \emph{$d$-admissible}.  For
    a given $d$, the above list of $d$-admissible numbers of edges has
    $d-1$ entries; for each one the multiplicity of the eigenvalue $d$
    is fixed, independently of the number $n$ of vertices of $\G$.
    Nevertheless since $\G$ has a $d$-clique, we have $n\ge d$ and
    since $\G$ is connected, we have $n \le d+r=m-d(d-3)/2$.

  \item
    \label{cliques.b}
    Note that we also allow multiple edges here.  For example if $G$
    is the complete graph with three vertices and one double edge
    (hence $m=4$), then its spectrum is $(0,3,5)$.  This graph has a
    $3$-clique by deleting one of the double edges, and $d=3$ has
    multiplicity (here exactly) $1$, as predicted by the theorem.
    Similarly, if $G$ is $K_4$ with one extra double edge, then the
    spectrum is $(0,4,4,6)$, and $d=4$ is a double eigenvalue, again
    as predicted.

  \item
    \label{cliques.c}
    In the proof of \Thm{cliques}, we do not need that the $d$-clique
    is \emph{maximal}, but considering $G$ as a graph with a
    $(d-1)$-clique then the range of $(d-1)$-admissible numbers of
    edges is disjoint from the range of $d$-admissible numbers of
    edges.  In particular, the theorem only makes sense for
    \emph{maximal} $d$-cliques.
  \end{enumerate}
\end{remark}

In~\cite{mohar:92}, Mohar excludes certain cycles as subgraphs by just
looking at the spectrum of the graph.  Here, we can exclude
$d$-cliques spectrally:
\begin{corollary}
  \label{cor:cliques}
  Assume that $\W$ is a connected graph.  If $d \in \N$ is not in the
  spectrum of $\W$, and if $\W$ has less than $(d-1)(d+2)/2$ edges,
  then $\W$ has no $d$-clique.
\end{corollary}
\begin{remark*}
  \indent
  \begin{enumerate}
  \item If the number of edges $m$ is below $d(d-1)/2$, then obviously
    $K_d$ cannot be a subgraph, the other values of $m$ are
    admissible.

  \item The converse of the above corollary is false (or the
    conclusion of \Thm{cliques} can be true also for graphs without a
    $d$-clique): the Petersen graph has (combinatorial Laplace)
    spectrum $(0,2_5,5_4)$ (the subscript indicating the
    multiplicity), hence $d=5$ is in its spectrum with multiplicity
    $4$ as said in \Thm{cliques}.  Also, the number of edges ($m=10$)
    is $5$-admissible, see Eq.~\eqref{eq:edge.adm}.  But the clique number
    of the Peterson graph is $2$ (and not $5$).
  \end{enumerate}
\end{remark*}
\begin{proof}[Proof of \Thm{cliques}]
  The strategy of the proof is to delete suitable edges on the
  complement of the maximal clique and control the spectral shifts
  $s,t$ so that we finally obtain relations
  $\W \less[s] \mathbf K_d \less[t] \W$.  Then we exploit the fact
  that for combinatorial weights the eigenvalue
  $d\in\sigma(\mathbf K_d)$ has high multiplicity, namely multiplicity
  $d-1$.
  
  Let $\W=(\G,0,\com)$ and set $n=\card{V(\G)}$.  The first Betti
  numbers of $\G$ and $K_d$ are $b_1(\G):=m-n+1$ and
  $b_1(K_d):= d(d-1)/2-d+1$ (see e.g.~\cite[Section~4]{sunada:12}).
  Denote its difference by
  \begin{equation*}
    p:=b_1(\G)-b_1(K_d)=m-n-\frac{d(d-3)}{2}.
  \end{equation*}
  Let $\wt E:=E(G)\setminus E(K_d)$.  We delete $p$ edges $\widehat E$
  from $\wt E$ in such a way that no cycles are present in the
  complement of the clique.  Construct first a subgraph
  $\G_1=\G-\widehat{E}=(V_1,E_1,\bd)$ with $V_1=V(G)$ and
  $E_1=E(K_d)\cup \left(\widetilde{E}\setminus\widehat{E}\right)$.
  Applying iteratively \Corenum{delete-edge}{cor.delete-edge.b} we obtain
  \begin{equation}
    \label{eq:graph1}
    \W \less[p] \W_1 \less \W.
  \end{equation}
  Note that the graph $\G_1-E(K_d)$ is a forest. Next, we delete all
  $n-d$ edges of this forest starting from the leaves and proceeding
  towards the $d$-clique $K_d$.  (Note that it is important to delete
  only leaves in this process, as otherwise the spectral shift is not
  optimal.)  Then, applying iteratively
  \Corenum{delete-pendant}{delete-pendant.b}, we obtain
  \begin{equation*}
    \W_1 \less \mathbf K_d \less[n-d] \W_1.
  \end{equation*}
  Using the transitivity of the relation $\less$ as well as
  Eq.~\eqref{eq:graph1} we get finally the relations
  \begin{equation}
    \label{eq:Kd}
    \W \less[p] \mathbf K_d \less[n-d] \W.
  \end{equation}
  Let $\lambda_1, \dots, \lambda_n$ resp.\ $\mu_1, \dots, \mu_d$ be
  the spectrum of (the combinatorial Laplacians of) $\W$ resp.\
  $\mathbf K_d$; as usual written in ascending order and repeated
  according to their multiplicities.  For the complete graph of order
  $d$ we have $\mu_1=0$ and $\mu_k=d$ ($k=2,\dots,d$).  The relations
  in Eq.~\eqref{eq:Kd} imply
  \begin{equation*}
    \lambda_k\leq \mu_{k+p}
    \quadtext{for} 1 \leq k\leq d-p
    \quadtext{and}
    \mu_k\leq \lambda_{k+(n-d)} 
    \quadtext{for} 1\leq k\leq d.
  \end{equation*}
  Note that the relations between the orders of the graphs and the
  spectral shift needed in \Def{spectral-order} are automatically
  satisfied in our case since $\card{E(K_d)}=d(d-1)/2 \leq
  m$. Combining the preceding inequalities, we obtain
  \begin{equation*}
    d
    = \mu_k\leq \lambda_{k+n-d}
    \leq \mu_{k+(n-d)+p}
    =d \quadtext{for} 2\leq k\leq 2d-n-p,
  \end{equation*}
  hence $d$ is an eigenvalue of $\W$ with multiplicity given by
  \begin{equation*}
    2d-n-p-1
    =\frac{d(d+1)}{2}-m-1
    =\frac{(d-1)(d+2)}{2}-m
  \end{equation*}
  which completes the proof.
\end{proof}

\begin{example}
  \label{ex:subgraph}
  We illustrate the preceding theorem with some examples having
  combinatorial weights, no magnetic potential and a maximal
  $6$-clique as shown in \Fig{high-mult}. Concretely, the graphs
  $\G_1$, $\G_2$ and $\G_3$ all have $m=17$ edges and a maximal
  $6$-clique $K_6$; moreover $\G_1$ and $\G_2$ have $8$ vertices,
  while $\G_3$ has $7$ vertices.  All three graphs have $d=6$ as
  eigenvalue in its spectrum with multiplicity at least $3$.  The
  range of admissible number of edges for $d=6$ here is
  $m \in \{15,16,17,18,19\}$.  The (minimal) multiplicity of the
  eigenvalue $d=6$ is then $5,4,3,2,1$.  In \Fig{high-mult}, we have
  $m=17$ and minimal multiplicity $3$.
  \begin{figure}[h]
    \centering
    \subfloat[{$\G_1$}\label{subfig:high-mult4}]	{  \begin{tikzpicture}[auto, vertex/.style={circle,draw=black!100,fill=black!100, thick,
          inner sep=0pt,minimum size=1mm}]
        \color{black}
        \node (A1) at (1/2,-2/2) [vertex,color=black,label=below:] {};
        \node (A2) at (-1/2,-2/2) [vertex,color=black,label=below:] {};
        \node (A3) at (-2/2,0) [vertex,color=black,label=left:] {};
        \node (A4) at (-1/2,2/2) [vertex,color=black,label=above:] {};
        \node (A5) at (1/2,2/2) [vertex,color=black,label=above:] {};
        \node (A6) at (2/2,0) [vertex,color=black,label=right:] {};
        \foreach \y in {1,...,6} { 
          \foreach \x in {1,...,6} {
            \path [-](A\x) edge node[right] {} (A\y);}
        }
        \node (A7) at (-1,1.25) [vertex,color=black,label=right:] {};
       	\node (A8) at (-1.25,.75) [vertex,color=black,label=right:] {};	
        \path [dashed,color=black,](A4) edge node[right] {} (A7);
        \path [dashed,color=black,](A7) edge node[right] {} (A8);
      \end{tikzpicture}} 
    \subfloat[{$\G_2$} \label{subfig:high-mult5}]	{  \begin{tikzpicture}[auto, vertex/.style={circle,draw=black!100,fill=black!100, thick,
          inner sep=0pt,minimum size=1mm}]
        \color{black}
        \node (A1) at (1/2,-2/2) [vertex,color=black,label=below:] {};
        \node (A2) at (-1/2,-2/2) [vertex,color=black,label=below:] {};
        \node (A3) at (-2/2,0) [vertex,color=black,label=left:] {};
        \node (A4) at (-1/2,2/2) [vertex,color=black,label=above:] {};
        \node (A5) at (1/2,2/2) [vertex,color=black,label=above:] {};
        \node (A6) at (2/2,0) [vertex,color=black,label=right:] {};
        \foreach \y in {1,...,6} { 
          \foreach \x in {1,...,6} {
            \path [-](A\x) edge node[right] {} (A\y);}
        }
        \node (A7) at (-1,1.25) [vertex,color=black,label=right:] {};
        \node (A8) at (-1.25,.75) [vertex,color=black,label=right:] {};	
        \path [dashed,color=black,](A4) edge node[right] {} (A7);
        \path [dashed,color=black,](A3) edge node[right] {} (A8);
      \end{tikzpicture}}
    \subfloat[{$\G_3$} \label{subfig:high-mult6}]	{  \begin{tikzpicture}[auto, vertex/.style={circle,draw=black!100,fill=black!100, thick,
          inner sep=0pt,minimum size=1mm}]
	\color{black}
	\node (A1) at (1/2,-2/2) [vertex,color=black,label=below:] {};
	\node (A2) at (-1/2,-2/2) [vertex,color=black,label=below:] {};
	\node (A3) at (-2/2,0) [vertex,color=black,label=left:] {};
	\node (A4) at (-1/2,2/2) [vertex,color=black,label=above:] {};
	\node (A5) at (1/2,2/2) [vertex,color=black,label=above:] {};
	\node (A6) at (2/2,0) [vertex,color=black,label=right:] {};
	\foreach \y in {1,...,6} { 
          \foreach \x in {1,...,6} {
            \path [-](A\x) edge node[right] {} (A\y);}
	}
	\node (A7) at (-1.25,.95) [vertex,color=black,label=right:] {};
	\path [dashed,color=black,](A4) edge node[right] {} (A7);
	\path [dashed,color=black,](A3) edge node[right] {} (A7);
      \end{tikzpicture}\qquad}	
    \caption{An illustration of the Theorem~\ref{thm:cliques} for
      $6$-clique.  Let $\W_1$, $\W_2$ and $\W_3$ be the \GAM-graphs
      (with combinatorial weight and no magnetic potential) defined by
      the graphs $G_1$, $G_2$ and $G_3$ respectively.  All graphs have
      $d=6$ in its spectrum with multiplicity at least $3$. }
    \label{fig:high-mult}
  \end{figure}
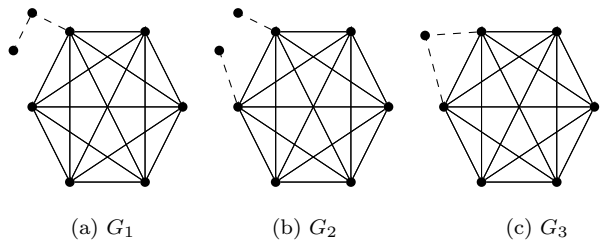
\end{example}

\subsubsection{Minors} %
\label{subsec:minors}
A fundamental notion in combinatorics is that of a graph minor.
Several fundamental results in this field are presented in terms of
minors (e.g.\ in the Robertson-Seymour theory
\cite[Chapter~12]{diestel:00}).  A graph $H$ is called a \emph{minor}
of a given graph $G$ if $H$ is obtained from $G$ by applying certain
elementary operations. We can generalise this construction to
MW-graphs and apply the results of the previous sections to give a
spectral relation between a graph and its minor. We consider the
following three elementary operations:
\begin{itemize}
\item Deleting an edge (\Subsec{delete-edge}),
\item Contracting an edge (\Subsec{edge-contra}),
\item Deleting pendant vertex (the same as contracting a pendant edge,
  \Subsec{delete-pendant}).
\end{itemize}
If $\W'\in\Ga$ is obtained from $\W$ by successive application of the
previous operations, then we say that $\W'$ is a \emph{minor} of $\W$
(see, e.g.~\cite{bondy-murty:08}).

\begin{proposition}
 \label{prp:minor-deg}
 Let $\W\in\Ga$ be a simple graph without magnetic potential and let
 $\W'$ be a minor of $\W$ obtained by deleting $p$ edges, contracting
 $q$ edges and deleting $s$ pendant vertices.
 \begin{enumerate}
 \item
   \label{minor-deg.b}
   If $\W,\W'\in\Co$, then $\W \less[p] \W' \less[r+s] \W$ where
   $r=\sum_{e\in E_0} \min\{\deg (\bd_+ e_0),\deg (\bd_- e_0)\}(\ge
   q)$ and $E_0$ is the set of $q$ edges, that are contracted.
 \item
   \label{minor-deg.a}
   If $\W,\W'\in\De$, then
   $\W \less[p+q-q'] \W' \less[p+q+s] \W$ where $q'$ is the number of
   \emph{bridge} edges that are contracted.
 \end{enumerate}
\end{proposition}
A simple consequence of the spectral preorder is the following:
\begin{corollary}
  \label{cor:minor-deg}
  Let $\W \in \Co$, If $\W' \in \Co$ is a minor of $\W$ obtained as in
  \Prp{minor-deg}, then an eigenvalue of $\W$ of multiplicity
  $m \ge p+r+s+1$ remains and eigenvalue of $\W'$ of multiplicity
  $m-p-r-s$.

  Similarly, for $\W, \W' \in \De$, an eigenvalue of $\W$ of
  multiplicity $m \ge 2p+2q+s-q'+1$ is an eigenvalue of $\W'$ of
  multiplicity $m-2p-2q-s+q'$.
\end{corollary}


\subsection{Cheeger constants and frustration index under
  \GAM-homomorphisms}
\label{subsec:cheeger}

The Cheeger constant is a quantitative measure of the connectedness of
a graph.  It can be used in a lower bound on the second (first
non-zero) eigenvalue of a graph, usually called \emph{Cheeger
  inequality} (and also in an upper bound).  Probably the first
occurrence of a lower bound on the second eigenvalue in terms of
geometric quantities is in~\cite[Section~4.3]{fiedler:73};
Dodziuk~\cite[Theorem~2.3]{dodziuk:84} proves a Cheeger inequality for
the combinatorial Laplacian, see also Mohar~\cite[Section~6]{mohar:91}
for further improvements and references, as well as Colin de
Verdi\`ere~\cite[Theorem~3.1]{colin:98}.  A Cheeger inequality for the
standard Laplacian can be found in Chung's
book~\cite[Chapter~2]{chung:97}. An extension of the Cheeger constant to
magnetic potentials need the so-called \emph{frustration index}.  For
a more detailed overview on the literature concerning Cheeger
constants and the frustration index on graphs and manifolds we refer
to~\cite{llpp:15} and references therein.

We now show that the concept of \GAM-homomorphisms also gives simple
inequalities for Cheeger constants.  Here, we always denote the
underlying graphs of $\W$ and $\W'$ by $\G=(V,E,\bd)$ and
$\G'=(V',E',\bd')$.

We first define an ingredient necessary in the presence of a magnetic
potential:
\begin{definition}[Frustration index] 
  \label{def:frustration.index}
  Let $\W=\Wfull$ be an MW-graph and consider a function
  $\map \tau V R$, where $R$ is a subgroup of $\R/2\pi\Z$. We set
  \begin{equation*}
    \iota(\W,\tau)
    := \norm[{\lp[1]{E,\m}}]{d_\alpha(\e^{\im \tau})}
    = \sum_{e \in E}\m_e \abs{\e^{\im \tau(\bd_+e)}-\e^{-\im \alpha_e}\e^{\im \tau(\bd_-e)}}.
  \end{equation*}
  The \emph{frustration index} of $\W$ is defined as
  \begin{equation}
    \label{eq:frustration.index}
    \iota(\W) := \inf_{\tau \in R^V} \iota(\W,\tau),
  \end{equation}
  where $R^V$ denotes the set of all maps $\map \tau V R$.
\end{definition}
Note that the infimum is actually a minimum.  It is not hard to see
that $\iota(\W)=0$ if and only if $\alpha \sim 0$, i.e.\ if the
magnetic potential is cohomologous to $0$.
\AGAM-homomorphism gives a natural inequality for the frustration indices:
\begin{lemma}
  \label{lem:gam-frust}
  Let $\map \pi \W {\W'}$ be \aGAM-homomorphism and let
  $\map {\tau'}{V'}R$ a map, then
  \begin{equation*}
    \iota(\W,\tau'\circ \pi) \le \iota(\W',\tau')
    \qquadtext{and}
    \iota(\W) \le \iota(\W').
  \end{equation*}
\end{lemma}
\begin{proof}
  Note first that
  $d_\alpha(\e^{\im \tau' \circ \pi})=(d_{\alpha'}(\e^{\im
    \tau'}))\circ \pi$
  as $\pi$ is a graph homomorphism and $\alpha' \circ \pi=\alpha$.
  Moreover, we have
  \begin{align*}
    \iota(\W,\tau'\circ \pi)
    &= \sum_{e \in E} w_e \bigabs{(d_\alpha(\e^{\im \tau'\circ \pi}))_e}
    = \sum_{e' \in E'} \sum_{e \in E, \pi(e)=e'} w_e 
         \bigabs{(d_{\alpha'}(\e^{\im \tau'}))_{e'}}
    = \sum_{e' \in E'} (\pi_*\m)_{e'} \bigabs{(d_{\alpha'}(\e^{\im \tau'}))_{e'}}\\
    &\le \sum_{e' \in E'} \m'_{e'} \bigabs{d_{\alpha'}(\e^{\im \tau'})}
      = \iota(\W',\tau'),
  \end{align*}
  as $\pi$ is \aGAM-homomorphism (and in particular,
  $(\pi_*w)_{e'} \le \m'_{e'}$ for all $e' \in E'$).  For the last
  inequality in the lemma, note that the set $R^V$ of maps $\map \tau V R$ is
  larger than the subset
  $\set{\tau' \circ \pi}{\tau' \in R^{V'}} \subset R^V$, hence we have
  \begin{equation*}
    \iota(\W)
    \le \inf_{\tau' \in R^{V'}} \iota(\W,\tau'\circ \pi)
    \le \inf_{\tau' \in R^{V'}} \iota(\W',\tau')
    = \iota(\W').
    \qedhere
  \end{equation*}
\end{proof}

We denote by $\W[V_0]$ the induced subgraph $\G[V_0]$ with vertex set
$V_0 \subset V$ and edge set $E(V_0)$ (see \Def{induced-subgraph})
together with the natural restrictions of $\m$ and $\alpha$ to $V_0$
respectively $E(V_0)$.  A \emph{$k$-subpartition} of $V$ is given by
$k$ pairwise disjoint non-empty subsets $V_1,\dots,V_k$ of $V$; the set
of all $k$-subpartitions $\Pi=\{V_1,\dots,V_k\}$ of $V$ is denoted by
$\Pi_k(V)$.
\begin{definition}
  \label{def:cheeger}
  Let $\W$ be an MW-graph.  For a subset $V_0 \subset V$ we set
  \begin{equation*}
    h(\W,V_0)
    := \frac{\iota(\W[V_0])+w(E(V_0, \compl{V_0}))}{w(V_0)}.
  \end{equation*}
  The \emph{$k$-th} (also called \emph{$k$-way)} \emph{(magnetic
    weighted) Cheeger constant} $h_k(\W)$ is defined as
  \begin{equation}
    \label{eq:cheeger-n}
    h_k(\W)
    :=\inf_{\Pi \in \Pi_k(V)}  \sup_{ V_0 \in \Pi} h(\W,V_0).
  \end{equation}
\end{definition}
Note that the infimum and supremum are actually minimum and maximum.
It is not hard to see that $h_k(\W) \le h_{k+1}(\W)$.  Moreover, for
$k=1$ resp.\ $k=2$ we have
\begin{equation*}
  h_1(\W)
  =\min_{V_0 \subset V, V_0 \ne \emptyset} h(\W,V_0)
  \quadtext{resp.}
  h_2(\W)
  =\min_{V_0 \subset V, V_0 \ne \emptyset, V_0 \ne V} 
  \max \{h(\W,V_0), h(\W, \compl V_0)\}
\end{equation*}
for the first and second Cheeger constant (the latter is usually
called \emph{the} Cheeger constant).  If $\alpha \sim 0$ then the
second (usual) Cheeger constant equals
\begin{equation*}
  h_2(\W)
  =\min_{V_0 \subset V, V_0 \ne \emptyset, V_0 \ne V} 
  \frac{\m(E(V_0, \compl V_0))}{\min \{\m(V_0),\m(\compl V_0)\}}.
\end{equation*}

\begin{lemma}
  \label{lem:gam-cheeger}
  Let $\map \pi \W {\W'}$ be \aGAM-homomorphism, then
  \begin{equation*}
    h(\W,\pi^{-1}(V_0'))
    \le h(\W',V_0')
  \end{equation*}
  for all $V_0' \subset V'$, $V_0' \ne \emptyset$.
\end{lemma}
\begin{proof}
  If $\map \pi \W {\W'}$ is \aGAM-homomorphism, then the restriction
  $\map \pi {\W[\pi^{-1}(V_0')]}{\W'[V_0']}$ is defined as map (as
  $\pi(\pi^{-1}(V_0'))\subset V_0'$ and
  $\pi(E^\G(\pi^{-1}(V_0')))\subset E^{\G'}(V_0')$ by
  \Lem{graph-homo-edgeset}), and again \aGAM-homomorphism.  In
  particular, from \Lem{gam-frust} we conclude that
  \begin{equation*}
    \iota(\W[\pi^{-1}(V_0')]) \le \iota(\W'[V_0']).
  \end{equation*}
  Next, we have
  $E^\G(\pi^{-1}(V_0'),\compl{(\pi^{-1}(V_0'))})
  =\pi^{-1}(E^{\G'}(V_0',\compl{(V_0')}))$
  again by \Lem{graph-homo-edgeset} and as
  $\compl{(\pi^{-1}(V_0'))}=\pi^{-1}(\compl{(V_0')})$.  In particular,
  we conclude
  \begin{equation*}
    \m\bigl(E^\G(\pi^{-1}(V_0'),\compl{(\pi^{-1}(V_0'))})\bigr)
    =(\pi_*\m)\bigl(E^{\G'}(V_0',\compl{(V_0')})\bigr)
    \le \m'\bigl(E^{\G'}(V_0',\compl{(V_0')})\bigr)
  \end{equation*}
  as $\pi$ is \aGAM-homomorphism.  Similarly, we have
  \begin{equation*}
    \m\bigl(\pi^{-1}(V_0')\bigr)
    = (\pi_* \m)(V_0')
    \ge \m'(V_0'),
  \end{equation*}
  and the desired inequality follows.
\end{proof}

We are now able to prove the main result of this section:
\begin{theorem}
  \label{thm:cheeger-and-morphisms}
  Let $\W$ and $\W'$ be two \GAM-graphs such that $\W \lesse \W'$,
  i.e.\ there is \aGAM-homomorphism $\map \pi {\W} {\W'}$ (see
  \Def{hom.gam-graph}), then we have
  \begin{equation*}
    h_k(\W) \le h_k(\W')
  \end{equation*}
  for all $k \in \N$.
\end{theorem}
\begin{proof}
  Let the minimum in $h_k(\W')$ be achieved at
  $\Pi'=\{V_1',\dots,V_k'\} \in \Pi_k(V')$, and the maximum at $V_1'$,
  i.e.\ assume that $h_k(\W')=h(\W',V_1')$.  As $\pi$ is surjective on
  the vertices (see \Prpenum{inj-sur}{mow-hom.a}),
  $\Pi:=\{\pi^{-1}(V_1'),\dots, \pi^{-1}(V_k ')\}$ is again a
  $k$-subpartition (all sets are pairwise disjoint \emph{and}
  non-empty by the surjectivity).  Now we have
  \begin{equation*}
    h_k(\W)
    \le \sup_{j=1,\dots,k} h(\W,\pi^{-1}(V_j'))
    \le \sup_{j=1,\dots,k} h(\W',(V_j'))
    = h(\W',V_1')
    = h_k(\W')
  \end{equation*}
  as $h_k(\G)$ is the infimum over all $k$-subpartitions, and $\Pi$ is
  such a $k$-partition of $V$ (first inequality).  The second
  inequality follows from \Lem{gam-cheeger}, and the last equality
  from the choice of the partition $\Pi'$ and $V_1'$.
\end{proof}

\begin{remark}
  \label{rem:ev-est-and-cheeger}
  Note that we have proven in \Thm{homomorphism} the inequality
  $\lambda_k(\W) \le \lambda_k(\W')$ if there is \aGAM-homomorphism
  $\map \pi \W {\W'}$.  We have just proven in
  \Thm{cheeger-and-morphisms} that \aGAM-homomorphism also increases
  the $k$-th Cheeger constant, hence \Thm{cheeger-and-morphisms} is in
  accordance with the (magnetic weighted) Cheeger inequalities
  \begin{equation*}
    \frac 12 \lambda_k(\W) 
    \le h_k(\W)
    \le Ck^3 \sqrt{\rho_\infty \lambda_k(\W)}
  \end{equation*}
  for all $k \in \{1,\dots,\card{\G}\}$ proven in~\cite{llpp:15},
  where $C>0$ is a universal constant (recall that $\rho_\infty$ is
  the supremum of the relative weight, see~\eqref{eq:rho.bdd}).  If
  $k=1$, then $C=1$, and if $k=2$ and if $\alpha \sim 0$, then one can
  choose $C=\sqrt 2/4$.
\end{remark}

\begin{example}
  \label{ex:cheeger}
  Let $\W=\Wfull$ and $\W'=(\G',\m',\alpha')$ be two \GAM-graphs.
  \begin{itemize}
  \item \myparagraph{Combinatorial weight and removing edges:} If
    $E_0 \subset E(\G)$ and $\W$, $\W'=\W-E_0 \in \Co$, then
    $h_k(\W-E_0)\le h(\W)$.  Heuristically, this means that removing
    edges decreases the connectivity.
  \item \myparagraph{Standard and combinatorial weight compared:} If
    $\G=\G'$, $\m=\deg$ and $\m'=\com$, then $h_k(\W)\leq h_k(\W')$
    (the combinatorial weight has higher Cheeger constants).
  \item \myparagraph{Standard weight and contracting vertices:} If
    $\sim$ is an equivalence relation on $V(\G)$, and if $\W$,
    $\W'=\W/{\sim} \in \De$, then $h_k(\W) \le h_k(\W/{\sim})$.
    Heuristically, this means that contracting vertices increases the
    connectivity.
  \end{itemize}

\end{example}

\subsection{Covering graphs and spectral gaps}
\label{subsec:covering} 
In~\cite{flp:18} and~\cite{lledo-post:08} we study the spectrum of the
discrete Laplacian of infinite (regular) covering graphs
${\wt \W}\rightarrow{\W}$ with finite quotient. For the case with a
periodic vector potential see~\cite{fabila-lledo:pre19}. In Section~5
of~\cite{flp:18} we consider only Abelian covering and interpret the
magnetic potential as a Floquet parameter to decompose the Laplacian
$\Delta^{\wt \W}$ as a direct integral of discrete magnetic weighted
Laplacians $\Delta_{\alpha}^{\W}$ on the finite quotient. We developed
a technique of virtualising specific edges and vertices on the
quotient $\W$ to obtain two new graphs $\W^-,\W^+$ satisfying the
relation
\begin{equation}
  \W^-\less\W\less \W^+.
\end{equation}
Let $\W=\Wfull$ be an \GAM-graph, virtualising a set of edges
$E_0\subset E$ means in this context to consider a new graph with
vertices $V=V^+$ and edges $E^-=E\setminus E_0$ but, contrary to the
case described in \Subsec{delete-edge}, we keep the value of weight on
the remaining edges, i.e.\ we define $\m^-:=\m\restr{E\setminus E_0}$
and $\alpha_e^-=\alpha_e$ for all $e\in E^-$. In particular, this
operation may change the type of weights used. For example, if $\m$ is
a standard weight on $\G$, then $\m' \restr{V'}$ need not to be
standard for the new graph (cf.\ Definition~3.4 in~\cite{flp:18}).
The virtualisation of vertices $V_0\subset V$ gives a new partial
\GAM-graph $\W^+=(G^+,\m^+,\alpha^+)$ defined as
$V^+=\G\setminus V_0$, $E^+=E\setminus E(V_0)$, $\m_v^+=\m_v$ for all
$v\in V^+$, $\m_e^+=\m_e$ and $\alpha_{e}^+=\alpha_{e}$ for all
$e\in E^+$ (see Definition~3.9 in~\cite{flp:18} for details and
additional motivation).
 
Denote the spectrum of $\W^{\pm}$ by $\{\lambda_k(\W^\pm)\}_k$ (written
as usual in ascending order and counting multiplicities).  Our
techniques allow to localise the spectrum of the covering graph by
 \begin{equation}
 \label{eq:loc.spect}
 \sigma(\Delta^{\wt \W}) 
 \subset  
 \bigcup_{k=1}^{\card{\G}} 
 \bigl[\lambda_{k}(\W^-),\lambda_{k}(\W^+)\bigr]
 \end{equation}
 where we denote by
 $J_k:= \bigl[\lambda_{k}(\W^-),\lambda_{k}(\W^+)\bigr]$ the
 bracketing intervals.  The elementary operations described in
 \Sec{geo} can now be applied to refine the spectral localisation
 given in \Eq{eq:loc.spect} and, in particular, one can discover new
 spectral gaps that can not be found with the method described
 in~\cite{flp:18,fabila-lledo:pre19}.
 
 We illustrate this in one specific example of covering graph but,
 there are many ways how to refine the spectral localisation using the
 methods described in \Sec{geo}. Consider the infinite $\Z$-covering
 graph $\wt \W$ given in Figure~\ref{subfig:brack3} with standard
 weights.
 
 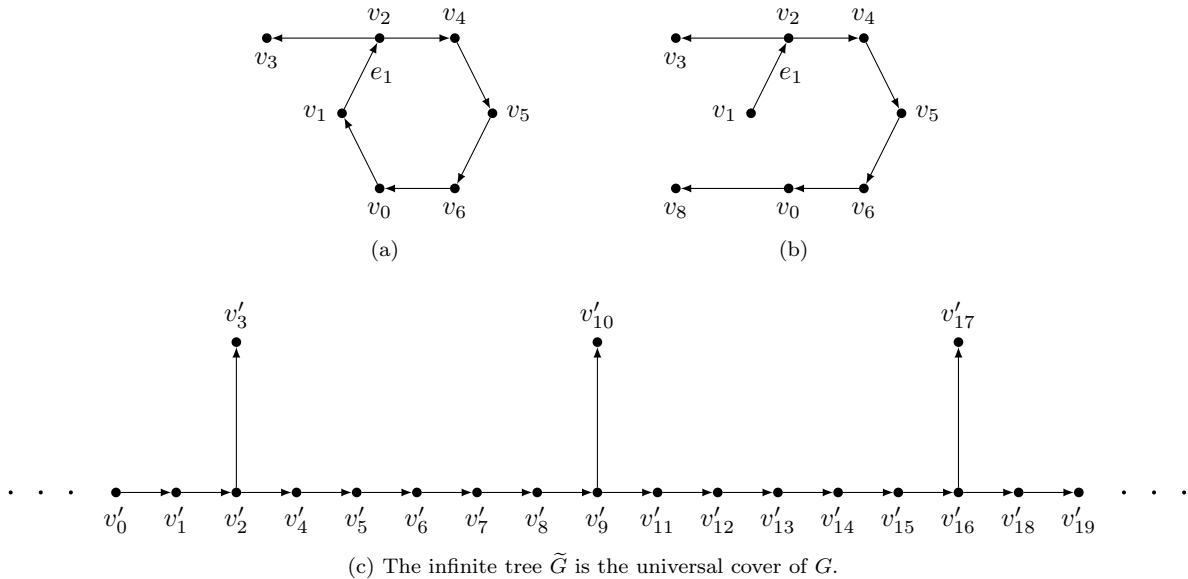
\begin{figure}[h]
 	\centering
 	\subfloat[{}\label{subfig:brack1}]
 	{  \qquad 
 		\begin{tikzpicture}[auto, 
 		vertex/.style={circle,draw=black!100,fill=black!100, thick,
 			inner sep=0pt,minimum size=1mm}]
 		\node (A) at (1/2,-2/2) [vertex,label=below:$v_6$] {};
 		\node (B) at (-1/2,-2/2) [vertex,label=below:$v_0$] {};
 		\node (C) at (-2/2,0) [vertex,label=left:$v_1$] {};
 		\node (D) at (-1/2,2/2) [vertex,label=above:$v_2$] {};
 		\node (E) at (1/2,2/2) [vertex,label=above:$v_4$] {};
 		\node (F) at (2/2,0) [vertex,label=right:$v_5$] {};
 		\node (G) at (-2,1) [vertex,label=below:$v_3$] {};
 		
 		\path [-latex](C) edge node[right] {$e_1$} (D);
 		\path [-latex](D) edge node[below] {} (E);
 		\path [-latex](E) edge node[left] {} (F);
 		\path [-latex](F) edge node[left] {} (A);
 		\path [-latex](A) edge node[above] {} (B);
 		\path [-latex](D) edge node[above] {} (G);
 		\path [-latex](B) edge node[above] {} (C);
 		\end{tikzpicture}\qquad}  
 	\color{black}\subfloat[{}\label{subfig:brack2}]
 	{  \qquad \begin{tikzpicture}[auto, 
 		vertex/.style={circle,draw=black!100,fill=black!100, thick,
 			inner sep=0pt,minimum size=1mm}]
 		\node (A) at (1/2,-2/2) [vertex,label=below:$v_6$] {};
 		\node (B) at (-1/2,-2/2) [vertex,label=below:$v_0$] {};
 		\node (C) at (-2/2,0) [vertex,label=left:$v_1$] {};
 		\node (D) at (-1/2,2/2) [vertex,label=above:$v_2$] {};
 		\node (E) at (1/2,2/2) [vertex,label=above:$v_4$] {};
 		\node (F) at (2/2,0) [vertex,label=right:$v_5$] {};
 		\node (G) at (-2,1) [vertex,label=below:$v_3$] {};
 		\node (H) at (-2,-1) [vertex,label=below:$v_8$] {};
 		
 		\path [-latex](C) edge node[right] {$e_1$} (D);
 		\path [-latex](D) edge node[below] {} (E);
 		\path [-latex](E) edge node[left] {} (F);
 		\path [-latex](F) edge node[left] {} (A);
 		\path [-latex](A) edge node[above] {} (B);
 		\path [-latex](D) edge node[above] {} (G);
 		\path [-latex](B) edge node[above] {} (H);
 		\end{tikzpicture}\qquad}
 	
 	\subfloat[{The infinite tree $\widetilde{\G}$ is the universal cover of $\G$.}\label{subfig:brack3}]{  \qquad  \begin{tikzpicture}[auto,
 		vertex/.style={circle,draw=black!100,fill=black!100, thick,
 			inner sep=0pt,minimum size=1mm},scale=4]
 		\node (C) at (-.95,0) [vertex,inner sep=.25pt,minimum size=.25pt,label=above:]{};
 		\node (C) at (-.85,0) [vertex,inner sep=.25pt,minimum size=.25pt,label=above:]{};
 		\node (C) at (-.75,0) [vertex,inner sep=.25pt,minimum size=.25pt,label=above:]{};                              
 		\node (X1) at (-.6,0) [vertex,,label=below:$v'_0$]{};
 		\node (X2) at (-.4,0) [vertex,label=below:$v'_1$]{};
 		\node (X3) at (-.2,0) [vertex,,label=below:$v'_2$]{};
 		\draw[-latex] (X1) to[] node[above] {} (X2);  
 		\draw[-latex] (X2) to[] node[above] {} (X3);   
 		
 		\node (X4) at (0,0) [vertex,label=below:$v'_4$]{};
 		\node (X5) at (.2,0) [vertex,label=below:$v'_5$]{};
 		\node (X6) at (.4,0) [vertex,label=below:$v'_6$]{};   
 		\node (X7) at (.6,0) [vertex,label=below:$v'_7$]{};
 		\node (X8) at (.8,0) [vertex,label=below:$v'_8$]{};
 		\node (X9) at (1,0) [vertex,label=below:$v'_9$]{};
 		\node (X10) at (1.2,0) [vertex,label=below:$v'_{11}$]{};
 		\node (X11) at (1.4,0) [vertex,label=below:$v'_{12}$]{};
 		\node (X12) at (1.6,0) [vertex,label=below:$v'_{13}$]{};   
 		\node (X13) at (1.8,0) [vertex,label=below:$v'_{14}$]{};
 		\node (X14) at (2,0) [vertex,label=below:$v'_{15}$]{};
 		\node (X15) at (2.2,0) [vertex,label=below:$v'_{16}$]{}; 
 		\node (X16) at (2.4,0) [vertex,label=below:$v'_{18}$]{};
 		\node (X17) at (2.6,0) [vertex,label=below:$v'_{19}$]{};                                 	                
 		\draw[-latex] (X3) to[] node[above] {} (X4);  
 		\draw[-latex] (X4) to[] node[above] {} (X5); 
 		\draw[-latex] (X5) to[] node[above] {} (X6); 
 		\draw[-latex] (X6) to[] node[above] {} (X7);  
 		\draw[-latex] (X7) to[] node[above] {} (X8); 
 		\draw[-latex] (X8) to[] node[above] {} (X9);
 		\draw[-latex] (X9) to[] node[above] {} (X10);  
 		\draw[-latex] (X10) to[] node[above] {} (X11); 
 		\draw[-latex] (X11) to[] node[above] {} (X12);
 		\draw[-latex] (X12) to[] node[above] {} (X13); 
 		\draw[-latex] (X13) to[] node[above] {} (X14);
 		\draw[-latex] (X14) to[] node[above] {} (X15); 
 		\draw[-latex] (X16) to[] node[above] {} (X17);
 		\draw[-latex] (X15) to[] node[above] {} (X16);                             
 		\node (Y1) at (-.2,.5) [vertex,label=above:$v'_3$]{};    
 		\draw[-latex] (X3) to[] node[above] {} (Y1); 
 		\node (Y2) at (1,.5) [vertex,label=above:$v'_{10}$]{};    
 		\draw[-latex] (X9) to[] node[above] {} (Y2);
 		\node (Y3) at (2.2,.5) [vertex,label=above:$v'_{17}$]{};    
 		\draw[-latex] (X15) to[] node[above] {} (Y3);  
 		
 		\node (C) at (2.75,0) [vertex,inner sep=.25pt,minimum size=.25pt,label=above:]{};
 		\node (C) at (2.85,0) [vertex,inner sep=.25pt,minimum size=.25pt,label=above:]{};
 		\node (C) at (2.95,0) [vertex,inner sep=.25pt,minimum size=.25pt,label=above:]{};                               
 		\end{tikzpicture}\qquad}                
 	
 	\caption{\itemref{subfig:brack1} The graph
 		$G$. \itemref{subfig:brack2} The graph
 		$G'\simeq G/\{v_1,v_8\}$.  \itemref{subfig:brack3} The
 		infinite graph $\tilde{G}$, is an infinite tree and maximal
 		Abelian covering of $\G$.}
 	\label{fig:brack}
\end{figure}

It is shown in Example~6.1 of~\cite{flp:18} that
\begin{equation*}
  \sigma(\Delta^{\wt \W}) 
  \subset  \bigcup_{k=1}^{7} 
  J_i\quad ,
\end{equation*}  
where
\begin{align*}
  J_1\approx 
  \left[0,0.121\right],
  \quad 
  J_2\approx 
  \left[0.116,0.358\right],
  \quad
  J_3\approx 
  \left[0.5,0.744\right],
  \quad
  J_4\approx  
  \left[0.713,1.256\right],\\
  J_5\approx 
  \left[1.145,1.642\right],
  \quad
  J_6\approx 
  \left[1.638,1.879\right],
  \quadtext{and}   
  J_7\approx  
  \left[1.889,2\right].                     
\end{align*} 
using also the symmetry of the spectrum from bipartiteness.  Note that
$J_2\cap J_3=J_6\cap J_7=\emptyset$, so that we have localised two
spectral gaps in $\sigma(\Delta^{\wt \W}) $.  We can now refine the
localisation of the spectrum as follows. Consider the graphs $\W$ and
$\W'$ as in Figures~\ref{subfig:brack1} and~\ref{subfig:brack2} where
$\W'=\W/\{v_1, v_8\}$.  Applying \Thm{vert-contra} we obtain the
following spectral relations
\begin{equation*}
  \W'_0=\W'_t\less\W_t\less[1]\W'_t=\W'_0.
\end{equation*}
where for $\W'_t$ we can take $t=0$ since $\G'$ is a tree. Therefore
we have the alternative localisation:
\begin{equation*}
  \sigma(\Delta^{\wt \W}) 
  \subset  \bigcup_{k=1}^{7} 
  J'_i\quad ,
\end{equation*}  
where
\begin{align*}
  J'_1\approx 
  \left[0,0.108\right],
  \quad 
  J'_2\approx 
  \left[0.108,0.463\right],
  \quad
  J'_3\approx 
  \left[0.463,1\right],
  \quad
  J'_4=\{1\},\\
  J'_5\approx 
  \left[1,1.536\right],
  \quad
  J'_6\approx 
  \left[1.536,1.891\right],
  \quadtext{and}
  J'_7\approx  
  \left[1.891,2\right].                     
\end{align*} 
Intersection both localisations $J$ and $J'$ we obtain a finer
bracketing:
\begin{equation*}
  \sigma(\Delta^{\wt \W}) 
  \subset  \bigcup_{k=1}^{7} 
  J\cap J'_i=\bigcup_{k=1}^{7} J'',
\end{equation*}  
where
\begin{align*}
  J''_1\approx 
  \left[0,0.108\right],
  \quad 
  J''_2\approx 
  \left[0.116,0.358\right],
  \quad
  J''_3\approx 
  \left[0.5,0.744\right],
  \quad
  J''_4=\{1\},\\
  J''_5\approx 
  \left[1.145,1.536\right],
  \quad
  J''_6\approx 
  \left[1.638,1.879\right],
  \quadtext{and}
  J''_7\approx  
  \left[1.891,2\right].                     
\end{align*} 
Note that $J''_k$ are better than $J_k$ for $k\in \{1,4,5,7\}$.  Using
this refinement obtained by applying a vertex splitting we are able to
determine that $1\in\sigma({\wt\W})$ with spectral gaps around it. We
also found new spectral gaps between \emph{all} bands, while
in~\cite{flp:18}, we only found two bands.


\newcommand{\etalchar}[1]{$^{#1}$}
\providecommand{\bysame}{\leavevmode\hbox to3em{\hrulefill}\thinspace}
\providecommand{\MR}{\relax\ifhmode\unskip\space\fi MR }
\providecommand{\MRhref}[2]{%
  \href{http://www.ams.org/mathscinet-getitem?mr=#1}{#2}
}
\providecommand{\href}[2]{#2}

\end{document}